\documentclass[11pt, oneside]{article}   	% use "amsart" instead of "article" for AMSLaTeX format
\usepackage{geometry}                		% See geometry.pdf to learn the layout options. There are lots.
\usepackage{graphicx}				% Use pdf, png, jpg, or eps§ with pdflatex; use eps in DVI mode
				
	\usepackage{subfigure}							 
	\usepackage{multirow}
	\usepackage{amsmath,amssymb,mathtools}
\usepackage{amsthm}
\newcommand{\Ac}{\mathcal{A}}

\newcommand{\Hc}{\mathcal{H}}
\newcommand{\Ic}{\mathcal{I}}

\newcommand{\Lc}{\mathcal{L}}
\newcommand{\Mc}{\mathcal{M}}
\newcommand{\Nc}{\mathcal{N}}
\newcommand{\Pc}{\mathcal{P}}

\newcommand{\Tc}{\mathcal{T}}
\newcommand{\Uc}{\mathcal{U}}

\newcommand{\Wc}{\mathcal{W}}
\newcommand{\Xc}{\mathcal{X}}
\newcommand{\Rbb}{\mathbb{R}}
\newcommand{\Ebb}{\mathbb{E}}
\newcommand{\Nbb}{\mathbb{N}}
\newcommand{\Pbb}{\mathbb{P}}

\newcommand{\level}{\mathrm{level}}
\newcommand{\depth}{\mathrm{depth}}
\newcommand{\rank}{\mathrm{rank}}
\newcommand{\linearspan}{\mathrm{span}}
\newcommand{\range}{\mathrm{Im}}
\newcommand{\id}{{id}}

\newtheorem{theorem}{Theorem}[section]
\newtheorem{remark}[theorem]{Remark}
\newtheorem{example}[theorem]{Example}

\newtheorem{lemma}[theorem]{Lemma}\newtheorem{proposition}[theorem]{Proposition}

\newtheorem{corollary}[theorem]{Corollary}

\usepackage{tikz,pgfplots}
\usetikzlibrary{plotmarks,trees,arrows,shapes}

\usepackage{xcolor}

\title{Higher-order principal component analysis for the approximation of tensors in tree-based low rank formats}

\author{Anthony Nouy\thanks{Centrale Nantes, 
LMJL, UMR CNRS 6629, 
1 rue de la No\"e, BP 92101, 44321 Nantes Cedex 3, France. email: anthony.nouy@ec-nantes.fr} \thanks{This research was supported by the French National Research Agency (grant ANR CHORUS MONU-0005).}}

\date{}							% Activate to display a given date or no date

\begin{document}
\maketitle

\begin{abstract}
This paper is concerned with the approximation of tensors using tree-based tensor formats, which are tensor networks whose graphs are dimension partition trees. We consider Hilbert tensor spaces of multivariate functions defined on a product set equipped with a  probability measure. This includes the case of multidimensional arrays corresponding to finite product sets. 
We propose and analyse an algorithm for the construction of an approximation using only point evaluations of a multivariate function, or evaluations of some entries of a multidimensional array. 
The algorithm is a variant of higher-order singular value decomposition which  constructs a hierarchy of subspaces associated with the different nodes of the tree and a corresponding hierarchy of interpolation operators. Optimal subspaces are estimated using empirical principal component analysis of interpolations of 
partial random evaluations of the function. 
The algorithm is able to provide an approximation in any tree-based format with either a prescribed rank or a prescribed relative error, with a number of evaluations of the order of the storage complexity of the approximation format. Under some assumptions on the estimation of principal components, we prove that the algorithm provides either a quasi-optimal approximation with a given rank, or an approximation satisfying the prescribed relative error, up to constants depending on the tree and the properties of interpolation operators. The  analysis takes into account the discretization errors for the approximation of infinite-dimensional tensors. 
For a tensor with finite and known rank in a tree-based format, the algorithm is able to recover the tensor in a stable way using a number of evaluations equal to the storage complexity of the representation of the tensor in this format. 
Several numerical examples illustrate the main results and the behavior of the algorithm for the approximation of high-dimensional functions using  hierarchical Tucker or tensor train tensor formats, and the approximation of univariate functions using {tensorization}. 
\end{abstract}

\noindent\textbf{Keywords:} high-dimensional approximation, tree-based tensor formats, deep tensor networks, higher-order singular value decomposition, higher-order principal component analysis, interpolation.\\[3pt]
\noindent\textbf{2010 AMS Subject Classifications:} 15A69, 41A05, 41A63, 65D15, 65J99
%  	15A69   	Multilinear algebra, tensor products
% 	41A05   	Interpolation
%	41A63   	Multidimensional problems 
%  	65D05   	Interpolation
%     65D15   	Algorithms for functional approximation
% 	65J99 	Numerical analysis in abstract spaces, None of the above, but in this section
%\\[3pt]

\section{Introduction}

 The approximation of high-dimensional functions is one of the most challenging tasks in computational science. 
Such high-dimensional problems arise in many domains of physics, chemistry, biology or finance, where
 the functions are the  solutions of high-dimensional partial differential equations (PDEs). 
 Such problems also typically arise in statistics or machine learning, for the estimation of high-dimensional probability density functions, 
 or the approximation of the relation between a certain random variable and some predictive variables, the typical task of supervised learning. 
The approximation of high-dimensional functions is also required in optimization or uncertainty quantification problems, where the functions represent the response of a system (or model) in terms of some parameters. 
%(design variables or random variables  reflecting uncertainties on the system (or model) under consideration. 
These problems require many evaluations of the functions and are usually intractable when one evaluation requires a specific experimental set-up or one run of a complex numerical code.  
\\

%Naive (grid-based) approximation methods yield an exponential increase in storage and computational complexity, which is the so-called curse of dimensionality. 
The approximation of high-dimensional functions from a limited number of information on the functions requires exploiting  low-dimensional structures of functions. This usually call for nonlinear approximation tools \cite{DeVore98nonlinearapproximation,TEM03}.
A prominent approach consists of exploiting the sparsity of functions relatively to a basis, a frame, or a more general dictionary of functions \cite{TEM11,BUN04,cohen2015approximation}.  Another approach consists of exploiting low-rank structures of multivariate functions, interpreted as elements of tensor spaces, which is related to notions of sparsity in (uncountably infinite) dictionaries of separable functions. 
For a multivariate function $v(x_1,\hdots,x_d)$ defined on a product set $\Xc_1\times \hdots \times \Xc_d$, which is here identified with a tensor of order $d$, a natural notion of rank is the \emph{canonical rank}, which 
is the minimal integer  $r$ such that  
$$
v(x_1,\hdots,x_d) = \sum_{k=1}^r v^1_k(x_1) \hdots v_k^d(x_d)
$$
for some univariate functions $v^\nu_k$ defined on $\Xc_\nu$. For $d=2$, this corresponds to the unique notion of rank, which coincides with the  matrix rank when the variables take values in finite index sets and $v$ is identified with a matrix. A function with low canonical rank $r$ has a number of parameters which scales only linearly with $r$ and $d$. However, it turns out that this format has several drawbacks when $d>2$ (see, e.g., \cite{DES08,hillar2013most}), which makes it unsuitable for approximation. Then, other notions of rank have been introduced. For a subset of dimensions $\alpha$ in $\{1,\hdots,d\}$, the $\alpha$-rank of a function $v$ is the minimal integer  
$ \rank_\alpha(v)$ such that 
 $$
v(x_1,\hdots,x_d) = \sum_{k=1}^{\rank_\alpha(v)} v^\alpha_k(x_\alpha)  v_k^{\alpha^c}(x_{\alpha^c})
$$
for some functions $v^\alpha_k$ and $v^{\alpha^c}_k$ of complementary  groups of variables $x_\alpha = (x_\nu)_{\nu\in \alpha}\in \Xc_\alpha $ and $x_{\alpha^c} = (x_\nu)_{\nu\in \alpha^c} \in \Xc_{\alpha^c}$, with $\alpha^c$ the complementary  subset of $\alpha$ in 
 $\{1,\hdots,d\}.$
Approximation formats can then be defined by imposing $\alpha$-ranks for a collection of subsets $\alpha$. 
More precisely, if $A$ is a collection of subsets in $\{1,\hdots,d\}$, we define an approximation format 
$$
\Tc_r^A  = \{v : \rank_\alpha(v) \le r_\alpha, \alpha \in A\} = \bigcap_{\alpha\in A} \Tc_{r_\alpha}^{\{\alpha\}},
$$
where $r=(r_\alpha)_{\alpha\in A}$ is a tuple of integers. 
When $A$ is a tree-structured collection of subsets (a subset of a dimension partition tree), $\Tc_r^A $ is a tree-based  tensor format whose elements admit a hierarchical and data-sparse representation. Tree-based tensor formats are tree tensor networks, i.e. tensor networks with tree-structured graphs \cite{Orus2014117}. 
They include the  hierarchical Tucker (HT) format \cite{hackbusch2009newscheme} and the tensor-train (TT) format \cite{oseledets2009breaking}. 
Tree-based formats have many favorable properties that make them favorable for numerical use. As an intersection of subsets of tensors with bounded $\alpha$-rank, $\alpha \in A$, these formats inherit most of the nice properties of the low-rank approximation format for order-two tensors. 
In particular, under suitable assumptions on tensor norms, 
best approximation problems in the set $\Tc_r^A$ are well-posed  \cite{Falco:2012uq,falco2015geometric}. Also, the $\alpha$-rank of a tensor  can be computed through singular value decomposition, and the notion of singular value decomposition can be extended (in different ways) to these formats \cite{DEL00,grasedyck2010,OSE11}. Another interesting property, which is not exploited in the present paper, is the fact that the set $\Tc_r^A$ is a  differentiable manifold \cite{Holtz:2012fk,uschmajew2013geometry,falco2015geometric,falco2016dirac}, which has interesting consequences in optimization or model order reduction of dynamical systems in tensor spaces \cite{Lubich:2013}. There are only a few results available on the approximation properties of tree-based formats \cite{Schneider201456}. However, it has been observed in many domains of applications  that tree-based formats have a high approximation power (or expressive power).  
Hierarchical tensor formats have been recently identified with deep neural networks with a particular architecture \cite{cohen2015expressive}.
%, while canonical tensor format has been identified with a shallow network .

 The reader is referred to the monograph \cite{hackbusch2012book} and surveys \cite{KOL09,Khoromskij:2012fk,Grasedyck:2013,nouy:2016_handbook,nouy:2017_morbook,bachmayr2016tensor} for an introduction to tensor numerical methods  and an overview of recent developments in the field. 
 \\

This paper is concerned with the problem of computing an approximation of a function $u(x_1,\hdots,x_d)$ using 
  point evaluations of this function, where evaluations can be selected adaptively. This includes problems  where the function represents the output of a black-box numerical code, a system or a physical experiment for a given value of the input variables $(x_1,\hdots,x_d)$. This also includes the solution of high-dimensional PDEs with a probabilistic interpretation, where Monte-Carlo methods can be used to obtain point evaluations of their solutions. This excludes problems where evaluations of the functions come as an unstructured data set. A multivariate function $u(x_1,\hdots,x_d) $ is here considered as an element of a Hilbert tensor space $\Hc_1\otimes \hdots \otimes \Hc_d$ of real-valued functions defined on a product set $ \Xc_1\times \hdots \times \Xc_d$ equipped with a probability measure. This  includes the case of multidimensional arrays when the variables $x_\nu$ take values in finite sets $\Xc_\nu$. In this case, a point evaluation corresponds to the evaluation of an entry of the tensor. 
  
Several algorithms have been proposed for the construction of approximations in tree-based formats using point evaluations of functions or entries of tensors. Let us mention algorithms that use adaptive and structured evaluations of tensors  \cite{OSE10,Ballani2013639} and statistical learning approaches  that use unstructured (random) evaluations of functions  \cite{Espig:2009yf,Doostan:2013,Chevreuil2015,chevreuil2017:learning}. Let us also mention the recent work \cite{luu2017new} for the approximation in Tucker format, with an approach similar to the one proposed in the present paper.   

In the present paper, we propose and analyse a new algorithm which is based on a particular extension of the 
 singular value decomposition for the tree-based format $\Tc_r^A$ which allows us to construct an approximation using only evaluations of a function (or entries of a tensor). The proposed algorithm constructs a hierarchy of subspaces $U_\alpha$  of functions of groups of variables $x_\alpha$, for all $\alpha\in A$, and associated interpolation operators $I_{U_\alpha}$ which are oblique projections onto $U_\alpha$.
For the construction of $U_\alpha$ for a particular node $\alpha \in A$, we interpret the function $u$ as a random variable $u(\cdot,x_{\alpha^c})$ depending on a set of random variables $x_{\alpha^c}$ with values in the space of functions of the variables $x_\alpha$. Then $U_\alpha$ is obtained by estimating the principal components of this function-valued random variable using random samples $u(\cdot,x^k_{\alpha^c})$. In practice, we estimate the principal components from interpolations 
$I_{{V_\alpha}} u(\cdot,x^k_{\alpha^c})$ of these samples on a subspace ${V_\alpha}$ which is a certain approximation space when $\alpha$ is a leaf of the tree, or the tensor product of subspaces $\{U_\beta\}_{\beta \in S(\alpha)}$ associated with the sons $S(\alpha)$ of the node $\alpha$ when $\alpha$ is not a leaf of the tree. This construction  only requires 
 evaluations of $u$ on a product set of points which is the product of an interpolation grid in $\Xc_{\alpha}$ (unisolvent for the space ${V_\alpha}$), and a random set of points in $\Xc_{\alpha^c}$. It is a sequential construction going from the leaves to the root of the tree. 
 
The proposed algorithm can be interpreted as an extension of principal component analysis for tree-based tensors which provides  
  a statistical estimation of low-dimensional subspaces of functions of groups of variables for the representation of a multivariate function. 
  It is able to provide an approximation $u^\star$ in any tree-based format $\Tc_r^A$ with either a prescribed rank $r$ or a prescribed relative error (by adapting the rank $r$). For a given $r$, it has the remarkable property that it is able to provide an approximation in $\Tc_r^A$ with a number of evaluations equal to the storage complexity of the resulting approximation.
Under some assumptions on the estimation of principal components, we prove that the algorithm, up to some discretization error $\rho$,  provides with high probability  a quasi-optimal approximation with a prescribed rank, i.e.
$$
\Vert u - u^\star \Vert \le c \min_{v\in \Tc^A_r} \Vert u-v \Vert + \rho,
$$
where the constant $c$ depends on the set $A$ and the properties of orthogonal projections and interpolation operators associated with  principal subspaces. Also, under some assumptions on the estimation of principal components and discretization error, we prove that the algorithm with prescribed tolerance $\epsilon$ is able to provide an approximation $u^\star$ such that 
$$
\Vert u - u^\star \Vert \le \tilde c \epsilon \Vert u \Vert
$$
holds with high probability,  where  the constant $\tilde c$  depends on the set $A$ and the properties of projections and interpolation operators. Sharp inequalities are obtained by considering the properties of projection and interpolation operators when restricted to minimal subspaces of tensors. 
The analysis takes into account the discretization errors for the approximation of infinite-dimensional tensors. 
For a tensor with finite and known rank in a tree-based format, and when there is no discretization error,  the algorithm is able to recover the tensor in a stable way using a number of evaluations equal to the storage complexity of the representation of the tensor in this format. 
 This algorithm may have important applications in the manipulation of big data, by providing a way to reconstruct a multidimensional array from a limited number of entries (tensor completion).
\\

 The outline of the paper is as follows.  
{In section \ref{sec:projections}, we introduce some definitions and properties of projections in Hilbert spaces, with a particular attention on Hilbert spaces of functions and projections based on point evaluations. In section \ref{sec:tensors}, we recall basic definitions on tensors and Hilbert tensor spaces of functions defined on measured product sets. Then we introduce some definitions and properties of operators on tensor spaces, with partial point evaluation functionals as a particular case. Finally, we introduce definitions and properties of projections on tensor spaces, with a particular attention on orthogonal projection and interpolation. In section 
\ref{sec:tree-based-formats}, we introduce tree-based low-rank formats in a general setting including classical HT and TT formats. In section \ref{sec:pca}, we first introduce the notion of principal component analysis for multivariate functions and then propose an extension of principal component analysis to tree-based tensor format. This is based on  a new variant of higher-order singular value decomposition of tensors in tree-based format. In section \ref{sec:empirical-pca}, we present and analyse a modified version of the algorithm presented in section \ref{sec:pca} which only requires point evaluations of functions, and which is based on empirical principal component analyses and interpolations.}
In section \ref{sec:examples}, the behavior of the proposed algorithm is illustrated and analysed in several numerical experiments.

\section{Projections}\label{sec:projections}

 For two vector spaces $V$ and $W$ equipped with norms $\Vert \cdot\Vert_V$ and $\Vert \cdot\Vert_W$ respectively, we denote by $L(V,W)$ the space of linear operators from $V$ to $W$. We denote by $\Lc(V,W)$ the space of linear and continuous operators from $V$ to $W$, with bounded operator norm 
$\Vert A \Vert_{V\to W} = \max_{\Vert v \Vert_V=1} \Vert A v \Vert_W$. We denote by $V^*=L(V,\Rbb)$ the algebraic dual of $V$ and by $V'=\Lc(V,\Rbb)$ the topological dual of $V$, and we let $\Vert \cdot \Vert_{V\to \Rbb} = \Vert \cdot \Vert_{V'}$. We denote by $\langle \cdot , \cdot \rangle$ the duality pairing between a space and its dual.
 We let $L(V) := L(V,V)$ and $\Lc(V) := \Lc(V,V)$, and we replace the notation $\Vert \cdot \Vert_{V\to V}$ by $\Vert \cdot \Vert_V$, where the latter notation also stands for the norm on $V$. 

\subsection{Projections}

Let $V$ be a Hilbert space and $U$ be a {finite-dimensional} subspace of $V$. 
An operator $P$ is a \emph{projection} onto a subspace $U$ if $\range(P)=U$ and $Pu=u$ for all $u\in U$. 
\\ \par 
The \emph{orthogonal projection} $P_U$ onto $U$ is a linear and continuous operator  which associates to $v \in V$ the unique solution $P_U v \in U$ of
$$
\Vert v - P_U v \Vert_V = \min_{u \in U} \Vert v-u \Vert_V,
$$
or equivalently
$
(u,P_U v - v) = 0, \; \forall u\in U.
$
The orthogonal projection $P_U$ has operator norm $\Vert P_U \Vert_V=1$.
\\\par 
Let $W$ be a {finite-dimensional subspace of $V^*$ such that 
\begin{subequations}
\begin{align} 
&\dim(W) = \dim(U), \text{ and} \label{condition_UWa}\\
& \{u \in U : \langle w , u\rangle = 0 \text{ for all } w \in W\} = \{0\},\label{condition_UWb}
\end{align}
 \end{subequations}
 where the latter condition is equivalent to $U\cap {}^\perp W = \{0\}$, with 
 ${}^\perp W $ the \emph{annihilator} of $W$ in $V$ (see \cite[Definition 1.10.4]{megginson2012introduction}).
Under the above assumptions, we have that for any $v\in V$, there exists a unique $u\in U$ such that $\langle w,u-v\rangle = 0$ for all $w\in W$.\footnote{Uniqueness comes from \eqref{condition_UWb} while existence comes from   \eqref{condition_UWa} and \eqref{condition_UWb}.} This allows to define the}    
 projection $P_{U}^W$ onto $U $ along $W $ which is the linear operator on $V$ which associates to $v \in V$ the unique solution $P_{U}^W v \in U$ of 
$$
\langle w ,P_{U}^W v- v \rangle=0,\; \forall w \in W.  
$$ 
For $W = R_V U$, where $R_V : V \to V'$ is the Riesz map, the projection $P_{U}^W$   coincides with the orthogonal  projection $P_U$. A non orthogonal projection is called an \emph{oblique projection}. 
{ 
If $W \subset V',$ then $P_U^W$ is a projection from $V$ onto $U$ parallel to $Ker(P_U^W)=Z^\perp$, where $Z = R_V^{-1}W$. 
If $W\subset \tilde U'$, with $\tilde U $ a closed subspace of $V$, then $P_U^W \vert_{\tilde U}$ is a projection from $\tilde U$ onto $U$ parallel to $Ker(P_U^W) \cap \tilde U = Z^\perp \cap \tilde U$, where $Z = R_{\tilde U}^{-1} W$, with $R_{\tilde U} $ the Riesz map from $\tilde U$ to $\tilde U'$.
}

{\begin{proposition}
Let $\tilde U$ be a closed subspace of $V$ and assume that $U\subset \tilde U$ and $W\subset \tilde U'.$\footnote{Note that $V'\subset \tilde U'$ and we may have $W \not\subset V'.$} Then $P_U^W$ is a continuous operator from $\tilde U$ to $V$.
%, with norm 
%$$
%\Vert P_U^W \Vert_{\tilde U\to V} \le \alpha^{-1},
%$$
%with 
%$$
%\alpha = \min_{0\neq u \in U} \max_{0\neq w \in W} \frac{\langle w,u\rangle}{\Vert u \Vert_V \Vert w \Vert_W}.
%$$
\end{proposition}
\begin{proof}
Let us equip $W$ with the norm $\Vert w \Vert_W = \Vert w \Vert_{\tilde U'} = \max_{v\in \tilde U} \langle w ,v \rangle/\Vert v \Vert_V,$ such that for all $v\in \tilde U$, $\langle w , v \rangle \le \Vert w \Vert_W \Vert v\Vert_V.$ Let 
$$
\alpha = \min_{0\neq u \in U} \max_{0\neq w \in W} \frac{\langle w,u\rangle}{\Vert u \Vert_V \Vert w \Vert_W}.
$$
Assumption \eqref{condition_UWb} implies that $\alpha>0$. Then for all $v\in \tilde U$, we have 
$$
\Vert P_U^W v \Vert_V \le \alpha^{-1} \max_{0\neq w \in W} \frac{\langle w, P_U^W v\rangle}{ \Vert w \Vert_W} = \alpha^{-1} \max_{0\neq w \in W} \frac{\langle w, v\rangle}{ \Vert w \Vert_W} \le \alpha^{-1} \Vert v \Vert_V,
$$
which ends the proof.
\end{proof}
}	
%{\begin{remark}
%Note that the above definition of the oblique projection $P_U^W$ holds when $U$ is an infinite dimensional subspace.
%\end{remark}
%}

\begin{proposition}\label{prop:properties_projections}
Let $P$ and $\tilde P$ be projections onto subspaces $U$ and $\tilde U$ respectively and assume  $U \subset \tilde U$. Then
$$ \tilde P P= P.$$ Moreover,  
 if $P$ and $\tilde P$ are projections along $W$ and $\tilde W$ respectively, with 
$W\subset \tilde W$,
then  $$  \tilde P P = P \tilde P = P.$$
 \end{proposition}
\begin{proof}
For all $v\in V$, $Pv \in U \subset \tilde U$, and therefore 
$\tilde P P v=Pv$, which proves the first statement. For the second statement, by definition of the  projection $P$, we have  that $
\langle \phi,P \tilde Pv - \tilde Pv \rangle = 0
$ for all $\phi\in W$.  Since $W\subset \tilde W$ and by definition of $\tilde P$, this implies that $
\langle \phi,P \tilde Pv - v \rangle = 0
$ for all $\phi \in W$. By definition of $Pv$ and since $P \tilde Pv \in U,$ this implies $P\tilde P v = Pv = \tilde P Pv $.
\end{proof}
\begin{proposition}\label{prop:interpolation_error}
Let $U$ and $\tilde U$ be two closed subspaces of $ V$, {with $U$ of finite dimension}. Let $P_U$ be the orthogonal projection onto $U$ and let $P_U^W$ be the projection onto $U$ along $W \subset \tilde U'$. For all $v \in \tilde  U$, 
$$
\Vert P_U^W v - P_U v \Vert_V \le  \Vert P_U^W - P_U \Vert_{ \tilde U \to V} \Vert v - P_U v \Vert_V,
$$  
with $$\Vert P_U^W - P_U \Vert_{ \tilde U \to V} =  \Vert P_U^W  \Vert_{ (id-P_U)\tilde U \to V} \le\Vert P_U^W  \Vert_{ \tilde U \to V} .$$ Also, for all $v \in \tilde  U$, 
$$
\Vert v - P_U^W v \Vert^2_V \le (1+\Vert P_U^W - P_U \Vert_{ \tilde U \to V}^2) \Vert v - P_U v \Vert_V^2 .
$$
%If $U\subset \tilde U$, then $U^\star$ is the orthogonal complement of $U$ in $\tilde U$ and $\Vert I_U - P_U \Vert_{ U^\star \to V} \le \Vert I_U - P_U \Vert_{ \tilde U \to V}$.
\end{proposition}
\begin{proof}
For $v\in \tilde U$, $\Vert P_U^W v - P_U v \Vert_V = \Vert P_U^W(v -P_U   v)\Vert_V =\Vert (P_U^W-P_U)( v - P_U  v)\Vert_V  \le \Vert P_U^W - P_U \Vert_{(id-P_U)\tilde U  \to V} \Vert v - P_U v \Vert_V$, with $\Vert P_U^W - P_U \Vert_{(id-P_U)\tilde U  \to V}  = \Vert P_U^W - P_U \Vert_{ \tilde U  \to V}=  \Vert P_U^W  \Vert_{ (id-P_U)\tilde U   \to V} $.  %=  \Vert P_U^W  \Vert_{ (id-P_U)\tilde U   \to V}  \le \Vert  P_U^W \Vert_{ \tilde U \to V} .
This proves the first statement. The second statement directly follows from $\Vert  v - P_U^W v \Vert_V^2 = \Vert  v - P_U v\Vert_V^2 + \Vert   P_U v - P_U^W v  \Vert_V^2$.
\end{proof}

\subsection{Projection of functions using point evaluations}

Let $V$ be a Hilbert space of functions defined on a set $X$.  {For $x\in X$,  the 
point evaluation functional $\delta_{x} \in V^*$ is defined by $\langle \delta_{x },v\rangle  =  v(x)$.}

\subsubsection{Interpolation}

Let $U $ be a $n$-dimensional subspace of $V$ and let $\Gamma = \{x^k\}_{k=1}^n$ be a set of $n$ interpolation points in $X$. The set of interpolation points $\Gamma$ is assumed to be unisolvent for $U$, i.e. for any $(a_k)_{k=1}^n \in \Rbb^n$, there exists a unique $u\in U$ such that $u(x^k) = a_k$ for all $1\le k\le n$. 
The interpolation operator $I_U$ associated with $\Gamma$ is a linear operator from $V$ to $U$ such that for $v\in V$, $I_U v$ is the unique element of $U$ such that 
$$
\langle \delta_{x},I_U v - v \rangle = I_Uv(x) - v(x) = 0 \quad \forall x \in \Gamma.
$$
%$$
%I_U v(x^k) = v(x^k), \quad 1\le k\le n,
%$$
The interpolation operator $I_U$ 
 is an oblique projection $P_U^W$ onto 
$U$ along $W = \linearspan\{\delta_{x} : x\in \Gamma\}$.
{Note that the condition that $\Gamma$ is unisolvent for $U$ is equivalent to the condition $\eqref{condition_UWb}$ on $U$ and $W$, which ensures that $I_U$ is well defined}.
 From Proposition \ref{prop:properties_projections}, we deduce the 
  following property.
\begin{proposition}\label{interpolations_nestedgrids_commute}
Let $U$ and $\tilde U$ be two subspaces associated with sets of interpolation points $\Gamma$ and $\tilde \Gamma$ respectively. If $U \subset \tilde U$ and $\Gamma \subset \tilde \Gamma$, then $$
I_{U}I_{\tilde U} = I_{\tilde U} I_U = I_U.
$$ 
\end{proposition}
\paragraph{Magic points.}\label{sec:magicpoints}
For a given basis $
 \{\varphi_i\}_{i=1}^n$ of $U$, a set of interpolation points $\Gamma = \{x^k\}_{k=1}^n$, called \emph{magic points}, can be determined with a greedy algorithm proposed in \cite[Remark 2]{MAD09}. The procedure for selecting the set $\Gamma$ in a subset $\Gamma_\star$ in $X$ is as follows. 
We first determine a point $x^1\in \Gamma_\star$ and an index $i_1$ such that 
$$\vert \varphi_{i_1}(x^1) \vert = \max_{x\in \Gamma_\star} \max_{1\le i \le n} \vert \varphi_i(x) \vert.$$
Then for $k\ge 1$, we define $\psi_i^{(k)}(x) = \varphi_i(x) -\sum_{m=1}^{k} \sum_{p=1}^{k}  \varphi_{i_m}(x) a^{(k)}_{m,p} \varphi_i(x^p)$, with the matrix $(a^{(k)}_{m,p})_{1\le m,p\le k}$ being the inverse of the matrix \\ $(\varphi_{i_m}(x^p))_{1\le p \le k,1\le m \le k}$, 
 such that $\psi_{i_m}^{(k)}(x)=0$ for all $1\le m \le k$ and $x\in X$, and $\psi_{i}^{(k)}(x^p)=0$ for all $1\le p \le k$ and $1\le i \le n$. Then, we determine the point $x^{k+1} \in \Gamma_\star$ and an index $i_{k+1}$ such that 
$$\vert \psi^{(k)}_{i_{k+1}}(x^{k+1}) \vert = \max_{x\in \Gamma_\star} \max_{1\le i \le n} \vert \psi_i^{(k)}(x) \vert.$$

% \paragraph{Interpolation of tensors.}
%Consider  a non-empty set $\alpha $ in  $D$.  Let $U_\alpha$ be a $n_\alpha$-dimensional subspace of $\Hc_\alpha$ and % with basis $\{\varphi^\alpha_k\}_{k=1}^{n_\alpha}$. 
%let $\Gamma_\alpha = \{x_\alpha^k\}_{k=1}^{n_\alpha}$ in $\Xc_\alpha$ be an associated unisolvent set of interpolation points and $W_\alpha = \linearspan\{\delta_{x_\alpha}\}_{x_\alpha\in \Gamma_\alpha}$. We denote by  $I_{U_\alpha} : \Hc_\alpha \to U_\alpha$ the interpolation operator 
% associated with the set of points $\Gamma_\alpha$, and by $\Ic_{U_\alpha} = I_{U_\alpha} \otimes id_{\alpha^c}$ the corresponding projection from $\Hc$ onto $U_\alpha \otimes \Hc_{\alpha^c}$. If $W_\alpha \subset \Hc_\alpha'$, then $I_{U_\alpha}$ and $\Ic_{U_\alpha}$ are continuous operators with equal norms 
% $
% \Vert \Ic_{U_\alpha} \Vert = \Vert I_{U_\alpha} \Vert_{\Hc_\alpha}.
% $ 

\subsubsection{Discrete least-squares projection}			
Let $U$ be a $n$-dimensional subspace of $V$ and let $\Gamma = \{x^k\}_{k=1}^m$ be a set of $m$ points in $X$, $m\ge n$, such that 
$\Vert v \Vert_\Gamma = (\sum_{x\in \Gamma} v(x)^2)^{1/2}$ defines a norm on $U$. The discrete least-squares projection $Q_U$ is the linear operator from $V$ to $U$ such that for $v\in V$, $Q_U v$ is the unique element in $U$ which minimizes  $\Vert v - u \Vert_\Gamma^2$ over all $u\in U$, or equivalently
$$
(u,v-Q_Uv )_\Gamma = \sum_{x\in \Gamma} u(x) \langle  \delta_x , v-Q_U v  \rangle = 0 \quad \forall u \in U,
$$
where $(\cdot,\cdot)_\Gamma$ is the inner product   associated with the norm $\Vert\cdot\Vert_\Gamma$ on $U$. The discrete least-squares projection $Q_U$ is an oblique projection onto $U$ along $W = \{\sum_{x\in \Gamma} u(x) \delta_x  : u \in U\}$. % = \linearspan\{\sum_{x\in \Gamma} \delta_x \varphi_i(x) : 1\le i \le n\},$ where $\{\varphi_i\}_{i=1}^n$ is a basis of $U$.
%$$
%\Vert v - Q_U v \Vert_\Gamma^2 = \min_{w \in U} \Vert v - w \Vert_\Gamma^2
%$$		
If $\#\Gamma = \dim(U)$ and $\Gamma$ is unisolvent for $U$, then $Q_U$ coincides with the interpolation operator $I_U$. 
\begin{proposition}\label{discrete_projections_nestedgrids_commute}
Let $U$ and $\tilde U$ be two finite-dimensional subspaces such that $U\subset \tilde U$. Let 
$Q_U$  be the discrete least-squares projection onto $U$ associated with a set of points $\Gamma$ in $X$, and let $Q_{\tilde U}$  be the discrete least-squares projection onto $\tilde U$ associated with  a set of points $\tilde \Gamma$ in $X$.  If either $\Gamma = \tilde \Gamma$ or 
$\Gamma \subset \tilde \Gamma$ and $\tilde \Gamma$ is unisolvent for $\tilde U$,   then 
$$
Q_{U}Q_{\tilde U} = Q_{\tilde U} Q_U = Q_U.
$$ 
\end{proposition}
\begin{proof}
$Q_U$ is the projection onto $U$ along $ W= \{\sum_{x\in \Gamma} u(x) \delta_x  : u \in U\}$, and 
$Q_{\tilde U}$ is the projection onto $\tilde U$ along $ \tilde W= \{\sum_{x\in \tilde \Gamma} \tilde u(x) \delta_x  : \tilde u \in \tilde U\}$. If we prove that $W\subset \tilde W$, then  the result follows from 
Proposition \ref{prop:properties_projections}. Let 
$w = \sum_{x\in \Gamma} u(x) \delta_x \in W$, with $u\in U$. If $\Gamma = \tilde \Gamma$, then since $u \in \tilde U$, we clearly have $w\in \tilde W$. If $\Gamma \subset \tilde \Gamma$
and $\tilde \Gamma$ is unisolvent for $\tilde U$, there exists a function $\tilde u \in \tilde U$ such that $\tilde u(x) = u(x)$ for all $x\in \Gamma$ and $\tilde u(x) = 0$ for all $x\in \tilde \Gamma\setminus \Gamma$. Therefore, $w = \sum_{x\in \tilde \Gamma} \tilde u(x) \delta_x$ is an element of $\tilde W$, which ends the proof. 
%Let $\{\varphi_i\}_{i=1}^{n}$ be a basis of $U$ and  $\{\varphi_i\}_{i=1}^{\tilde n}$ be a basis of $\tilde U$. Let $\Gamma = \{x^1,\hdots,x^m\}$, $\tilde \Gamma = \{x^1,\hdots,x^{\tilde m}\}$,  $W = \linearspan\{\sum_{k=1}^m \delta_{x^k} \varphi_i(x^k) \}_{i=1}^n$ and $\tilde W = \linearspan\{\sum_{k=1}^{\tilde m} \delta_{x^k} \varphi_i(x^k) \}_{i=1}^{\tilde n}$. To prove that $W \subset \tilde \tilde W$, we have to prove that matrices $\boldsymbol{\Phi} = ( \varphi_i(x^k))_{1\le k \le m,1\le i \le n} \in \Rbb^{m\times n}$ and $\boldsymbol{\tilde \Phi} = ( \varphi_i(x^k))_{1\le k \le \tilde m,1\le i \le \tilde n} \in \Rbb^{\tilde m\times \tilde n}$ are such that there exists a matrix $\mathbf{G} \in \Rbb^{\tilde n\times n }$ such that $ \boldsymbol{\tilde \Phi} \mathbf{G}  = \begin{pmatrix} \boldsymbol{ \Phi}  \\ \mathbf{0} \end{pmatrix}$.
\end{proof}

\section{Tensors}\label{sec:tensors}

Let $\Hc_\nu$ be Hilbert spaces of real-valued functions defined on sets $\Xc_\nu$ equipped with probability  measures $\mu_\nu$, $1\le \nu \le d$. We denote by $\Vert \cdot \Vert_{\Hc_\nu}$ the norm on $\Hc_\nu$ and by $(\cdot,\cdot)_{\Hc_\nu}$ the associated inner product. 
Let $\Xc = \Xc_1\times \hdots \times \Xc_d$ and $\mu = \mu_1\otimes \hdots \otimes \mu_d$. 
The tensor product of $d$ functions $v^\nu \in \Hc_\nu$, $1\le \nu\le d$, denoted $v^1\otimes \hdots \otimes v^d$, is a multivariate function defined on $\Xc$ such that $(v^1\otimes \hdots \otimes v^d)(x) = v^1(x_1)\hdots v^d(x_d)$ for $x=(x_1,\hdots,x_d) \in \Xc$. Such a function is called an \emph{elementary tensor}. 
The algebraic tensor space $\Hc_{1}\otimes_a \hdots \otimes_a \Hc_d$ is defined as the linear span 
of all elementary tensors, which is a pre-Hilbert space when equipped with the canonical inner product $(\cdot,\cdot)$ defined for elementary tensors by 
$$
( v^1\otimes \hdots \otimes v^d,w^1\otimes \hdots w^d ) = ( v^1 , w^1)_{\Hc_1} \hdots  ( v^d ,w^d)_{\Hc_d},
$$ 
and then extended by linearity to the whole algebraic tensor space. We denote by $\Vert \cdot\Vert$ the norm associated with inner product $(\cdot,\cdot)$.
A Hilbert tensor space 
$
\Hc = 
\overline{\Hc_{1}\otimes_a \hdots \otimes_a \Hc_d}^{\Vert\cdot\Vert}$
is then obtained by the completion of the algebraic tensor space, which we simply denote 
$$
\Hc = \Hc_1\otimes \hdots \otimes \Hc_d = \bigotimes_{\nu=1}^d \Hc_\nu.
$$
\begin{example}
Consider finite sets $\Xc_\nu$ and  $\Hc_\nu = \Rbb^{\Xc_\nu}$ equipped with the norm $\Vert v \Vert_{\Hc_\nu}^2 = \sum_{x_\nu \in \Xc_\nu} \mu_\nu(\{x_\nu\}) \vert v(x_\nu)\vert^2$. Then, $\Hc$ is the space of multidimensional arrays  $  \Rbb^{\Xc_1}\otimes \hdots \otimes \Rbb^{\Xc_d}$ and $\Vert v \Vert^2 = \sum_{x \in \Xc} \mu(\{x\}) \vert v(x) \vert^2 $, where $\mu(\{x_1,\hdots,x_d\}) = \prod_{\nu=1}^d \mu_\nu(\{x_\nu\})$. 
\end{example}
\begin{example}\label{ex:L2}
Consider $\Xc_\nu = \Rbb$, $\mu_\nu$ a finite measure on $\Rbb$, and $\Hc_\nu = L^2_{\mu_\nu}(\Xc_\nu)$ equipped with the natural norm $\Vert v \Vert_{\Hc_\nu}^2 = \int \vert v(x_\nu) \vert^2 \mu_\nu(dx_\nu)$. Then $\Hc$ is identified with $L^2_\mu(\Xc)$, where $\mu = \mu_1\otimes \hdots \otimes \mu_d$, and $\Vert v \Vert^2 = \int \vert v(x) \vert^2 \mu(dx).$
\end{example}
\begin{example}
Consider for $\Hc_\nu$ a reproducing kernel Hilbert space (RKHS) with reproducing kernel $k_\nu : \Xc_\nu\times \Xc_\nu \to \Rbb$. 
%For $x_\nu \in \Xc_\nu$, the point evaluation functional $\delta_{x_\nu} : v \mapsto v(x_\nu)$ is linear and continuous on $\Hc_\nu$ and $v(x_\nu)=(k_\nu(x_\nu,\cdot),v)_{\Hc_\nu} $. 
Then $\Hc$ is a RKHS with reproducing kernel $k(x,x') = k_1(x_1,x_1')\hdots k_d(x_d,x_d')$.
%, and $v(x) = (k(x,\cdot),v)$.
\end{example}
For a non-empty subset $\alpha$ in $\{1,\hdots,d\}:=D$, we let $\Xc_\alpha $ be the set $ \bigtimes_{\nu\in \alpha} \Xc_\nu$ equipped with the product measure $\mu_\alpha = \bigotimes_{\nu \in \alpha} \mu_\nu$. We denote 
by $\Hc_\alpha = \bigotimes_{\nu \in \alpha} \Hc_\nu$ the Hilbert tensor space of functions defined on $\Xc_\alpha$, equipped with the canonical norm $\Vert \cdot \Vert_{\Hc_\alpha}$ such that 
$$
\Vert \bigotimes_{\nu\in \alpha} v^\nu \Vert_{\Hc_\alpha} = \prod_{\nu \in \alpha} \Vert v^\nu \Vert_{\Hc_\nu}$$
for $v^\nu \in \Hc_\nu$, $1\le \nu\le d$.  We have $\Hc_D = \Hc$ and we use the convention $\Hc_\emptyset = \Rbb$.
\paragraph{Matricisations and $\alpha$-ranks.}

Let $\alpha\subset D$, with $\alpha \notin \{\emptyset,D\}$, and let $\alpha^c = D\setminus \alpha$ be its complement in $D$. For $x\in \Xc$, we denote by $x_\alpha$ the subset of variables $ (x_\nu)_{\nu\in \alpha}$. A tensor 
 $v \in \Hc$ can be identified with an order-two tensor $$\Mc_\alpha(v) \in \Hc_\alpha \otimes \Hc_{\alpha^c},$$ where $\Mc_\alpha$ is the \emph{matricisation operator} associated with $\alpha$, which defines a linear isometry between $\Hc$ and $\Hc_\alpha \otimes \Hc_{\alpha^c}$. We use  the conventions $\Mc_\emptyset (v) = \Mc_D (v) = v$ and $\Hc_\emptyset \otimes \Hc_D = \Hc_D \otimes \Hc_\emptyset = \Hc$.
 %, i.e. 
 %$
%\Vert \Mc_\alpha v \Vert = \Vert v \Vert.
% $ 
 The \emph{$\alpha$-rank} of a tensor $v \in \Hc$, denoted $\rank_\alpha(v)$, is defined as the rank of the order-two tensor $\Mc_{\alpha}(v)$, which is uniquely defined as the minimal integer such that 
 \begin{align}
 \Mc_{\alpha}(v) = \sum_{k=1}^{\rank_\alpha(v)} v^\alpha_k \otimes v^{\alpha^c}_k,\quad \text{or equivalently} \quad v(x) = \sum_{k=1}^{\rank_\alpha(v)} v^\alpha_k(x_\alpha)  v^{\alpha^c}_k(x_{\alpha^c}),\label{decomp_alpha}
 \end{align}
  for some functions $v^{\alpha}_k\in \Hc_\alpha$ and $v^{\alpha^c}_k \in \Hc_{\alpha^c}$ of complementary subsets of variables $x_\alpha$ and $x_{\alpha^c}$  respectively.
By convention, we have $\rank_\emptyset(v) = \rank_D(v) = 1.$ 
 From now on, when there is no ambiguity, $\Mc_\alpha(v)$ and $\Hc_\alpha  \otimes \Hc_{\alpha^c}$ will be identified with $v$ and $\Hc$ respectively.

\paragraph{Minimal subspaces.}
The \emph{minimal subspace} $U^{min}_\alpha(v)$ of $v$ is defined as the smallest closed subspace in $ \Hc_\alpha$ such that 
  $$
v \in U^{min}_\alpha(v) \otimes \Hc_{\alpha^c},
  $$
  and we have
  $\rank_\alpha(v) = \dim(U^{min}_\alpha(v))$
 (see \cite{Falco:2012uq}). 
 If $v$ admits the representation  \eqref{decomp_alpha}, then 
 $
 U^{min}_\alpha(v)$ is the closure of $ {\linearspan\{v^\alpha_k\}_{k=1}^{\rank_\alpha(v)}}.
 $
For any partition $S(\alpha)$ of $\alpha$, we have 
$$
U^{min}_{\alpha}(v) \subset \bigotimes_{\beta \in S(\alpha)} U^{min}_{\beta}(v).
$$
 We have $U^{min}_D(v) =\Rbb v$ and for any partition $S(D)$ of $D$, %and in particular for $A=\{\{1\},\hdots,\{d\}\}$, 
 $$
 v \in \bigotimes_{\beta \in S(D)} U^{min}_{\beta}(v) .
 $$

 \subsection{Operators on tensor spaces}\label{sec:operators}
 
  Let consider the Hilbert tensor space $\Hc = \bigotimes_{\nu=1}^d \Hc_\nu$ equipped with the canonical norm $\Vert \cdot\Vert $. For linear operators from $\Hc$ to $\Hc$, we also denote by $ \Vert \cdot\Vert $ the operator norm $\Vert \cdot\Vert_{\Hc\to \Hc} = \Vert \cdot\Vert_\Hc$.

 We denote by $\id$ the identity operator on $\Hc$. For a non-empty subset
 $\alpha \subset D$, we denote by $\id_\alpha$ the identity operator on $\Hc_\alpha$.
 For $A_\alpha$ in $L(\Hc_\alpha)$, we define the linear operator 
 $
 A_{\alpha} \otimes \id_{\alpha^c}$ such that 
  for $v^\alpha \in \Hc_\alpha$ and $ v^{\alpha^c} \in \Hc_{\alpha^c}$,
 $$
 (A_{\alpha} \otimes \id_{\alpha^c})(v^\alpha\otimes v^{\alpha^c} ) = (A_\alpha v^\alpha)\otimes v^{\alpha^c},
  $$
 and we extend this definition by linearity to the whole algebraic tensor space $\Hc_{\alpha}\otimes_a \Hc_{\alpha^c}$. 
For a finite dimensional tensor space $\Hc$, this completely characterizes a linear operator on $\Hc$. For 
an infinite dimensional tensor space $\Hc$, if $A_\alpha \in \Lc(U_\alpha,\Hc_\alpha)$, with $U_\alpha \subset \Hc_\alpha$, then $A_{\alpha} \otimes \id_{\alpha^c}$ can be extended by continuity to $U_\alpha \otimes \Hc$.

We denote by $\Ac_\alpha$, using calligraphic font style, the linear operator in $L(\Hc)$ associated with an operator  $A_\alpha$ in $L(\Hc_\alpha)$, defined by 
$
\Ac_\alpha = \Mc_\alpha^{-1} (A_{\alpha} \otimes \id_{\alpha^c}) \Mc_\alpha,
$
and simply denoted 
$$
\Ac_\alpha = A_{\alpha} \otimes \id_{\alpha^c}
$$  
when there is no ambiguity.  
If $A_\alpha \in \Lc(\Hc_\alpha)$, then 
 $\Ac_\alpha  \in \Lc(\Hc)$ and the two operators have the same operator norm 
$
\Vert \Ac_\alpha \Vert  = \Vert A_\alpha \Vert_{\Hc_\alpha}.$ Also, we have the  following more general result.  
\begin{proposition}\label{prop:Aalpha_correspondance}
If $A_\alpha \in \Lc(U_\alpha,\Hc_\alpha)$, with $U_\alpha \subset \Hc_\alpha$, then 
 $\Ac_\alpha  \in \Lc(U_\alpha\otimes \Hc_{\alpha^c},\Hc)$ and the two operators have the same operator norm 
$$
\Vert \Ac_\alpha \Vert_{U_\alpha\otimes \Hc_{\alpha^c} \to \Hc}  = \Vert A_\alpha \Vert_{U_\alpha \to \Hc_\alpha}.$$   
\end{proposition}
\begin{corollary}\label{lem:operator_continuous_Umin}
For a tensor $v\in \Hc$ and an operator 
$A_\alpha \in \Lc(U^{min}_\alpha(v),\Hc_\alpha)$, 
 $$
 \Vert \Ac_\alpha v \Vert \le  \Vert A_{\alpha} \Vert_{U^{min}_\alpha(v)\to \Hc_\alpha}  \Vert v \Vert 
 .
 $$
\end{corollary}

\bigskip

%For two non-empty subsets $\alpha$ and $\beta$ in $D$, let $A_\alpha$ and $A_\beta$ be two linear operators in $\Lc(\Hc_\alpha)$ and $\Lc(\Hc_\beta)$ respectively, and let $\Ac_\alpha$ and $\Ac_\beta$ be the corresponding operators in $\Lc(\Hc).$
%\begin{lemma}
%For $\alpha \subset \beta$ or 
%\end{lemma}
%	
%\bigskip

Let $S = \{\alpha_1,\hdots,\alpha_K\}$ be a collection of disjoint subsets of $D$ and let $A_{\alpha} \in L(\Hc_\alpha)$ be linear operators, $\alpha \in S$. Then we can define a linear operator $A_{\alpha_1} \otimes \hdots \otimes A_{\alpha_K}:=\bigotimes_{\alpha\in S} A_\alpha $ on $\Hc_{\alpha_1} \otimes_a \hdots \otimes_a \Hc_{\alpha_K}$ such that 
$$
(\bigotimes_{\alpha\in S} A_\alpha )(\bigotimes_{\alpha\in S} v^\alpha) = \bigotimes_{\alpha\in S} (A_\alpha v^\alpha)
$$
for $v^\alpha \in \Hc_\alpha$, $\alpha\in S$. 
The operator 
 $\bigotimes_{\alpha\in S} A_\alpha$  can be identified with an operator $$\Ac = \prod_{\alpha\in S} \Ac_\alpha ,$$
defined on the algebraic tensor space $\Hc_1\otimes_a \hdots \otimes_a \Hc_d$. The definition of $\Ac$ is independent of the ordering of the elements of $S$. If the operators $A_\alpha$ are continuous, then $\Ac$ defines a continuous operator from $\Hc$ to $\Hc$ and 
since $\Vert \cdot \Vert$ is a uniform crossnorm (see \cite[Proposition 4.127]{hackbusch2012book}), the operator $\Ac$ has for 
 operator norm 
 $$
\Vert \Ac \Vert = \prod_{\alpha \in S} \Vert \Ac_\alpha \Vert = \prod_{\alpha\in S} \Vert A_\alpha \Vert_{\Hc_\alpha}.
$$
Also, we have the following more general result.
\begin{proposition}
Let $S$ be a collection of disjoint subsets of $D$ and let $\beta \subset D$ such that $\beta \cup (\cup_{\alpha\in S} \alpha)=D$. Let $U_\alpha$ be a subspace of $\Hc_\alpha$ and 
 $A_{\alpha} \in \Lc(U_\alpha, \Hc_\alpha)$, for $\alpha \in S$. Then $\Ac = \prod_{\alpha \in S} \Ac_\alpha$ is a continuous operator from $\Uc := (\bigotimes_{\alpha \in S} U_\alpha)\otimes \Hc_{\beta}$ to $\Hc$
 such that
 $$
 \Vert \Ac \Vert_{\Uc \to \Hc} = \prod_{\alpha \in S} \Vert \Ac_\alpha \Vert_{U_\alpha\otimes \Hc_{\alpha^c} \to \Hc} =\prod_{\alpha \in S} \Vert A_\alpha \Vert_{U_\alpha  \to \Hc_\alpha}  .
 $$  
\end{proposition}
\begin{corollary}
Let $S$ be a collection of disjoint subsets of $D$. For a tensor $v\in \Hc$ and operators $A_\alpha$, $\alpha \in S$, such that $A_\alpha \in \Lc(U^{min}_\alpha(v),\Hc_\alpha)$, the operator $\Ac = \prod_{\alpha \in S}  \Ac_\alpha$ is such that 
 $$
 \Vert \Ac v \Vert \le \Vert v \Vert \prod_{\alpha\in S} \Vert A_{\alpha} \Vert_{U^{min}_\alpha(v)\to \Hc_\alpha}  .
 $$
\end{corollary}

 \subsection{Partial evaluations of tensors}

Let $\alpha$ be a
 non-empty subset of $ D$. For a linear form $\psi_\alpha \in \Hc_\alpha^*$, $\psi_\alpha \otimes id_{\alpha^c}$ is a  linear operator from $\Hc_{\alpha}\otimes_a \Hc_{\alpha^c}$ to $\Hc_{\alpha^c}$ such that 
$
(\psi_\alpha\otimes id_{\alpha^c} )(v^\alpha\otimes v^{\alpha^c}) = \psi_\alpha(v^\alpha) v^{\alpha^c}.$ 
If $\psi_\alpha \in \Hc_{\alpha}'$, the definition of $\psi_\alpha \otimes id_{\alpha^c}$  can be extended by continuity to $\Hc$. Then $\psi_\alpha \otimes id_{\alpha^c}$ is a continuous operator from $\Hc$ to $\Hc_{\alpha^c}$ with operator norm
$\Vert \psi_{\alpha} \otimes id_{\alpha^c} \Vert_{\Hc \to \Hc_{\alpha^c}}= \Vert \psi_\alpha \Vert_{\Hc_\alpha'}.$ Also, we have the following result. 
\begin{proposition}\label{prop:partial_linear_form_subspace}
If $\psi_\alpha \in U_{\alpha}'$, with $U_\alpha $ a subspace of $ \Hc_\alpha$, then 
$\psi_\alpha \otimes id_{\alpha^c} \in \Lc(U_\alpha\otimes \Hc_{\alpha^c},\Hc_{\alpha^c})$  and 
$$
\Vert \psi_{\alpha} \otimes id_{\alpha^c} \Vert_{U_\alpha\otimes \Hc_{\alpha^c} \to \Hc_{\alpha^c}}= \Vert \psi_\alpha \Vert_{U_\alpha'}.
$$
\end{proposition}
\begin{corollary}\label{lem:partial_linear_form}
For a tensor $v \in \Hc$ and $\psi_\alpha \in U^{min}_\alpha(v)'$, we have 
$$
\Vert (\psi_\alpha \otimes id_{\alpha^c}) v \Vert \le \Vert \psi_{\alpha} \Vert_{(U^{min}_\alpha(v))'} \Vert v \Vert.
$$
\end{corollary}

%Let $V$ be a Hilbert space of functions defined on a set $X$. For $x\in X$, we define the 
%point evaluation functional $\delta_{x} \in V^*$ by $\langle \delta_{x },v\rangle  =  v(x)$. 
For 
  a point $x_\alpha \in \Xc_\alpha$, we denote by $\delta_{x_\alpha} \in \Hc_\alpha^*$   the point evaluation functional at $x_\alpha$, defined by $\langle \delta_{x_\alpha} , v^\alpha \rangle = v^\alpha(x_\alpha)$ for $v^\alpha \in \Hc_\alpha$.
Then $\delta_{x_\alpha} \otimes id_{\alpha^c}$ defines a partial evaluation functional, which is a linear operator from $\Hc$ to $\Hc_{\alpha^c}$ such that 
 $$
 (\delta_{x_\alpha}\otimes id_{\alpha^c})(v^\alpha\otimes v^{\alpha^c}) = v^\alpha(x_\alpha)v^{\alpha^c}.
 $$
From Corollary \ref{lem:partial_linear_form}, we deduce that 
for a given tensor $v \in \Hc$, if $\delta_{x_{\alpha}} \in U^{min}_\alpha(v)'$, then the definition of $\delta_{x_{\alpha}}\otimes id_{\alpha^c}$ can be extended by continuity to 
 $U^{min}_\alpha(v) \otimes \Hc_{\alpha^c}$ and the partial evaluation 
$$
v(x_\alpha,\cdot) = (\delta_{x_\alpha}\otimes id_{\alpha^c})v 
$$
is an element of $\Hc_{\alpha^c}$ such that 
$$
\Vert v(x_\alpha,\cdot) \Vert_{\Hc_{\alpha^c}} = \Vert (\delta_{x_\alpha}\otimes id_{\alpha^c})v  \Vert \le \Vert \delta_{x_\alpha} \Vert_{ U^{min}_\alpha(v)'} \Vert v \Vert.
$$

%If $\Kc_\alpha$ is a RKHS with reproducing kernel $k_\alpha$, then $\delta_{x_\alpha} \in \Kc_\alpha'$ for all $x_\alpha \in \Xc_\alpha$, and $\Vert \delta_{x_\alpha} \Vert_{\Kc_{\alpha}'} = k_\alpha(x_\alpha,x_\alpha)$.

\subsection{Projection of tensors}
Let $\alpha $ be a non-empty and strict subset of $D$ and let $U_\alpha$ be a {finite-dimensional subspace} of $\Hc_\alpha$.
If   $P_{\alpha}$ is a projection from $\Hc_\alpha$ onto $U_\alpha$, then 
$P_{\alpha}\otimes id_{\alpha^c}$ is a projection from $\Hc_\alpha\otimes \Hc_{\alpha^c}$ onto $U_\alpha \otimes \Hc_{\alpha^c}$. 

\begin{proposition}\label{prop:rank_bound}
Let $v \in \Hc$ and $\alpha,\beta \subset D$. Let $P_{\beta}$ be a projection 
from $\Hc_\beta$ to a subspace $U_\beta$ and let $\Pc_{\beta}$ be the corresponding projection onto $U_\beta\otimes \Hc_{\beta^c}$.  
  If $\beta \subset \alpha$ or $\beta \subset D\setminus \alpha$, we have
$$
\rank_\alpha(\Pc_{\beta} v) \le \rank_\alpha( v).
$$
\end{proposition}
\begin{proof}
A tensor $v$ admits a representation
$
v = \sum_{k=1}^{\rank_\alpha(v)} v^\alpha_k \otimes w^{\alpha^c}_k.
$
If $\beta \subset \alpha$, then $\Pc_{\beta} = (P_{\beta}\otimes \id_{\alpha\setminus \beta} ) \otimes \id_{D\setminus \alpha}$ and 
$
\Pc_{\beta} v = \sum_{k=1}^{\rank_\alpha(v)} ((P_{\beta}\otimes \id_{\alpha\setminus \beta})v^\alpha_k) \otimes w^{\alpha^c}_k. $
If  $\beta \subset D\setminus \alpha$, then $\Pc_{\beta} = \id_\alpha \otimes (P_{\beta} \otimes \id_{D\setminus \{\alpha \cup \beta\}})$ and
$
\Pc_{\beta} v = \sum_{k=1}^{\rank_\alpha(v)} v^\alpha_k \otimes ((P_{\beta}\otimes \id_{D\setminus \{\alpha \cup \beta\}})w^{\alpha^c}_k). 
$
The result follows from the definition of the $\alpha$-rank. 
\end{proof}
If $P_{U_\alpha}$ is the orthogonal projection from $\Hc_\alpha$ onto $U_\alpha$, then  
$P_{U_\alpha}\otimes id_{\alpha^c}$ coincides with the orthogonal projection $P_{U_\alpha\otimes \Hc_{\alpha^c}}$ from $\Hc_\alpha\otimes \Hc_{\alpha^c}$ onto $U_\alpha \otimes \Hc_{\alpha^c}$, and is identified with the orthogonal projection  $\Pc_{U_\alpha} = P_{U_\alpha}\otimes id_{\alpha^c}$ in $\Lc(\Hc)$. 
If $P_{U_\alpha}^{W_\alpha}$ is the oblique projection onto $U_\alpha$ along $W_\alpha \subset \Hc_\alpha^*$, then $\Pc_{U_\alpha}^{W_\alpha}:=P_{U_\alpha}^{W_\alpha}\otimes id_{\alpha^c}$ is the oblique projection  from $\Hc_\alpha\otimes \Hc_{\alpha^c}$ onto $U_\alpha \otimes \Hc_{\alpha^c}$ along $W_\alpha \otimes \Hc_{\alpha^c}'$. If $W_\alpha \subset \Hc_\alpha'$, then  $P_{U_\alpha}^{W_\alpha}$ and $\Pc_{U_\alpha}^{W_\alpha}$ are continuous operators with equal norms $\Vert \Pc_{U_\alpha}^{W_\alpha} \Vert = \Vert P_{U_\alpha}^{W_\alpha} \Vert_{\Hc_\alpha}.$
 
\begin{proposition}\label{prop:oblique_projections_norms}
Let $U_\alpha$ be a {finite-dimensional} subspace of $\Hc_\alpha$ and let $P_{U_\alpha}^{W_\alpha}$ be the  projection onto $U_\alpha$ along  $W_\alpha$. For a tensor $v\in \Hc$ such that $W_\alpha \subset U^{min}_\alpha(v)'$, $\Pc^{W_\alpha}_{U_\alpha} v$ is an element of $U_\alpha \otimes \Hc_{\alpha^c}$ such that 
 $$
 \Vert \Pc^{W_\alpha}_{U_\alpha} v \Vert \le \Vert P^{W_\alpha}_{U_\alpha} \Vert_{U^{min}_\alpha(v) \to \Hc_\alpha } \Vert v \Vert,
 $$
 and
 $$
  \Vert \Pc^{W_\alpha}_{U_\alpha} v - \Pc_{U_\alpha} v \Vert \le \Vert P^{W_\alpha}_{U_\alpha} - P_{U_\alpha} \Vert_{U^{min}_\alpha(v) \to \Hc_\alpha } \Vert v \Vert,
 $$
with 
$$\Vert P^{W_\alpha}_{U_\alpha} - P_{U_\alpha} \Vert_{U^{min}_\alpha(v) \to \Hc_\alpha } = \Vert P^{W_\alpha}_{U_\alpha}  \Vert_{(id_\alpha - P_{U_\alpha}) U^{min}_\alpha(v) \to \Hc_\alpha } \le \Vert P^{W_\alpha}_{U_\alpha} \Vert_{U^{min}_\alpha(v) \to \Hc_\alpha }.$$
Also,  
 $$ \Vert v - \Pc^{W_\alpha}_{U_\alpha} v \Vert^2 \le ({1+\Vert P^{W_\alpha}_{U_\alpha} - P_{U_\alpha} \Vert^2_{U^{min}_\alpha(v)\to \Hc_\alpha}}) \Vert v - \Pc_{U_\alpha} v\Vert^2.
 $$
\end{proposition}
\begin{proof}
We have $v \in U^{min}_\alpha(v)\otimes \Hc_{\alpha^c}$.
% and $v - \Pc_{U_\alpha} v \in U^\star_\alpha(v) \otimes \Hc_{\alpha^c}$. 
Noting that $\Vert \Pc^{W_\alpha}_{U_\alpha} \Vert_{U^{min}_\alpha(v)\otimes \Hc_{\alpha^c} \to \Hc} = \Vert P^{W_\alpha}_{U_\alpha} \Vert_{U^{min}_\alpha(v) \to \Hc_\alpha}$ and $\Vert \Pc^{W_\alpha}_{U_\alpha} - \Pc_{U_\alpha} \Vert_{U^{min}_\alpha(v)\otimes \Hc_{\alpha^c} \to \Hc} = \Vert P^{W_\alpha}_{U_\alpha} - P_{U_\alpha} \Vert_{U^{min}_\alpha(v) \to \Hc_\alpha}$, the results directly follow from Proposition \ref{prop:interpolation_error}.
\end{proof}
 Now, let $\alpha$ be a non-empty subset of $D$ and let $S(\alpha)$ be a partition of $\alpha$. Let $P_{U_\beta}^{W_\beta}$ be oblique projections onto subspaces $U_\beta$ of $\Hc_\beta$ along $W_\beta\subset \Hc_\beta^*$, $\beta \in S(\alpha)$. Then $  \bigotimes_{\beta \in S(\alpha)} P_{U_\beta}^{W_\beta} := P_{U_{S(\alpha)}}^{W_{S(\alpha)}}$ is the oblique projection from $\Hc_{S(\alpha)} =  \bigotimes_{\beta \in S(\alpha)} \Hc_\beta$ onto $  \bigotimes_{\beta \in S(\alpha)} U_\beta := U_{S(\alpha)}$ along $  \bigotimes_{\beta \in S(\alpha)} W_\beta := W_{S(\alpha)}$, and $\Pc^{W_{S(\alpha)}}_{U_{S(\alpha)}} = P_{U_{S(\alpha)}}^{W_{S(\alpha)}} \otimes id_{  \alpha^c}$ is the oblique projection from $\Hc_{\alpha} \otimes \Hc_{\alpha^c}$ to $U_{S(\alpha)} \otimes \Hc_{\alpha^c}$ along $W_{S(\alpha)}\otimes \Hc_{\alpha^c}'$. From Proposition \ref{prop:properties_projections}, we directly obtain the following result. 
 \begin{proposition}\label{oblique_projections_tensors_commute}
 If $  U_\alpha \subset   \bigotimes_{\beta \in S(\alpha)} U_\beta  $ and $W_\alpha \subset    \bigotimes_{\beta \in S(\alpha)} W_\beta$, then 
$$
 \Pc_{U_\alpha}^{W_\alpha} (\prod_{\beta \in S(\alpha)} \Pc_{U_\beta}^{W_\beta}) = (
\prod_{\beta \in S(\alpha)} \Pc_{U_\beta}^{W_\beta})\Pc_{U_\alpha}^{W_\alpha}  = \Pc_{U_\alpha}^{W_\alpha} .
$$
 %the  
 % oblique projection $\Pc_{U_\alpha}^{W_\alpha}$ from onto $U_\alpha \otimes \Hc_{\alpha^c}$  %along $W_\alpha \otimes \Hc_{\alpha^c}'$ is such that 
%$$
%P_U^W (\bigotimes_{\alpha \in S} P_{U_\alpha}^{W_\alpha}) = (\bigotimes_{\alpha \in S} P_{U_\alpha}^{W_\alpha}) P_U^W=P_U^W.
%$$
 \end{proposition}

\section{Tree-based tensor formats}\label{sec:tree-based-formats}

Let $T \subset 2^D\setminus \emptyset$ be a dimension partition tree over $D$, with root $D$. The elements of $T$ are called the nodes of the tree. Every node $\alpha \in T$ with $\#\alpha \ge 2$ has a set of sons $S(\alpha)$ which form a partition of $\alpha$, i.e. $\bigcup_{\beta \in S(\alpha)} \beta = \alpha$. A node $\alpha \in T$ with $\#\alpha =1$ is such that $S(\alpha)=\emptyset $ and is called a leaf of the tree. The set of leaves of $T$ is denoted $\Lc(T)$ (see an example on Figure \ref{fig:tree}). 
\begin{figure}[h] \centering \scriptsize
 \begin{tikzpicture}[scale=.5]
\tikzstyle{level 1}=[sibling distance=60mm]
\tikzstyle{level 2}=[sibling distance=20mm]
\tikzstyle{root}=[circle,draw,thick,fill=white]
\tikzstyle{interior}=[circle,draw,solid,thick,fill=gray!20]
\tikzstyle{leaves}=[circle,draw,solid,thick,fill=gray!20]
\tikzstyle{active}=[circle,draw,solid,thick,fill=red]
\tikzstyle{inactive}=[circle,draw,solid,thick,fill=white]

\node [root,label=above:{$\{1,2,3,4,5,6\}$}, 
fill=white]  {}
  child {node [interior,label=above left:{$\{1,2,3\}$},
  fill=white]  {}
    child { node [leaves,label=below:{$\{1\}$},
    fill=blue] {}} 
    child { node [interior,label=above right:{$\{2,3\}$},
    fill=white] {}
    child { node [leaves,label=below:{$\{2\}$},
    fill=blue] {}}
    child { node [leaves,label=below:{$\{3\}$},
    fill=blue] {}}   }      }
  child {node [interior,label=above right:{$\{4,5,6\}$},
  fill=white]  {}
   child { node [leaves,label=below:{$\{4\}$},
   fill=blue] {}}
    child { node [leaves,label=below:{$\{5\}$},
    fill=blue] {}}  
    child { node [leaves,label=below:{$\{6\}$},
    fill=blue] {}}  
   }
       ;
\end{tikzpicture}
\caption{A dimension partition tree $T$ over $D=\{1,2,3,4,5,6\}$ and its leaves (blue nodes).}\label{fig:tree}
\end{figure}
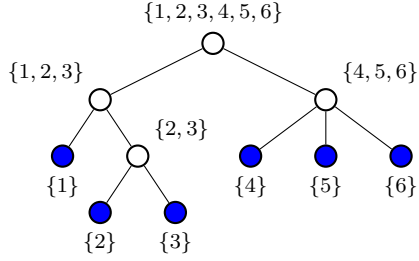
For $\alpha\in T$, we denote by $\level(\alpha)$ the level of $\alpha$ in $T$, such that $\level(D)=0$ and $\level(\beta) = \level(\alpha)+1$ if $\beta \in S(\alpha)$.
We let $L = \depth(T)= \max_{\alpha\in T} \level(\alpha)$ be the depth of $T$, which is the maximum level of the nodes in $T$, and 
 $T_{\ell} = \{\alpha\in T : \level(\alpha)=\ell\}$ be the subset of nodes with level $\ell$, $0\le \ell\le L$. 
We let
$
t_\ell = \bigcup_{\alpha \in T_{\ell}} \alpha.
$
We have $t_{\ell+1} \subset t_{\ell} $ and 
$
t_{\ell} \setminus t_{\ell+1} \subset \Lc(T)
$ (see example on Figure \ref{fig:tree-levels}).
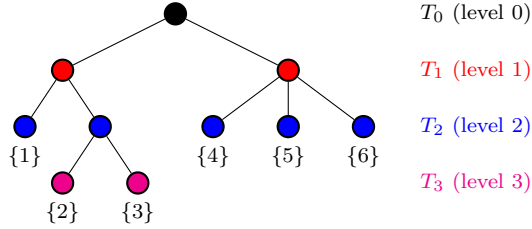
\begin{figure}[h]  \centering \scriptsize
  \begin{tikzpicture}[scale=.5]
\tikzstyle{level 1}=[sibling distance=60mm]
\tikzstyle{level 2}=[sibling distance=20mm]
\tikzstyle{root}=[circle,draw,thick,fill=white]
\tikzstyle{interior}=[circle,draw,solid,thick,fill=gray!20]
\tikzstyle{leaves}=[circle,draw,solid,thick,fill=gray!20]
\tikzstyle{active}=[circle,draw,solid,thick,fill=red]
\tikzstyle{inactive}=[circle,draw,solid,thick,fill=white]

\node [root,%label=above:{$\{1,2,3,4,5,6\}$}, 
,fill=black]  {}
  child {node [interior,%label=above left:{$\{1,2,3\}$},
  fill=red]  {}
    child { node [leaves,label=below:{$\{1\}$},
    fill=blue] {}} 
    child { node [interior,%label=above right:{$\{2,3\}$},
    fill=blue] {}
    child { node [leaves,label=below:{$\{2\}$},
    fill=magenta] {}}
    child { node [leaves,label=below:{$\{3\}$},
    fill=magenta] {}}   }      }
  child {node [interior,%label=above right:{$\{4,5,6\}$},
  fill=red]  {}
   child { node [leaves,label=below:{$\{4\}$},
   fill=blue] {}}
    child { node [leaves,label=below:{$\{5\}$},
    fill=blue] {}}  
    child { node [leaves,label=below:{$\{6\}$},
    fill=blue] {}}  
   }
       ;
       \node [] at (8,0) {\textcolor{black}{$T_{0}$ (level $0$)}};
       \node [] at (8,-1.5) {\textcolor{red}{$T_{1}$ (level $1$)}};
       \node [] at (8,-3.) {\textcolor{blue}{$T_{2}$ (level $2$)}};
              \node [] at (8,-4.5) {\textcolor{magenta}{$T_{3}$ (level $3$)}};
\end{tikzpicture}

\caption{A dimension partition tree $T$ over $D=\{1,\hdots,6\}$ with depth $L=3$ and the corresponding subsets $T_{\ell}$, $0\le \ell \le L$. Here $t_3 = \{2,3\}$ and $t_2=t_1 =t_0= D$.}\label{fig:tree-levels}
\end{figure}

We introduce a subset of active nodes $A \subset T \setminus \{D\}$ such that
$T \setminus A \subset \{D\}\cup  \Lc(T) $, which means that the set of non active nodes in $T \setminus \{D\}$ is a subset of the leaves (see Figure \ref{fig:tree-active}). A set $A$ is admissible if for any $\alpha \in A$, the parent node of $\alpha$ is in $A\cup \{D\}$. 
We let $\Lc(A) = A \cap \Lc(T)$, $A_\ell = A \cap T_{\ell}$ for $1\le \ell \le L$, and $a_\ell = \cup_{\alpha \in A_\ell} \alpha$.  
We define the $A$-rank of a tensor $v\in \Hc$ as the tuple  $\rank_A(v) = \{\rank_\alpha(v)\}_{\alpha\in A}$. 

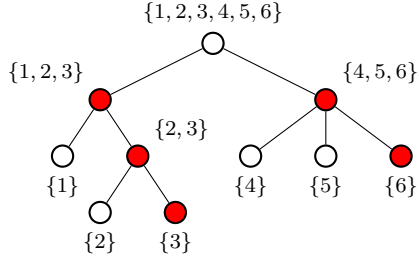
\begin{figure}[h] \centering \scriptsize
 \begin{tikzpicture}[scale=.5]
\tikzstyle{level 1}=[sibling distance=60mm]
\tikzstyle{level 2}=[sibling distance=20mm]
\tikzstyle{root}=[circle,draw,thick,fill=white]
\tikzstyle{interior}=[circle,draw,solid,thick,fill=gray!20]
\tikzstyle{leaves}=[circle,draw,solid,thick,fill=gray!20]
\tikzstyle{active}=[circle,draw,solid,thick,fill=red]
\tikzstyle{inactive}=[circle,draw,solid,thick,fill=white]

\node [root,label=above:{$\{1,2,3,4,5,6\}$}, 
fill=white]  {}
  child {node [interior,label=above left:{$\{1,2,3\}$},
  fill=red]  {}
    child { node [leaves,label=below:{$\{1\}$},
    fill=white] {}} 
    child { node [interior,label=above right:{$\{2,3\}$},
    fill=red] {}
    child { node [leaves,label=below:{$\{2\}$},
    fill=white] {}}
    child { node [leaves,label=below:{$\{3\}$},
    fill=red] {}}   }      }
  child {node [interior,label=above right:{$\{4,5,6\}$},
  fill=red]  {}
   child { node [leaves,label=below:{$\{4\}$},
   fill=white] {}}
    child { node [leaves,label=below:{$\{5\}$},
    fill=white] {}}  
    child { node [leaves,label=below:{$\{6\}$},
    fill=red] {}}  
   }
       ;
\end{tikzpicture}
\caption{A dimension partition tree $T$ over $D=\{1,2,3,4,5,6\}$ and an admissible subset of active nodes $A$ (red nodes).}\label{fig:tree-active}
\end{figure}

Now we consider a tensor $v\in \Hc$ with $\rank_A(v) =(r_\alpha)_{\alpha \in A}$. We let $r_D = \rank_D(v) =1$. 
For all $\alpha \in A \cup \{D\}$, we denote  by $\{v^\alpha_{k_\alpha}\}_{k_\alpha=1}^{r_\alpha}$ a basis of the minimal subspace $U^{min}_\alpha(v)\subset \Hc_\alpha$, and we let $v^D_1=v$. 
For $\alpha \in A \cup \{D\}$ such that $\emptyset \neq S(\alpha)\subset A$, since  
$U^{min}_\alpha(v) \subset \bigotimes_{\beta \in S(\alpha)} U^{min}_\beta(v)$, the tensor $v^{\alpha}_{k_\alpha}$  admits a representation 
$$
v^\alpha_{k_\alpha}(x_\alpha) = \sum_{\substack{1\le k_\beta \le r_\beta \\\beta \in S(\alpha)}} C^\alpha_{k_\alpha,(k_\beta)_{\beta\in S(\alpha)}} \prod_{\beta\in S(\alpha)} v^\beta_{k_\beta}(x_\beta),
$$
with a tensor of coefficients $C^\alpha \in \Rbb^{r_\alpha \times \bigtimes_{\beta\in S(\alpha)} r_\beta}$. 
For $\alpha \in A \cup \{D\}$ such that $\emptyset \neq S(\alpha)\not\subset A$, we have  
$U^{min}_\alpha(v) \subset (\bigotimes_{\beta \in S(\alpha)\cap A} U^{min}_\beta(v))\otimes (\bigotimes_{\beta \in S(\alpha)\setminus A}\Hc_\beta)$, and therefore the tensor $v^{\alpha}_{k_\alpha} $  admits a representation 
$$
v^\alpha_{k_\alpha}(x_\alpha) = \sum_{\substack{1\le k_\beta \le r_\beta \\\beta \in S(\alpha)\cap A}} C^\alpha_{k_\alpha,(k_\beta)_{\beta\in S(\alpha)\cap A}}((x_{\beta})_{\beta\in S(\alpha)\setminus A}) \prod_{\beta\in S(\alpha) \cap A} v^\beta_{k_\beta}(x_\beta),
$$
with $C^\alpha \in \Rbb^{r_\alpha \times \bigtimes_{\beta\in S(\alpha)\cap A} r_\beta} \otimes (\bigotimes_{\beta \in S(\alpha)\setminus A}\Hc_\beta).$
Finally, a tensor $v$ such that $\rank_A(v) = (r_\alpha)_{\alpha \in A}$ admits a representation
\begin{align}
v = \sum_{\substack{1\le k_\alpha \le r_\alpha \\ \alpha\in A \cup \{D\}  }} %C^D_{1,(k_\beta)_{\beta \in S(D)\cap A}}((x_{\beta})_{\beta\in S(D)\setminus A})
\prod_{\alpha \in (A \cup \{D\} )\setminus \Lc(A)} C^\alpha_{k_\alpha,(k_\beta)_{\beta \in S(\alpha)\cap A}}((x_{\beta})_{\beta\in S(\alpha)\setminus A})  \prod_{\alpha \in \Lc({A})} v^{\alpha}_{k_\alpha}(x_\alpha)\label{representation-TrA}
\end{align}
For a tuple $r=(r_\alpha)_{\alpha\in A}$, we define the subset $\Tc^A_r(\Hc)$ of tensors in $\Hc$ with $A$-rank bounded by $r$,
$$
\Tc^A_r(\Hc) = \{v \in \Hc : \rank_\alpha(v) \le r_\alpha, \alpha\in A\} = \bigcap_{\alpha \in A} \Tc^{\{\alpha\}}_{r_\alpha}(\Hc).
$$
{\begin{remark}
A tensor $v \in  \Tc^A_r(\Hc)$ admits a representation as a composition of functions. 
For $\alpha \in A$,  let
 $v^\alpha(x_\alpha) = (v_{1}^\alpha,\hdots,v_{r_\alpha}^\alpha) \in \Rbb^{r_\alpha}$. If $\emptyset \neq S(\alpha) \subset A$, the tensor $C^\alpha $ can be identified with a multilinear function $f^\alpha : \bigtimes_{\beta \in S(\alpha)}\Rbb^{r_\beta} \to \Rbb^{r_\alpha}$, and $v^\alpha(x_\alpha)$ admits the representation 
 $$
 v^\alpha(x_\alpha) = f^\alpha((v^{\beta}(x_\beta))_{\beta \in S(\alpha)}).
 $$
 For $\alpha \in A\cup \{D\}$ such that $\emptyset \neq S(\alpha)\not\subset A$, the tensor  $C^\alpha((x_{\beta})_{\beta\in S(\alpha)\setminus A})$ can be identified with a multilinear function $f^\alpha(\cdot , (x_{\beta})_{\beta\in S(\alpha)\setminus A}) :  \bigtimes_{\beta \in S(\alpha)\cap A}\Rbb^{r_\beta}  \to \Rbb^{r_\alpha}$, and 
 $v^\alpha(x_\alpha)$ admits the representation 
 $$
 v^\alpha(x_\alpha) = f^\alpha((v^{\beta}(x_\beta))_{\beta \in S(\alpha)\cap A} ,  (x_{\beta})_{\beta\in S(\alpha)\setminus A}),
 $$
 where the $f^\alpha$ is linear in the arguments associated with active nodes $\beta \in S(\alpha)\cap A$. 
As an example, for the case of Figure \ref{fig:tree-active}, the tensor $v$ admits the representation
 $$
 v(x)=f^{1,2,3,4,5,6}(f^{1,2,3}(x_1,f^{1,2}(x_2,v^{3}(x_3))),f^{4,5,6}(x_4,x_5,v^{6}(x_6))).
$$ 
 \end{remark}
}

\begin{proposition}\label{storage_complexity_TrA}
Let $V=V_1\otimes \hdots \otimes V_d \subset \Hc$, with $V_\nu$ a subspace of $\Hc_\nu$ with dimension  $\dim(V_\nu)=n_\nu$,  $1\le \nu\le d$. 
The storage complexity of a tensor in $\Tc_{r}^A(\Hc) \cap V = \Tc_r^A(V)$ 
is 
$$
\mathrm{storage}(\Tc^A_r(V)) = \sum_{\alpha \in (A\cup \{D\})\setminus \Lc(A)} r_\alpha \prod_{\beta \in S(\alpha)\cap A} r_\beta \prod_{\beta\in S(\alpha)\setminus A} n_\beta + \sum_{\alpha \in \Lc(A)} r_\alpha n_\alpha.
$$
\end{proposition}

\begin{example}[Tucker format] \label{ex:tucker}
The Tucker format corresponds to a trivial tree $T = \{\{1,\hdots,d\} ,\{1\},\hdots,\{d\}\}$ with depth $L=1$, and 
$A = T\setminus  \{D\}$ (see Figure \ref{fig:tree-tucker}). A tensor $v$ with $A$-rank bounded by $(r_1,\hdots,r_d)$ admits a representation of the form
\begin{align}
v(x) = \sum_{k_1=1}^{r_1}\hdots \sum_{k_d=1}^{r_d} C_{k_1,\hdots,k_d} v^{1}_{k_1}(x_1) \hdots  v^d_{k_d}(x_d),\label{repres-tucker}
\end{align}
where $C \in \Rbb^{r_1\times \hdots \times r_d}$, and $v^\nu_{k_\nu} \in \Hc_\nu$, $1\le \nu \le d$, {or equivalently
$$
v(x) = f^{1,...,d}(v^1(x_1),\hdots,v^d(x_d)).
$$}
   \begin{figure}[h]\centering
  \scriptsize
 \begin{tikzpicture}[scale=0.5]
\tikzstyle{level 1}=[sibling distance=20mm]
\tikzstyle{root}=[circle,draw,thick,fill=white]
\tikzstyle{interior}=[circle,draw,solid,thick,fill=red]
\tikzstyle{leaves}=[circle,draw,solid,thick,fill=red]
\tikzstyle{active}=[circle,draw,solid,thick,fill=red]

\node [root,label=above:{$\{1,2,3,4,5\}$}]  {}
  child {node [interior,label=below:{$\{1\}$}]  {}}
  child {node [interior,label=below:{$\{2\}$}]  {}}
  child {node [interior,label=below:{$\{3\}$}]  {}}
  child {node [interior,label=below:{$\{4\}$}]  {}}
  child {node [interior,label=below:{$\{5\}$}]  {}}       ;
\end{tikzpicture} 
\caption{Tucker format. Dimension partition tree $T$ over $D=\{1,\hdots,5\}$ and subset of active nodes $A$ (red nodes).}\label{fig:tree-tucker}
\end{figure}
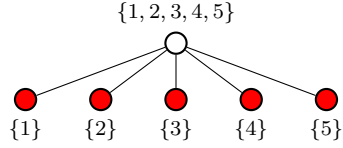
\end{example}
\begin{example}[Degenerate Tucker format]\label{ex:degenerate-tucker}
A degenerate Tucker format corresponds to a trivial tree $T = \{\{1,\hdots,d\} ,\{1\},\hdots,\{d\}\}$ with depth $L=1$, and 
an active set of nodes $A$ strictly included in $ T\setminus  \{D\}$. Up to a permutation of dimensions, this corresponds to  $A = \{\{1\},\hdots,\{p\}\}$, with $p<d$.  A tensor $v$ with $A$-rank bounded by $(r_1,\hdots,r_p)$ admits a representation of the form
\begin{align}
v(x) = \sum_{k_1=1}^{r_1}\hdots \sum_{k_p=1}^{r_p} C_{k_1,\hdots,k_p}(x_{p+1},\hdots,x_d) v^{1}_{k_1}(x_1) \hdots  v^{{p}}_{k_p}(x_p),\label{repres-tucker-degenerate}
\end{align}
where $C \in \Rbb^{r_1\times \hdots \times r_p} \otimes \Hc_{\{p+1,\hdots,d\}}$, and $v^\nu_{k_\nu} \in \Hc_\nu$, $1\le \nu \le p$, {or equivalently
$$
v(x) = f^{1,\hdots,d}(v^1(x_1),\hdots,v^{p}(x_p),x_{p+1},\hdots,x_d).
$$}
\end{example}
\begin{example}[Tensor train format]\label{ex:tt}
The tensor train (TT) format corresponds to a linear tree $T=\{\{1\},\{2\},\hdots,\{d\},\{1,2\},\hdots,\{1,\hdots,d\}\}$ and \\$A = \{\{1\},\{1,2\},\hdots,\{1,\hdots,d-1\}\}$ (see Figure \ref{fig:tree-tt}). Here, $A$ is a strict subset of $T\setminus \{D\}$. The nodes $\{2\},\hdots,\{d\}$ in $T$ are not active\footnote{Note that since $\rank_{\{d\}}(v) = \rank_{\{1,\hdots,d-1\}}(v)$, adding the node $\{d\}$ in the set of active nodes  $A$ would yield an equivalent tensor format.}. 
 A tensor $v$ with $A$-rank bounded by $(r_1,\hdots,r_{d-1})$ admits a representation of the form
$$
v(x) = \sum_{k_1=1}^{r_1}\hdots \sum_{k_{d-1}=1}^{r_{d-1}} v^{1}_{k_1}(x_1) C^{2}_{k_1,k_2}(x_2)  \hdots  C^{d-1}_{k_{d-2},k_{d-1}}(x_{d-1}) C^d_{k_{d-1},1}(x_d),
$$
where $v^1\in \Rbb^{r_1}\otimes \Hc_1$, $C^{\nu} \in \Rbb^{r_{\nu-1} \times r_\nu}  \otimes \Hc_\nu$ for $2\le \nu \le d$, with the convention $r_d=1$.
Here $L = d-1$, and for $1\le \ell\le L$, $T_{\ell} = \{\{1,\hdots,d-\ell\},\{d-\ell+1\}\}$, $t_\ell = \{1,\hdots,d-\ell+1\}$, $A_\ell = \{\{1,\hdots,d-\ell\}\}$ and $a_\ell = \{1,\hdots,d-\ell\}$. {The tensor $v$ admits the equivalent representation
$$
v(x) = f^{1,\hdots,d}(f^{1,\hdots,d-1}(...f^{1,2}(v^1(x_1),x_2)...,x_{d-1}),x_d).
$$ }
\begin{figure}[h]\centering\scriptsize
\begin{tikzpicture}[scale=0.5]
\tikzstyle{level 1}=[sibling distance=20mm]
\tikzstyle{level 2}=[sibling distance=20mm]
\tikzstyle{root}=[circle,draw,thick,fill=white]
\tikzstyle{interior}=[circle,draw,solid,thick,fill=red]
\tikzstyle{leaves}=[circle,draw,solid,thick,fill=red]
\tikzstyle{active}=[circle,draw,solid,thick,fill=red]
\tikzstyle{inactive}=[circle,draw,solid,thick,fill=white]

\node [root,label=above:{$\{1,2,3,4,5\}$}]  {}
  child {node [interior,label=above left:{$\{1,2,3,4\}$}]  {}
    child { node [leaves,label=above left:{$\{1,2,3\}$}] {}
      child { node [leaves,label=above left:{$\{1,2\}$}] {}
        child { node [leaves,label=below:{$\{1\}$}] {}}
        child { node [inactive,label=below:{$\{2\}$}] {}}}
      child { node [inactive,label=below:{$\{3\}$}] {}}}
     child { node [inactive,label=below:{$\{4\}$}] {}}}
    child { node [inactive,label=below:{$\{5\}$}] {}}
       ;
\end{tikzpicture}
\caption{Tensor train format. Dimension partition tree $T$ over $D=\{1,\hdots,5\}$ and active nodes $A$ (red nodes).}\label{fig:tree-tt}
\end{figure}
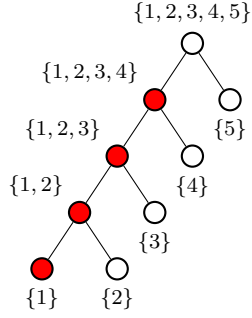
\end{example}

\begin{example}[Tensor train Tucker format]\label{ex:ttt}
The tensor train Tucker (TTT) format corresponds to a linear tree $T= \{\{1\},\hdots,\{d\},\{1,2\},\hdots,\{1,\hdots,d\}\}$ and $A = T\setminus \{D\}$ (see Figure \ref{fig:tree-ttt}). A tensor $v$ having a $A$-rank bounded by $(r_1,\hdots,r_{d},s_{2},\hdots,s_{d-1})$ admits a representation of the form \eqref{repres-tucker}
with a tensor $C \in \Rbb^{r_1\times \hdots \times r_d}$ such that 
$$
C_{k_1,\hdots,k_d} = \sum_{i_2=1}^{s_{2}}\hdots \sum_{i_{d-1}=1}^{s_{d-1}} C^{2}_{k_1,k_2,i_2} C^{3}_{i_2,k_3,i_3} \hdots   C^{d-1}_{i_{d-2},k_{d-1},i_{d-1}} C^d_{i_{d-1},k_{d},1},
$$
where $C^2 \in \Rbb^{r_1\times r_2 \times s_2}$ and $C^{k} \in \Rbb^{s_{k-1}\times r_k \times s_{k}}$ for $3\le k \le d$, with the convention $s_d=1$. Here $L = d-1$, $T_{\ell} = A_\ell = \{\{1,\hdots,d-\ell\},\{d-\ell+1\}\}$ and 
$t_\ell =a_\ell = \{1,\hdots,d-\ell+1\}$ for $1\le \ell\le L$. {The tensor $v$ admits the equivalent representation
$$
v(x) = f^{1,\hdots,d}(f^{1,\hdots,d-1}(...f^{1,2}(v^1(x_1),v^2(x_2))...,v^{d-1}(x_{d-1})),v^d(x_d)).
$$ }
\begin{figure}[h]\centering\scriptsize
\begin{tikzpicture}[scale=0.5]
\tikzstyle{level 1}=[sibling distance=20mm]
\tikzstyle{level 2}=[sibling distance=20mm]
\tikzstyle{root}=[circle,draw,thick,fill=white]
\tikzstyle{interior}=[circle,draw,solid,thick,fill=red]
\tikzstyle{leaves}=[circle,draw,solid,thick,fill=red]
\tikzstyle{active}=[circle,draw,solid,thick,fill=red]
\tikzstyle{inactive}=[circle,draw,solid,thick,fill=white]

\node [root,label=above:{$\{1,2,3,4,5\}$}]  {}
  child {node [interior,label=above left:{$\{1,2,3,4\}$}]  {}
    child { node [leaves,label=above left:{$\{1,2,3\}$}] {}
      child { node [leaves,label=above left:{$\{1,2\}$}] {}
        child { node [leaves,label=below:{$\{1\}$}] {}}
        child { node [leaves,label=below:{$\{2\}$}] {}}}
      child { node [leaves,label=below:{$\{3\}$}] {}}}
     child { node [leaves,label=below:{$\{4\}$}] {}}}
    child { node [leaves,label=below:{$\{5\}$}] {}}
       ;
\end{tikzpicture}
\caption{Tensor train Tucker format. Dimension partition tree $T$ over $D=\{1,\hdots,5\}$ and active nodes $A$ (red nodes).}\label{fig:tree-ttt}
\end{figure}
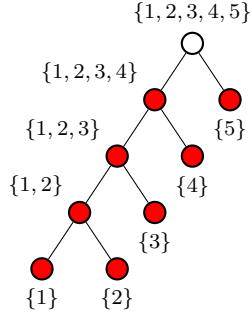
\end{example}

\section{Principal component analysis for tree-based tensor format}\label{sec:pca}

\subsection{Principal component analysis of multivariate functions}\label{sec:pca-order-2}

{Here we introduce the notion of principal component analysis for multivariate functions.} We consider a given non-empty and strict subset $\alpha $ of $D$. Any  tensor in $ \Hc$ is identified {(through the linear isometry $\Mc_\alpha$)} with its $\alpha$-matricisation in $ \Hc_\alpha \otimes \Hc_{\alpha^c}$. A tensor $u$  with $\alpha$-rank 
 $\rank_{\alpha}(u) \in \Nbb\cup \{+\infty\}$  admits a \emph{singular value decomposition} (see \cite[Section 4.4.3]{hackbusch2012book})
\begin{align}
u = \sum_{k=1}^{\rank_{\alpha}(u)} \sigma_k^\alpha u^\alpha_k \otimes u^{\alpha^c}_k,
\end{align}
where $\{u^\alpha_k\}_{k=1}^{\rank_\alpha(u)}$ and $\{u^{\alpha^c}_k\}_{k=1}^{\rank_\alpha(u)}$ are orthonormal vectors in $\Hc_\alpha$ and  $\Hc_{\alpha^c}$ respectively, and where the $\sigma_k^\alpha$ are the \emph{$\alpha$-singular values} of $u$ which are supposed to be arranged by decreasing values. The minimal subspace $U^{min}_\alpha(u)$ of $u$ is given by 
$$U^{min}_\alpha(u) = \overline{\linearspan\{u^\alpha_k\}_{k=1}^{\rank_\alpha(u)}}^{\Vert\cdot\Vert_{\Hc_\alpha}}.$$
For $r_\alpha < \rank_{\alpha}(u)$, the truncated singular value decomposition
$$
u_{r_\alpha} =  \sum_{k=1}^{r_\alpha} \sigma^\alpha_k u^\alpha_k \otimes u^{\alpha^c}_k,
$$
is such that 
$$
\Vert u - u_{r_\alpha} \Vert^2 = \min_{\rank_{\alpha}(v) \le r_\alpha} \Vert u- v \Vert^2 = \sum_{k=r_\alpha+1}^{\rank_\alpha({u})} (\sigma_k^\alpha)^2.
$$
The functions $\{u^\alpha_k\}_{k=1}^{r_\alpha}$ are the $r_\alpha$ principal components of  $u$ associated with dimensions $\alpha$, hereafter called \emph{$\alpha$-principal components}. The corresponding  subspace 
$ U_\alpha^\star = \linearspan\{u^\alpha_k\}_{k=1}^{r_\alpha}$, which is a subspace of $U^{min}_{\alpha}(u)$,  is hereafter called a \emph{$\alpha$-principal subspace} of dimension $r_\alpha$. Denoting $\Pc_{U_\alpha^\star} = P_{U_\alpha^\star}\otimes id_{\alpha^c}$ the orthogonal projection from $\Hc$ to $U_\alpha^\star \otimes \Hc_{\alpha^c}$, we have 
$
u_{r_\alpha} = \Pc_{U_\alpha^\star} u,
$\footnote{{For all $m \ge r_\alpha$, we have 
 $\Pc_{U_\alpha^\star} u_{m} = \sum_{k=1}^{m} \sigma_k^\alpha (P_{U_\alpha^\star} u_k^\alpha) \otimes u_k^{\alpha^c} =   \sum_{k=1}^{r_\alpha} \sigma_k^\alpha   u_k^\alpha \otimes u_k^{\alpha^c}  = u_{r_\alpha}$. Then using the continuity of $\Pc_{U_\alpha^\star}$ and taking the limit with $m$, we obtain $\Pc_{U_\alpha^\star} u = u_{r_\alpha}$.}}  
and 
\begin{align}
\Vert u - \Pc_{U_\alpha^\star} u  \Vert = \min_{\rank_{\alpha}(v) \le r_\alpha} \Vert u- v \Vert= \min_{\dim(U_\alpha)=r_\alpha} \Vert u - \Pc_{U_\alpha} u \Vert. \label{pca_best_subspace}
\end{align}
\begin{remark}
The optimization problem \eqref{pca_best_subspace} over subspaces of dimension $r_\alpha$ in $\Hc_\alpha$ admits a unique solution $U_\alpha^\star$ if and only if 
 $\sigma_{r_\alpha+1}^{\alpha}>\sigma_{r_\alpha}^\alpha$. 
\end{remark}

\subsection{Principal component analysis for tree-based tensor format}\label{sec:pca-tree-based-tensors}

Here, we propose and analyse 
an algorithm for the construction of an approximation $u^\star$ of a function $u$ in tree-based format $\Tc^A_r(\Hc)$. 
%We introduce approximation spaces $V_\nu \subset \Hc_\nu$, $1\le \nu \le d$, and $V = \bigotimes_{\nu=1}^d V_\nu \subset \Hc$. The proposed algorithm provides an approximation in $\Tc^T_r(V)$, the set of tensors in $V$ with $T$-rank bounded by $r$. 
It is based 
 on the construction of a hierarchy of subspaces $U_\alpha$, $\alpha\in A$, from principal component analyses {of approximations of $u$ in low-dimensional spaces in $ \Hc_\alpha$}. 
 This is a variant of the leaves-to-root higher-order singular value decomposition method proposed in \cite{grasedyck2010} (see also \cite[Section 11.4.2.3]{hackbusch2012book}).

% \subsection{Description of the algorithm}
{For each leaf node $\alpha \in \Lc(T)$, we introduce a finite-dimensional approximation space  
$V_\alpha \subset \Hc_\alpha$ with dimension $\dim(V_\alpha)=n_\alpha$, and we let $V = \bigotimes_{\alpha \in \Lc(T)} V_\alpha \subset \Hc$. For each non active node $\alpha \in \Lc(T)\setminus A$, we let $U_\alpha = V_\alpha$.
The algorithm then goes through all active nodes of the tree, going from the leaves to the root. 
 For each $\alpha \in A$, we let 
  $$
 u_\alpha = \Pc_{V_\alpha} u,
 $$ 
 where for $\alpha \notin \Lc(T)$, $V_\alpha$ is defined 
by
$$
V_\alpha = \bigotimes_{\beta\in S(\alpha)} U_\beta,
$$  
where the $U_\beta$, $\beta \in S(\alpha)$, have been determined at a previous step.
}
Then we determine 
the $r_\alpha$-dimensional  
$\alpha$-principal subspace $U_\alpha$ of $u_\alpha$, which is solution of 
   \begin{align}
 \Vert u_\alpha - \Pc_{U_\alpha} u_\alpha \Vert = \min_{\rank_{\alpha}(v) \le r_\alpha} \Vert u_\alpha - v \Vert.\label{optimal_Ualpha}
\end{align} 
% Then we define 
% $$
% u^{\ell} = \Pc_{T_\ell} u,
% $$
% where
% $$
% \Pc_{T_\ell} = \prod_{\alpha\in T_\ell} \Pc_{U_\alpha} 
% $$
%is the orthogonal projection from $\Hc$ onto $U_{T_\ell} \otimes \Hc_{t_\ell^c}$, with  
%$U_{T_\ell} = \bigotimes_{\alpha\in T_\ell} U_\alpha.$ 
 { Finally, we define 
 \begin{align}
 u^\star = \Pc_{V_D} u,
 \label{ustar_ideal}
\end{align}
where $\Pc_{V_D}$ is the orthogonal projection from $\Hc$ onto $V_D = \bigotimes_{\beta\in S(D)} U_\beta.$
}

 \subsection{Analysis of the algorithm}

\begin{lemma}\label{lem:nestesness}
For $\alpha \in \Lc(A)$, $U_\alpha \subset V_\alpha$. For $\alpha\in A \setminus \Lc(A)$, $$U_\alpha \subset \bigotimes_{\beta \in S(\alpha)} U_\beta.$$
\end{lemma}
\begin{proof}
For $\alpha \in A$, we have $U_\alpha \subset U^{min}_\alpha(u_\alpha)$. If $\alpha \in \Lc(A)$, we have 
$U^{min}_\alpha(u_\alpha) \subset V_\alpha$ since $u_\alpha = \Pc_{V_\alpha} u$. If $\alpha \in A \setminus \Lc(A)$, we have   $U^{min}_\alpha(u_\alpha) \subset \bigotimes_{\beta\in S(\alpha)} U^{min}_\beta(u_\alpha)$, and $U^{min}_\beta(u_\alpha) \subset U_\beta$ since $u_\alpha = \prod_{\beta \in S(\alpha)} \Pc_{U_\beta} u.$
\end{proof}

\begin{proposition}
The approximation $u^\star$   is an element of $\Tc_r^A(\Hc) \cap V = \Tc^A_r(V)$.
\end{proposition}
\begin{proof}
Since $u^\star = {\Pc_{V_D} u}$, we have $u^\star \in \bigotimes_{\alpha \in S(D)} U_\alpha $. Then using Lemma 
 \ref{lem:nestesness}, we prove by recursion that $u^\star \in \bigotimes_{\alpha \in \Lc(T)} V_\alpha = V$. Also, for any $\beta \in A$, Lemma \ref{lem:nestesness} implies that $u^\star \in U_{\beta}\otimes \Hc_{\beta^c}$. Therefore, $U^{min}_\beta(u^\star) \subset U_\beta$, and $\rank_{\beta}(u^\star) \le \dim(U_\beta)=r_\beta$. This proves that $u^\star \in \Tc_r^A(\Hc)$.
%We first note that for all $\alpha$ and $\beta$ in $T\setminus \{D\}$, we have either 
%$\beta \subset \alpha$ of $\beta \in D\setminus \alpha$. Therefore, from Proposition  \ref{prop:rank_bound}, we have that $\rank_{\alpha}(\Pc_{U_\beta}v) \le \rank_{\alpha}(v),$ for all $\alpha,\beta\in A$ and for all tensor $v \in \Hc$. 
%Then from Lemma \ref{lem:commute_Pl} and Proposition \ref{prop:rank_bound}, we have that for  $\alpha \in A_\ell$,  
%$\rank_{\alpha}(u^\star) = \rank_{\alpha}(\Pc_{T_1}\hdots \Pc_{T_{\ell}} u) \le \rank_\alpha(\Pc_{T_{\ell}} u) = 
%\rank_\alpha(\prod_{\beta\in T_\ell} \Pc_{U_\beta} u) \le \rank_\alpha(\Pc_{U_\alpha} u) \le r_\alpha $. Also, noting that $u^\star \in \bigotimes_{\alpha \in T_1} U_\alpha $, and using Lemma \ref{lem:nestesness},
% we easily prove  that $u^\star \in \bigotimes_{\alpha \in \Lc(T)} V_\alpha = V$. 
\end{proof}

{For any level $\ell$, $1 \le \ell\le L$, let $\Pc_{T_\ell} = \prod_{\alpha \in T_\ell} \Pc_{U_\alpha}$ be the orthogonal projection from $\Hc$ onto $U_{T_\ell} \otimes \Hc_{t_\ell^c}$, with $U_{T_\ell} = \bigotimes_{\alpha \in T_\ell} U_\alpha,$ and let 
$$
u^\ell = \Pc_{T_\ell} u^{\ell+1},
$$
with the convention $u^{L+1}=u.$}

\begin{lemma}\label{lem:commute_Pl}
For all $1\le \ell<\ell' \le L$, we have 
$$
\Pc_{T_{\ell'}}\Pc_{T_\ell} = \Pc_{T_\ell} =\Pc_{T_\ell}\Pc_{T_{\ell'}} .  
$$
\end{lemma}
\begin{proof}
For $1\le \ell<L$, we deduce from Lemma \ref{lem:nestesness}  that 
$$U_{T_\ell}  = \bigotimes_{\alpha\in T_\ell} U_\alpha \subset (\bigotimes_{\beta \in T_{\ell+1}} U_\beta) \otimes (\bigotimes_{\substack{\alpha \in T_\ell \\ S(\alpha) = \emptyset}} U_\alpha) \subset 
 U_{T_{\ell+1}} \otimes  \Hc_{t_{\ell}\setminus t_{\ell+1}},
$$
and then  $U_{T_\ell} \otimes \Hc_{t_\ell^c} \subset  U_{T_{\ell+1}} \otimes \Hc_{t_{\ell+1}^c}.$ 
 Therefore, for $1\le \ell < \ell' \le L$, we have 
$U_{T_\ell} \otimes \Hc_{t_\ell^c} \subset  U_{T_{\ell'}} \otimes \Hc_{t_{\ell'}^c},$
 and the result follows from  Proposition \ref{prop:properties_projections}.
\end{proof}
{From Lemma \ref{lem:commute_Pl}, we have that $$u^\ell = \Pc_{T_\ell} u^{\ell+1}  = \Pc_{T_\ell} \hdots \Pc_{T_L} u =  \Pc_{T_\ell} u,$$
for $1\le \ell \le L$, 
and 
$$
u^\star = \Pc_{T_1} u = u^1.
$$}

We now state the two main results about the proposed algorithm. 
			
 \begin{theorem}
 For a given $r$, the approximation $u^\star \in \Tc_r^A(\Hc)\cap V$   satisfies 
 $$\Vert u - u^\star \Vert^2 \le {\#A} \min_{v\in \Tc^A_r(\Hc)} \Vert u-v \Vert^2 + \sum_{\alpha \in \Lc(T)} \Vert u - \Pc_{V_\alpha} u \Vert^2.$$
\end{theorem}
	\begin{proof}
We first note that for all $1\le \ell < \ell' \le L$, $u^{\ell}-u^{\ell+1}$ is orthogonal to $u^{\ell'}-u^{\ell'+1}$. Indeed, using Lemma \ref{lem:commute_Pl}, we obtain that
\begin{align*}
(u^{\ell}-u^{\ell+1} , u^{\ell'} - u^{\ell'+1}) &= (u^{\ell}-u^{\ell+1}, \Pc_{T_{\ell'}}u^{\ell'+1} - \Pc_{T_{\ell'+1}} u^{\ell'+1}) \\
&= (\Pc_{T_{\ell'}}(u^{\ell}-u^{\ell+1}) , \Pc_{T_{\ell'+1}} (\Pc_{T_{\ell'}}u^{\ell'+1} -  u^{\ell'+1}))\\
&= (\Pc_{T_{\ell'+1}} \Pc_{T_{\ell'}}(u^{\ell}-u^{\ell+1}) ,  \Pc_{T_{\ell'}}u^{\ell'+1} -  u^{\ell'+1})\\
&=(\Pc_{T_{\ell'}}(u^{\ell}-u^{\ell+1}) , ( \Pc_{T_{\ell'}}-id)u^{\ell'+1}) = 0.
\end{align*}
%$u^{\ell'}-u^{\ell'+1} = (\Pc_{T_{\ell'}}-\id) u^{\ell'+1},$ and from 
%Lemma \ref{lem:commute_Pl}, we have $u^{\ell} - u^{\ell+1} = \Pc_{T_{\ell+1}} (u^{\ell} - u^{\ell+1}) =   \Pc_{T_{\ell'}} \Pc_{T_{\ell+1}} (u^{\ell} - u^{\ell+1})  = \Pc_{T_{\ell'}} (u^{\ell} - u^{\ell+1})$.  
Then, we have  	
 \begin{align*}
 \Vert u - u^\star \Vert^2 &= \sum_{\ell=1}^L \Vert u^{\ell+1} - u^{\ell}  \Vert^2 = \sum_{\ell=1}^L  \Vert u^{\ell+1} - \Pc_{T_\ell}  u^{\ell+1} \Vert^2\\
 & \le  \sum_{\ell=1}^L \sum_{\alpha\in T_\ell}  \Vert u^{\ell+1} - \Pc_{U_\alpha}  u^{\ell+1} \Vert^2. 
 \end{align*}
From Lemma \ref{lem:commute_Pl},  we know that $u^{\ell+1} = \Pc_{T_{\ell+1}} u $, where we use the convention $\Pc_{T_{L+1}} = id$.
For $\alpha \in   \Lc(T_\ell)$, since $\Pc_{U_\alpha}$ and $\Pc_{T_{\ell+1}}$ commute, $\Vert u^{\ell+1} - \Pc_{U_\alpha}  u^{\ell+1} \Vert =  \Vert \Pc_{T_{\ell+1}} u - \Pc_{U_\alpha}  \Pc_{T_{\ell+1}} u \Vert = \Vert  \Pc_{T_{\ell+1}} (u- \Pc_{U_\alpha} u) \Vert \le\Vert  u- \Pc_{U_\alpha} u \Vert .$ Therefore, for $\alpha \in \Lc(T_\ell)\setminus \Lc(A)$, 
we have  
 $\Vert u^{\ell+1} - \Pc_{U_\alpha}  u^{\ell+1} \Vert \le \Vert  u- \Pc_{V_\alpha} u \Vert $, 
 and for 
 $\alpha \in  \Lc(A_\ell)$, 
 we have  
 $\Vert u^{\ell+1} - \Pc_{U_\alpha}  u^{\ell+1} \Vert^2 \le \Vert  u- \Pc_{V_\alpha} u \Vert^2 +  \Vert  \Pc_{V_\alpha}  u- \Pc_{U_\alpha} u \Vert^2   = \Vert  u- \Pc_{V_\alpha} u \Vert^2  + \Vert u_\alpha - \Pc_{U_\alpha} u_\alpha \Vert^2 $. 
 For $\alpha \in A_\ell\setminus \Lc(A)$,  we have 
    $$u^{\ell+1} = \Pc_{T_{\ell+1}} u= \prod_{\delta\in T_{\ell+1}  \setminus S(\alpha)}  \Pc_{U_\delta}  \prod_{\beta \in S(\alpha)} \Pc_{U_\beta}u = \prod_{\delta\in T_{\ell+1}  \setminus S(\alpha)}\Pc_{U_\delta}  u_\alpha,$$  
   so that 
   $$\Vert u^{\ell+1} - \Pc_{U_\alpha}  u^{\ell+1} \Vert = \Vert \prod_{\delta\in T_{\ell+1} \setminus S(\alpha)} \Pc_{U_\delta}  (u_\alpha- \Pc_{U_\alpha} u_\alpha) \Vert \le  \Vert u_\alpha - \Pc_{U_\alpha} u_\alpha \Vert.$$ 
   Gathering the above results, we obtain
\begin{align}
\Vert u - u^\star \Vert^2 =  \sum_{\alpha\in A } 
  \Vert  u_\alpha- \Pc_{U_\alpha}  u_\alpha \Vert^2+ \sum_{\alpha \in \Lc(T)} \Vert u - \Pc_{V_\alpha} u\Vert^2  .\label{eq:u-ustar}
 \end{align}
For $\alpha \in A$, we let $ U_\alpha^\star$ be the subspace in $\Hc_{\alpha} $ such that
$$\Vert u - \Pc_{U_\alpha^\star} u \Vert = \min_{\rank_\alpha(v)\le r_\alpha} \Vert u - v \Vert \le \min_{\rank_A(v)\le r} \Vert u - v \Vert  .$$
For $\alpha \in \Lc(A)  $, we have $u_\alpha = \Pc_{V_\alpha} u$. From Proposition \ref{prop:rank_bound}, we know that $\rank_\alpha(\Pc_{V_\alpha}\Pc_{U_\alpha^\star} u ) \le \rank_\alpha(\Pc_{U_\alpha^\star} u)\le r_\alpha$. The optimality of $U_\alpha$ then implies that 
 $$\Vert  u_\alpha- \Pc_{U_\alpha}  u_\alpha \Vert \le  \Vert \Pc_{V_\alpha}u -  \Pc_{V_\alpha}\Pc_{U_\alpha^\star} u \Vert \le \Vert u - \Pc_{U_\alpha^\star} u \Vert.$$ Now consider $\alpha \notin A\setminus \Lc(A)$. 
We know that $\rank_\alpha(\prod_{\beta \in S(\alpha)} \Pc_{U_\beta} \Pc_{U_\alpha^\star}u) \le \rank_\alpha( \Pc_{U_\alpha^\star}u) \le r_\alpha$ from Proposition \ref{prop:rank_bound}.  The optimality of $U_\alpha$ then implies that  
\begin{align*}
 \Vert  u_\alpha- \Pc_{U_\alpha}  u_\alpha \Vert  & \le \Vert u_\alpha - \prod_{\beta \in S(\alpha)} \Pc_{U_\beta} \Pc_{U_\alpha^\star}u \Vert =  \Vert  \prod_{\beta \in S(\alpha)} \Pc_{U_\beta} (u - \Pc_{U_\alpha^\star}u )\Vert
 \\
 & \le \Vert u - \Pc_{U_\alpha^\star}u \Vert.
 \end{align*}
Finally, we obtain 
$$
  \sum_{\alpha\in A } 
  \Vert  u_\alpha- \Pc_{U_\alpha}  u_\alpha \Vert^2 \le   \sum_{\alpha\in A} \min_{\rank_\alpha(v)\le r_\alpha } \Vert u - v \Vert^2 \le \#A \min_{\rank_A(v)} \Vert u - v \Vert^2,
$$
which ends the proof. 
\end{proof}
 \begin{theorem}
For any $\epsilon \ge 0$, if for all $\alpha \in A$,  the rank $r_\alpha$ is chosen such that 
 $$
 \Vert u_\alpha - \Pc_{U_\alpha} u_\alpha \Vert \le \frac{\epsilon}{\sqrt{\#A}} \Vert u_\alpha \Vert,
 $$
 the approximation $u^\star$ satisfies
 $$
 \Vert u - u^\star \Vert^2 \le  \sum_{\alpha \in \Lc(T)} \Vert u - \Pc_{V_\alpha} u \Vert^2 + \epsilon^2 \Vert u \Vert^2.
 $$
 \end{theorem}
 \begin{proof}
Starting from \eqref{eq:u-ustar}, we obtain 
  \begin{align*}
  \Vert u - u^\star \Vert^2 &\le  \sum_{\alpha \in \Lc(T)} \Vert u - \Pc_{V_\alpha} u \Vert^2 +  \sum_{\alpha\in A} \frac{\epsilon^2}{{\#A}} \Vert u_\alpha \Vert^2 ,
 \end{align*}
and the result follows from 
$\Vert u_\alpha \Vert = \Vert \prod_{\beta \in S(\alpha)}  \Pc_{U_\alpha}u\Vert \le \Vert u \Vert$ if $\alpha \notin A\setminus \Lc(A)$, and $\Vert u_\alpha \Vert = \Vert  \Pc_{V_\alpha} u\Vert \le \Vert u \Vert$ if $\alpha \in \Lc(A)$.
 \end{proof}
 	
\section{Empirical principal component analysis for tree-based tensor format}\label{sec:empirical-pca}
 
 \subsection{Empirical principal component analysis of multivariate functions} 
{Here we present the empirical principal component analysis for the statistical estimation of $\alpha$-principal subspaces of a multivariate function (see Section \ref{sec:pca-order-2}).}  We  consider that $\Hc = L^2_{\mu}(\Xc)$ or that $\Hc$ is a separable reproducing kernel Hilbert space compactly embedded in $L^2_{\mu}(\Xc)${, equipped with the natural norm  in $L^2_{\mu}(\Xc)$}. 
Let $ (X_\alpha,X_{\alpha^c})$ be the random vector with values in  $\Xc_\alpha \times \Xc_{\alpha^c}$ with probability law $ \mu_\alpha \otimes \mu_{\alpha^c}$. The tensor $u$ can be identified with a random variable defined on $\Xc_{\alpha^c}$ with values 
in $\Hc_{\alpha}$ which associates to $x_{\alpha^c} \in \Xc_{\alpha^c}$ the function $u(\cdot,x_{\alpha^c}) = (id_\alpha \otimes \delta_{x_{\alpha^c}}) u,$ {this random variable being an element of the Bochner space $L^2_{\mu_{\alpha^c}}(\Xc_{\alpha^c};\Hc_\alpha)$}. 
{Then problem  \eqref{pca_best_subspace} is equivalent to find a $r_\alpha$-dimensional subspace in $\Hc_\alpha$ solution of 
\begin{equation}
%\Ebb\left( \Vert u(\cdot, X_{\alpha^c}) - P_{U_\alpha^\star} u(\cdot, X_{\alpha^c}) \Vert_{\Hc_{\alpha}} \right) = 
\min_{\dim(U_\alpha) = r_\alpha}  \Ebb\left( \Vert u(\cdot, X_{\alpha^c}) - P_{U_\alpha} u(\cdot, X_{\alpha^c}) \Vert_{\Hc_{\alpha}}^2 \right).\label{PCA_expectation}
\end{equation}}
Given a set $\{x_{\alpha^c}^k\}_{k=1}^{m_\alpha}$ of  $m_\alpha$ samples of  $X_{\alpha^c}$, the $\alpha$-principal subspace can be estimated by 
an \emph{empirical $\alpha$-principal subspace}   $\widehat U_\alpha$  solution of 
 \begin{align}
 \Vert u - \Pc_{\widehat U_\alpha} u \Vert_{\alpha,m_\alpha} = { \min_{ \dim(U_\alpha) = r_\alpha} \Vert u - \Pc_{ U_\alpha} u\Vert_{\alpha,m_\alpha}},\label{optimal_Ualpha}
\end{align} 
where 
$$
{\Vert  u - \Pc_{ U_\alpha} u \Vert_{\alpha,m_\alpha}^2 = \frac{1}{m_\alpha} \sum_{k=1}^{m_\alpha}
\Vert  u(\cdot,x_{\alpha^c}^k) - P_{U_\alpha} u  (\cdot,x_{\alpha^c}^k) \Vert_{\Hc_\alpha}^2.}
$$
 The problem is equivalent to finding the  $r_\alpha$ left principal components of  $\{u(\cdot,x_{\alpha^c}^k)\}_{k=1}^{m_\alpha} $, which is identified with an order-two tensor in $\Hc_\alpha \otimes \Rbb^{m_\alpha}$. We note that the number of samples $m_\alpha$ must be such that $m_\alpha\ge r_\alpha$ in order to estimate $r_\alpha$ principal components.  {In the case of i.i.d. samples, the semi-norm $\Vert  \cdot \Vert_{\alpha,m_\alpha}$ on $\Hc$ is the natural statistical estimation of the Bochner norm $\Vert \cdot\Vert_\alpha$ in $L^2_{\mu_{\alpha^c}}(\Xc_{\alpha^c};\Hc_\alpha)$, defined by  
 $\Vert v \Vert_\alpha^2 = \Ebb(\Vert v(X_{\alpha^c}) \Vert_{\Hc_\alpha})$. This norm $\Vert \cdot\Vert_\alpha$  coincides with the norm $\Vert \cdot \Vert$ on $\Hc$ when $\Hc $ is equipped with the $L^2_{\mu}(\Xc)$-norm.\footnote{{Note that when $\Hc$ is equipped with a norm stronger than the norm in $L^2_{\mu}(\Xc)$, then $\Vert \cdot \Vert_\alpha$ does not coincides with the norm $\Vert \cdot \Vert$ on $\Hc$, so that the subspaces solutions of \eqref{pca_best_subspace} and \eqref{PCA_expectation} will be different in general.}}
}

For some results on the comparison between $\Vert u - \Pc_{\widehat U_\alpha} u \Vert$ and the best approximation error $\Vert u - \Pc_{ U_\alpha^\star} u \Vert,$ see {
 \cite{blanchard2007statistical,reiss2016non,2018arXiv180303868J,2018arXiv180202869J}}.  Under suitable assumptions on $u$ (e.g., $u$ uniformly bounded), for any $\eta >0$ and $\epsilon>0$, there exists a $m_\alpha$ sufficiently large (depending on $\eta,\epsilon,r_\alpha$ and $u$) such that 
 $$
 \Vert u - \Pc_{\widehat U_\alpha} u \Vert^2 \le \Vert u - \Pc_{ U_\alpha^\star} u \Vert^2 + \epsilon^2
 $$
 holds with probability higher than $1-\eta$. Then, for any $\tau > 0$, there exists a $m_\alpha$ sufficiently large (depending on $\eta,\tau,r_\alpha$ and $u$) such that 
  $$
 \Vert u - \Pc_{\widehat U_\alpha} u \Vert^2 \le (1+\tau^2) \Vert u - \Pc_{ U_\alpha^\star} u \Vert^2
 $$
 holds with probability higher than $1-\eta$.

 \subsection{Empirical principal component analysis for tree-based format}
 
 Now we propose a  modification of the 
 algorithm proposed in Section \ref{sec:pca-tree-based-tensors} using only evaluations of the function $u$ at some selected points in $\Xc$. It is based 
 on the construction of a hierarchy of subspaces $\{U_\alpha\}_{\alpha\in A}$, from 
 empirical principal component analysis, and a corresponding hierarchy of commuting interpolation operators associated with nested sets of points.  
 
{For each leaf node $\alpha \in \Lc(T)$, we introduce a finite-dimensional approximation space  
$V_\alpha \subset \Hc_\alpha$ with dimension $\dim(V_\alpha)=n_\alpha$,  we introduce a set {$\Gamma_{V_\alpha}$} of  points in $\Xc_\alpha$ which is unisolvent for $V_\alpha$, we denote by $I_{V_\alpha}$ the corresponding interpolation operator from $\Hc_\alpha$ to $V_\alpha$, and we let 
$\Ic_{V_\alpha} = I_{V_\alpha} \otimes id_{\alpha^c}$ be the corresponding oblique 
projection from $\Hc$ to $V_\alpha \otimes \Hc_{\alpha^c}$. We let $V = \bigotimes_{\alpha \in \Lc(T)} V_\alpha \subset \Hc$. 
For each non active $\alpha \in \Lc(T)\setminus A$, we let $U_\alpha = V_\alpha$ and $\Gamma_{U_\alpha} = \Gamma_{V_\alpha}$. }

{The algorithm then goes through all active nodes of the tree, going from the leaves to the root. }

{For each active node $\alpha \in A$, we let
$$u_\alpha = \Ic_{V_\alpha} u$$
where for $\alpha\notin \Lc(A)$, the space $V_\alpha$ is defined by 
$$
V_\alpha = \bigotimes_{\beta\in S(\alpha)} U_\beta,
$$
where the $U_\beta$, $\beta \in S(\alpha)$, have been determined at a previous step. For $\alpha \notin \Lc(A)$, 
$\Ic_{V_\alpha}  =   I_{V_\alpha} \otimes id_{\alpha^c}$, where $I_{V_\alpha}$ is the interpolation operator onto $V_\alpha = \bigotimes_{\beta \in S(\alpha)} U_\beta$ associated with the product grid $\Gamma_{V_\alpha} = \bigtimes_{\beta\in S(\alpha)} \Gamma_{U_\beta}$, where each $ \Gamma_{V_\beta} $ have been determined at a previous step.}
Then we determine a {$r_\alpha$-dimensional empirical $\alpha$-principal} subspace $U_\alpha$ of $u_\alpha$, which is solution of 
 \begin{align}
 \Vert u_\alpha - \Pc_{U_\alpha} u_\alpha \Vert_{\alpha,m_\alpha} = \min_{\rank_{\alpha}(v) \le r_\alpha} \Vert u_\alpha - v \Vert_{\alpha,m_\alpha},\label{optimal_Ualpha}
\end{align} 
where 
$$
\Vert u_\alpha - v \Vert_{\alpha,m_\alpha}^2 = \frac{1}{m_\alpha} \sum_{k=1}^{m_\alpha}
\Vert u_{\alpha}(\cdot,x_{\alpha^c}^k) - v(\cdot,x_{\alpha^c}^k) \Vert_{\Hc_\alpha}^2,
$$
and where $\{x_{\alpha^c}^k\}_{k=1}^{m_\alpha}$ are $m_\alpha$ random samples of  $X_{\alpha^c}$, with $m_\alpha\ge r_\alpha$. The problem is equivalent to finding the $r_\alpha$ left principal components of  $\{u_\alpha(\cdot,x_{\alpha^c}^k)\}_{k=1}^{m_\alpha} $, which is identified with an order two tensor in {$V_\alpha \otimes \Rbb^{m_\alpha}$}.  The number of evaluations of the function $u$ for computing $U_\alpha$ is
$m_\alpha \times  {\dim(V_\alpha)}$. 
We let $\{\varphi^\alpha_k\}_{k=1}^{r_\alpha}$ be the set of principal components, such that 
 $U_\alpha = \linearspan\{\varphi^\alpha_k\}_{k=1}^{r_\alpha}$. We then construct a set of points 
 $\Gamma_{U_\alpha}$ which is unisolvent for $U_\alpha$, and such that 
  \begin{align}
\Gamma_{U_\alpha} \subset \Gamma_{V_\alpha}. \label{nestedgrids}
 \end{align}
 For the practical construction of the set $\Gamma_{U_\alpha}$, we use the procedure described in Section \ref{sec:magicpoints}. 
We denote by $I_{U_\alpha}$ the interpolation operator from $\Hc_\alpha$ onto $U_\alpha$ associated with the grid $\Gamma_{U_\alpha}$, and we let 
$\Ic_{U_\alpha} = I_{U_\alpha} \otimes id_{\alpha^c}$ be the corresponding   projection   from $\Hc$ onto $U_\alpha \otimes \Hc_{\alpha^c}$.
 
 {Finally, we compute 
 \begin{align}
 u^\star = \Ic_{V_D} u,
 \label{ustar_ideal}
\end{align}
where $\Ic_{V_D} = \bigotimes_{\beta \in S(D)} I_{U_\beta}$ is the interpolation operator from $\Hc$ onto $V_D = \bigotimes_{\beta\in S(D)} U_\beta,$ associated with the product grid $\Gamma_{V_D} =  \bigtimes_{\beta\in S(D)} \Gamma_{U_\beta}$.
}

\subsection{Analysis of the algorithm}  

{Let us first prove that the algorithm produces an approximation $u^\star$ in the desired tensor format. }

\begin{lemma}\label{lem:nestesness_interp}
For $\alpha \in \Lc(T)\setminus A$, $
U_\alpha = V_\alpha$. For $\alpha \in \Lc(A)$, $U_\alpha \subset V_\alpha$. For $\alpha\in A \setminus \Lc(A)$, $$U_\alpha \subset \bigotimes_{\beta \in S(\alpha)} U_\beta.$$
\end{lemma}
\begin{proof}
For $\alpha \in A$, we have $U_\alpha \subset U^{min}_\alpha(u_\alpha)$. If $\alpha \in \Lc(A)$, we have 
$U^{min}_\alpha(u_\alpha) \subset V_\alpha$ since $u_\alpha = \Ic_{V_\alpha} u$. If $\alpha \in A \setminus \Lc(A)$, we have   $U^{min}_\alpha(u_\alpha) \subset \bigotimes_{\beta\in S(\alpha)} U^{min}_\beta(u_\alpha)$, and $U^{min}_\beta(u_\alpha) \subset U_\beta$ since $u_\alpha = \prod_{\beta \in S(\alpha)} \Ic_{U_\beta} u.$
\end{proof}

\begin{proposition}\label{interpolation_ustar_in_TcrA}
The algorithm produces an approximation $$u^\star \in \Tc_r^A(\Hc) \cap V = \Tc^A_r(V).$$
\end{proposition}
\begin{proof}
Since $u^\star = { \Ic_{V_D} u}$, we have $u^\star \in {V_D =}  \bigotimes_{\alpha \in S(D)} U_\alpha $. Then using Lemma 
 \ref{lem:nestesness_interp}, we prove by recursion that $u^\star \in \bigotimes_{\alpha \in \Lc(T)} V_\alpha = V$. Also, for any $\alpha \in A$, Lemma \ref{lem:nestesness_interp} implies  that {$u^\star \in U_{\alpha} \otimes \Hc_{\alpha^c}$}. Therefore, $U^{min}_\alpha(u^\star) \subset U_\alpha$, and $\rank_{\alpha}(u^\star) \le \dim(U_\alpha)=r_\alpha$. This proves that $u^\star \in \Tc_r^A(\Hc)$.
 %
%We first note that for all $\alpha$ and $\beta$ in $T$, we have either 
%$\beta \subset \alpha$ or $\beta \in D\setminus \alpha$. Therefore, from Proposition  \ref{prop:rank_bound}, we have that $\rank_{\alpha}(\Ic_{U_\beta}v) \le \rank_{\alpha}(v),$ for all $\alpha,\beta$ and for all tensor $v$. 
%Then from Lemma \ref{lem:commute_Il} and Proposition \ref{prop:rank_bound}, we have that for  $\alpha \in A_\ell$,  
%$\rank_{\alpha}(u^\star) = \rank_{\alpha}(\Ic_{T_1}\hdots \Ic_{T_{\ell}} u) \le \rank_\alpha(\Ic_{T_{\ell}} u) = 
%\rank_\alpha(\prod_{\beta\in T_\ell} \Ic_{U_\beta} u) \le \rank_\alpha(\Ic_{U_\alpha} u) \le r_\alpha $. Also, noting that $u^\star \in \bigotimes_{\alpha \in T_1} U_\alpha $, and using Lemma \ref{lem:nestesness_interp},
% we prove that $u^\star \in \bigotimes_{\alpha \in \Lc(T)} V_\alpha = V$. 
\end{proof}

For all $\alpha \in T$, the operator $\Ic_{V_\alpha}= I_{V_\alpha} \otimes id_{\alpha^c}$ is a projection from $\Hc$ onto $V_\alpha \otimes \Hc_{\alpha^c}$ along $W_{\alpha}^\star\otimes \Hc^*_{\alpha^c},$ with $W_\alpha^\star = \linearspan\{\delta_x : x \in \Gamma_{V_\alpha}\}$. 
For all {$\alpha \in T\setminus \{D\}$}, the operator $\Ic_{U_\alpha}= I_{U_\alpha} \otimes id_{\alpha^c}$ is an oblique projection from $\Hc$ onto $U_\alpha \otimes \Hc_{\alpha^c}$ along $W_\alpha\otimes \Hc^*_{\alpha^c},$ with $W_\alpha = \linearspan\{\delta_x : x \in \Gamma_{U_\alpha}\}$. From the {property  \eqref{nestedgrids}} of the grids, we deduce the following result. 
\begin{lemma}\label{lem:nestesness_interp_oblique}
For $\alpha \in \Lc(T)\setminus A$, $
W_\alpha = W_{\alpha}^\star $. For $\alpha \in \Lc(A)$, $W_\alpha \subset W_\alpha^\star$. For $\alpha\in A \setminus \Lc(A)$, $$W_\alpha \subset W_{S(\alpha)} = \bigotimes_{\beta \in S(\alpha)} W_\beta.$$
\end{lemma}

\begin{remark} Note that interpolation operators $I_{U_\alpha}$, $\alpha\in A$, could be replaced by oblique projections $P^{W_\alpha}_{U_\alpha}$ onto $U_\alpha$ along subspaces $W_\alpha$ in $\Hc_\alpha^*$, with subspaces $W_\alpha$ satisfying  
for $\alpha \notin \Lc(T)$, $W_\alpha \subset \bigotimes_{\beta \in S(\alpha)} W_\beta$. 
Under this condition, all results of this section remain valid. 
\end{remark}
%, which is an oblique projection from $\Hc_{\alpha}$ onto $U_{S(\alpha)}$ along $W_{S(\alpha)} = \bigotimes_{\beta \in S(\alpha)} W_\beta = \linearspan\{\delta_x : x\in \Gamma^{S(\alpha)}\}$.  
% along $W_\alpha \otimes \Hc^*_{\alpha^c}$, with $W_\alpha = \linearspan\{\delta_x : x \in \Gamma^\alpha\}.$
 %which is the oblique projection from $\Hc_{t_\ell}$ to $U_{T_\ell}$ along $W_{T_\ell} = \linearspan\{\delta_{x} : x\in \Gamma^{T_\ell}\} = \bigotimes_{\alpha\in T_\ell} W_\alpha.$

{For any level $\ell$, $1\le \ell \le L$, let $$
  \Ic_{T_\ell} =  \prod_{\alpha \in T_\ell} \Ic_{U_\alpha} = I_{U_{T_\ell}}  \otimes id_{t_\ell^c},$$
  where $I_{U_{T_\ell}} = \bigotimes_{\alpha\in T_\ell} I_{U_\alpha}$ is the interpolation operator from $\Hc_{t_\ell} $ to $U_{T_\ell} = \bigotimes_{\alpha \in T_\ell} U_\alpha$ associated with the tensor product grid $\Gamma^{T_\ell} = \bigtimes_{\alpha\in T_\ell} \Gamma^\alpha,$ and let 
  $$u^\ell =     \Ic_{T_\ell}u^{\ell+1}, $$
  with the convention $u^{L+1} = u.$}
We then prove that operators $\Ic_{T_\ell}$, $1\le \ell \le L$, are commuting oblique projections.
%For $\alpha \in A\setminus \Lc(T)$, the operator $ \Ic_{S(\alpha)}  = \Ic_{U_{S(\alpha)}} = \prod_{\beta\in S(\alpha)} \Ic_{U_\beta}$ is an oblique projection  from $\Hc$ onto $U_{S(\alpha)}\otimes \Hc_{\alpha^c}$ along $W_{S(\alpha)} \otimes \Hc^*_{\alpha^c}$, with $W_{S(\alpha)} = \linearspan\{\delta_x : x\in \Gamma^{S(\alpha)}\}$. Since the grids satisfy the property \eqref{nestedgrids}, we have that $W_{\alpha} \subset W_{S(\alpha)} = \bigotimes_{\beta \in S(\alpha)} W_\beta.$
%The operator  $\Ic_{T_\ell}$
%is an oblique projection from $\Hc$ onto $ U_{T_\ell} \otimes \Hc_{t_\ell^c}$ along 
%$ W_{T_\ell} \otimes \Hc_{t_\ell^c}$, where $U_{T_\ell} = \bigotimes_{\alpha\in T_\ell} U_\alpha$ and $W_{T_\ell} = \bigotimes_{\alpha\in T_\ell} W_\alpha$. 
%

\begin{lemma}\label{lem:commute_Il}
For all $1\le \ell \le L$, the operator $\Ic_{T_\ell}$ is an oblique projection from $\Hc$ to $\Uc_{\ell} := U_{T_\ell} \otimes \Hc_{t_\ell^c}$ along $\Wc_{\ell} := W_{T_\ell} \otimes \Hc_{t_\ell^c}^*$. For all $1\le \ell < \ell' \le  L$, we have $\Uc_{\ell} \subset \Uc_{\ell'}$ and 
$\Wc_{\ell}\subset \Wc_{\ell'}$, and therefore
$$
\Ic_{T_\ell}\Ic_{T_{\ell'}} = \Ic_{T_\ell} = \Ic_{T_{\ell'}} \Ic_{T_\ell}.
$$
\end{lemma}
\begin{proof}
For  $1\le \ell <    L$, we have 
$$
\Uc_{\ell} =  \Big(\bigotimes_{\alpha \in T_\ell \setminus \Lc(T)}  U_\alpha \Big) \otimes \Big( 
\bigotimes_{\alpha \in T_\ell \cap \Lc(T)} U_\alpha \Big) \otimes \Hc_{t_\ell^c}\;  \text{and} \; \;   
  \Uc_{{\ell+1}} = \Big(\bigotimes_{\beta \in T_{\ell+1}} U_\beta \Big) \otimes \Hc_{t_{\ell+1}^c}.
$$
From Lemma  \ref{lem:nestesness_interp}, we know that $\bigotimes_{\alpha \in T_\ell \setminus \Lc(T)} U_\alpha $
 is a subspace of $\bigotimes_{\beta \in T_{\ell+1}} U_\beta \subset \Hc_{t_{\ell+1}}.$ Therefore, we obtain $\Uc_\ell \subset \Uc_{\ell+1}$. In the same way, using Lemma \ref{lem:nestesness_interp_oblique}, we obtain that $\Wc_\ell \subset \Wc_{\ell+1}$.
We then deduce $\Ic_{T_\ell} \Ic_{T_{\ell+1}} =\Ic_{T_{\ell+1}} \Ic_{T_\ell} = \Ic_{T_\ell}  $  from Proposition \ref{prop:properties_projections}, which ends the proof.
%For $1\le \ell <    L$, we have 
%$$\Ic_{T_\ell} = \Big(\bigotimes_{\alpha \in T_\ell \setminus \Lc(T)} I_{U_\alpha} \Big) \otimes \Big( 
%\bigotimes_{\alpha \in T_\ell \cap \Lc(T)} I_{U_\alpha}\Big) \otimes id_{t_\ell^c}, \quad \text{and} \quad 
%  \Ic_{T_{\ell+1}} = \Big(\bigotimes_{\beta \in T_{\ell+1}} I_{U_\beta} \Big) \otimes id_{t_{\ell+1}^c}.$$
% From Lemma  \ref{lem:nestesness_interp}, we know that $\bigotimes_{\alpha \in T_\ell \setminus \Lc(T)} U_\alpha $
% is a subspace of $\bigotimes_{\beta \in T_{\ell+1}} U_\beta \subset \Hc_{t_{\ell+1}}.$ From the nestedness property \eqref{nestedgrids} of the grids, we also know that the grid $\bigtimes_{\alpha \in T_\ell \setminus \Lc(T)} \Gamma^\alpha$
% is a subset of the grid   $\bigtimes_{\beta \in T_{\ell+1}} \Gamma^\beta \subset \Xc_{t_{\ell+1}}.$  Therefore, from 
%Proposition  \ref{interpolations_nestedgrids_commute}, we deduce that 
%the interpolation operators $\bigotimes_{\alpha \in T_\ell \setminus \Lc(T)} I_{U_\alpha}$ and $\bigotimes_{\beta \in T_{\ell+1}} I_{U_\beta}$ commute and their product is equal to $\bigotimes_{\alpha \in T_\ell \setminus \Lc(T)} I_{U_\alpha}$.
%From the expressions of $\Ic_{T_\ell} $ and $\Ic_{T_{\ell+1}}$, we easily deduce that $\Ic_{T_\ell} \Ic_{T_{\ell+1}} =\Ic_{T_{\ell+1}} \Ic_{T_\ell} = \Ic_{T_\ell}  $, which ends the proof.
\end{proof}

 \begin{lemma}\label{lem:bound_non_orthogonal}
 The approximation $u^\star $ satisfies 
 $$\Vert u - u^\star \Vert^2 \le (1+\delta(L-1)) \sum_{\ell=1}^{L} \Vert u^{\ell+1} - u^\ell \Vert^2  ,$$
 where $\delta = \max_{\ell} \delta_{T_\ell}$ and 
$$
 \delta_{T_\ell} = \Vert I_{U_{T_\ell}}-P_{U_{T_\ell}} \Vert_{U^{min}_{T_\ell}(u^{\ell+1}) \to \Hc_{t_\ell}},$$
 with $t_\ell = \cup_{\alpha \in T_\ell} \alpha$. 
% with $U^{min}_{t_\ell}(u^{\ell+1}) \subset U_{T_{\ell+1}} \otimes U^{min}_{t_\ell\setminus t_{\ell+1}}(u)$.
If $u\in V$, then
$$\delta_{T_\ell}\le \delta_{A_\ell} := \Vert I_{U_{A_\ell}} - P_{U_{A_\ell}} \Vert_{U^{min}_{A_\ell}(u^{\ell+1})\to \Hc_{a_\ell}},$$ with 
$a_\ell = \cup_{\alpha\in A_\ell} \alpha.$
\end{lemma}
	\begin{proof}
Since $u - u^\star = \sum_{\ell=1}^L  (u^{\ell+1} -  u^{\ell})$,
we have 
 \begin{align*}
 \Vert u - u^\star \Vert^2 &= \sum_{\ell=1}^L  \Vert u^{\ell+1} -  u^{\ell} \Vert^2 + 2\sum_{\ell' < \ell}(u^{\ell+1} -  u^{\ell},u^{\ell'+1} -  u^{\ell'}). 
 \end{align*}
For $\ell'<\ell$, since $\Pc_{T_{\ell}}(u^{\ell'+1} - u^{\ell'})  = u^{\ell'+1} - u^{\ell'}$,
 we have 
  \begin{align*}
(u^{\ell+1} -u^{\ell},u^{\ell'+1} -   u^{\ell'})&=(u^{\ell+1} -u^{\ell},\Pc_{T_{\ell}}(u^{\ell'+1} -   u^{\ell'}))\\
 &=   (\Pc_{T_\ell} (u^{\ell+1} -  u^{\ell}) , u^{\ell'+1} -   u^{\ell'})\\
&= (\Pc_{T_\ell} u^{\ell+1} -  \Ic_{T_\ell} u^{\ell+1} , u^{\ell'+1} -   u^{\ell'})
\\
&= ( (\Pc_{T_{\ell}} -  \Ic_{T_{\ell}})(u^{\ell+1} - u^{\ell}), u^{\ell'+1} -  u^{\ell'})
\\
&\le \Vert (\Pc_{T_{\ell}} -  \Ic_{T_{\ell}})(u^{\ell+1} - u^\ell )\Vert \Vert u^{\ell'+1} - u^{\ell'} \Vert,
 \end{align*}
 where we have used the fact that $\Pc_{T_\ell}\Ic_{T_\ell} = \Ic_{T_\ell}$ and $(\Pc_{T_{\ell}} -  \Ic_{T_{\ell}})u^{\ell}=0.$
Since $\Pc_{T_\ell}- \Ic_{T_\ell}= (P_{U_{T_\ell}} - I_{U_{T_\ell}}) \otimes id_{t_{\ell}^c}$ and $u^{\ell+1}-u^\ell = u^{\ell+1}-\Ic_{T_\ell} u^{\ell+1} \subset U^{min}_{T_\ell}(u^{\ell+1})\otimes \Hc_{  t_\ell^c}$, we obtain from Proposition \ref{prop:oblique_projections_norms} that 
  \begin{align*}
\vert (u^{\ell+1} -   u^{\ell},u^{\ell'+1} -   u^{\ell'}) \vert &
 \le \delta_{T_\ell} \Vert  u^{\ell+1} - u^\ell \Vert \Vert u^{\ell'+1} - u^{\ell'} \Vert,
 \end{align*}
 for $\ell'<\ell$. 
 We deduce that  
 $$\Vert u- u^\star \Vert^2 \le \sum_{\ell,\ell'=1}^L B_{\ell,\ell'} \Vert u^{\ell+1} -   u^{\ell} \Vert \Vert u^{\ell'+1} - u^{\ell'} \Vert \le \rho(B) \sum_{\ell=1}^L \Vert u^{\ell+1} -   u^{\ell} \Vert^2,$$
 where the matrix $B\in \Rbb^{L\times L}$ is such that $B_{\ell,\ell}=1$ and $B_{\ell,\ell'} = \delta_{T_{\max\{\ell,\ell'\}}}$ if $\ell\neq \ell'$. Using the theorem of Gerschgorin, we have that $$\rho(B)\le 1 + \max_\ell \sum_{\ell'\neq \ell} B_{\ell,\ell'} = 1 + \max_\ell ((\ell-1)\delta_{T_\ell} + \sum_{\ell'>\ell} \delta_{T_{\ell'}}) \le 1+ \delta (L-1),$$
 with $\delta = \max_{\ell} \delta_{T_\ell}$.
 
 Finally, when $u\in V$, we have $U^{min}_\alpha(u^{\ell+1}) \subset U^{min}_\alpha(u)\subset V_\alpha$ for all $\alpha\in \Lc(T).$ Therefore, $I_{V_\alpha}v = P_{V_\alpha}v$ for all $v \in U^{min}_\alpha(u^{\ell+1}) $ and $\alpha \in \Lc(T)$, and 
\begin{align*} 
\delta_{T_\ell} 
&= \Vert I_{U_{A_\ell}}  \otimes  I_{V_{T_\ell\setminus A_\ell}} - P_{U_{A_\ell}} \otimes P_{V_{T_\ell\setminus A_\ell}} \Vert_{U^{min}_{T_\ell}(u^{\ell+1})\to \Hc_{t_\ell}}\\
&=\Vert (I_{U_{A_\ell}} - P_{U_{A_\ell}} ) \otimes P_{V_{T_\ell\setminus A_\ell}} \Vert_{U^{min}_{T_\ell}(u^{\ell+1})\to \Hc_{t_\ell}}\\
&= \Vert I_{U_{A_\ell}} - P_{U_{A_\ell}} \Vert_{U^{min}_{A_\ell}(u^{\ell+1}) \to \Hc_{a_\ell}} \Vert  P_{V_{T_\ell\setminus A_\ell}}   \Vert_{U^{min}_{T_\ell\setminus A_\ell}(u^{\ell+1}) \to \Hc_{t_\ell\setminus a_\ell}}
\\&\le \Vert I_{U_{A_\ell}} - P_{U_{A_\ell}} \Vert_{U^{min}_{A_\ell}(u^{\ell+1}) \to \Hc_{a_\ell}} = \delta_{A_\ell}.
\end{align*}
\end{proof}

\begin{lemma}\label{lem:norm_diff}
For $1\le \ell \le L$,
 \begin{align*}
  \Vert u^{\ell+1} - u^\ell \Vert^2 \le&   (1+\delta_{T_\ell}^2) \Big(  \sum_{\alpha \in A_\ell} \Lambda_{T_{\ell+1}\setminus S(\alpha)}^2 (1+a_\alpha) \Vert u_\alpha - \Pc_{U_\alpha} u_\alpha \Vert^2  
\\
&+ \sum_{\alpha \in T_\ell \cap \Lc(T)} \Lambda_{T_{\ell+1}}^2 (1+ 2 a_\alpha \delta_\alpha^2) \Vert u - \Pc_{V_\alpha} u \Vert^2 \Big),
   \end{align*}
     where for $S\subset T$, $$\Lambda_{S} = \prod_{\alpha \in S} \Lambda_\alpha(U_\alpha), %\lambda(U_\delta), \quad \text{with} \quad
 %\lambda(U_{\delta})=  \max_{v \in U^{min}_{\delta}(u)} \frac{\Vert I_{U_\delta} v \Vert_{\Hc_{\delta}} }{ \Vert v \Vert_{\Hc_{\delta}}} =
\quad \Lambda_\alpha(U_\alpha) =   \Vert I_{U_\alpha} \Vert_{U^{min}_\alpha(u)\to \Hc_\alpha},
  $$ 
\begin{align}
a_\alpha =\mathbf{1}_{\alpha\in \Lc(A)}\mathbf{1}_{\delta_\alpha\neq 0},\label{a_alpha}
\end{align} 
 and  
 \begin{align}\delta_\alpha =   \Vert I_{V_\alpha} - P_{V_\alpha} \Vert_{U^{min}_\alpha(u)\to \Hc_\alpha} \label{delta_alpha}
 \end{align}
 for $\alpha \in \Lc(T)$.
 Moreover, if $u\in V$, then  $\delta_\alpha = 0$ for all $\alpha \in \Lc(T)$, and $a_\alpha = 0$  for all $\alpha \in T$.
 
\end{lemma}

	\begin{proof}
For all $1 \le \ell \le L$, we have 
 \begin{align*}
  \Vert u^{\ell+1} -    u^{\ell}  \Vert^2&=\Vert u^{\ell+1} -  \Ic_{T_\ell}  u^{\ell+1}  \Vert^2 = \Vert u^{\ell+1} -  \Pc_{T_\ell}  u^{\ell+1}  \Vert^2 + \Vert \Ic_{T_\ell} u^{\ell+1}  - \Pc_{T_\ell} u^{\ell+1}  \Vert^2
 \\
 &=\Vert u^{\ell+1} -  \Pc_{T_\ell}  u^{\ell+1}  \Vert^2 + \Vert (\Ic_{T_\ell} - \Pc_{T_\ell}) (u^{\ell+1}  -  \Pc_{T_\ell} u^{\ell+1})  \Vert^2
 \\
 &\le (1+\delta_{T_\ell}^2) \Vert u^{\ell+1} -  \Pc_{T_\ell}  u^{\ell+1}  \Vert^2
\le (1+\delta_{T_\ell}^2)   \sum_{\alpha\in T_\ell}  \Vert u^{\ell+1} - \Pc_{U_\alpha}  u^{\ell+1} \Vert^2. 
 \end{align*}
For $\alpha\in T_\ell \setminus \Lc(T) = A_\ell \setminus \Lc(T)$,  $$u^{\ell+1} = \Ic_{T_{\ell+1}} u= \prod_{\delta\in T_{\ell+1} \setminus S(\alpha)}  \Ic_{U_\delta}   \prod_{\beta \in  S(\alpha)}  \Ic_{U_\beta} u = \prod_{\delta\in T_{\ell+1} \setminus S(\alpha)}  \Ic_{U_\delta} u_\alpha, $$
 and since $\Pc_{U_\alpha}$ and $\prod_{\delta\in T_{\ell+1}\setminus S(\alpha)} \Ic_{U_\delta}$ commute, we have 
\begin{align*}
 \Vert u^{\ell+1} - \Pc_{U_\alpha}  u^{\ell+1} \Vert &=\Vert \prod_{\delta\in T_{\ell+1} \setminus S(\alpha)} \Ic_{U_\delta} (u_\alpha -  \Pc_{U_\alpha}  u_\alpha ) \Vert \le \Lambda_{T_{\ell+1} \setminus S(\alpha)}  \Vert  u_\alpha- \Pc_{U_\alpha}  u_\alpha \Vert.
  \end{align*}
Now for $\alpha \in T_\ell \cap \Lc(T)$, we have that $\Pc_{U_\alpha}$ and $\Ic_{T_{\ell+1}}$ commute, and therefore  
   \begin{align*}
 \Vert u^{\ell+1} -  \Pc_{U_\alpha}  u^{\ell+1}  \Vert & = \Vert \Ic_{T_{\ell+1}} (u- \Pc_{U_\alpha} u)
 \Vert \le \Lambda_{T_{\ell+1}} \Vert u- \Pc_{U_\alpha} u \Vert.
 \end{align*}
 If $\alpha \in T_\ell \setminus A_\ell$, we have $U_\alpha=V_\alpha$. If $\alpha \in A_\ell \cap \Lc(T)$, we have 
  \begin{align*}
 \Vert u- \Pc_{U_\alpha} u \Vert^2  =  \Vert u- \Pc_{U_\alpha}\Pc_{V_\alpha} u \Vert^2 &= 
 \Vert u- \Pc_{V_\alpha} u \Vert^2 + \Vert (id - \Pc_{U_\alpha}) \Pc_{V_\alpha} u\Vert^2,\end{align*}
 so that if  
 $\delta_\alpha = \Vert I_{V_\alpha} - P_{V_\alpha} \Vert_{U^{min}_\alpha(u)\to \Hc_\alpha} =0$, we have 
  $\Pc_{V_\alpha} u = \Ic_{V_\alpha} u = u_\alpha$ and 
  \begin{align}
 \Vert u- \Pc_{U_\alpha} u \Vert^2  \le   \Vert u- \Pc_{V_\alpha} u \Vert^2 + \Vert (id - \Pc_{U_\alpha}) u_\alpha\Vert^2,
\label{eq11}
 \end{align}
 and if $\delta_\alpha \neq 0$, we have 
  \begin{align}
  \Vert u- \Pc_{U_\alpha} u \Vert^2&\le 
   \Vert u- \Pc_{V_\alpha} u \Vert^2 + 2 \Vert (id - \Pc_{U_\alpha}) (\Pc_{V_\alpha} - \Ic_{V_\alpha}) u\Vert^2 + 2 
 \Vert (id - \Pc_{U_\alpha})   \Ic_{V_\alpha} u\Vert^2 \nonumber \\
  &=  \Vert u- \Pc_{V_\alpha} u \Vert^2 + 2 \Vert (id - \Pc_{U_\alpha}) (\Pc_{V_\alpha} - \Ic_{V_\alpha}) (u-\Pc_{V_\alpha}u)\Vert^2 \nonumber\\
  &\quad + 2 
 \Vert (id - \Pc_{U_\alpha}) u_\alpha\Vert^2\nonumber \\
  &\le (1 + 2 \delta_\alpha^2 )\Vert u -\Pc_{V_\alpha}u \Vert^2  + 2\Vert u_\alpha -  \Pc_{U_\alpha} u_\alpha\Vert^2,
  \label{eq12}
   \end{align}
   where we have used Proposition \ref{prop:oblique_projections_norms}.
   We conclude from \eqref{eq11} and \eqref{eq12} that if $\alpha \in A_\ell \cap \Lc(T)$, 
   \begin{align*}
  \Vert u- \Pc_{U_\alpha} u \Vert^2&\le (1 + 2 a_\alpha \delta_\alpha^2 )\Vert u -\Pc_{V_\alpha}u \Vert ^2 +(1+a_\alpha)\Vert u_\alpha -  \Pc_{U_\alpha} u_\alpha\Vert^2.
      \end{align*}
   Gathering the above results, we obtain 
   \begin{align*}
  & \Vert u^{\ell+1} - u^\ell \Vert^2 \le  (1+\delta_{T_\ell}^2) \Big(\sum_{\alpha \in A_\ell \setminus \Lc(T)}   \Lambda_{T_{\ell+1}\setminus S(\alpha)}^2 \Vert u_\alpha - \Pc_{U_\alpha} u_\alpha \Vert^2 \\
   &+ \sum_{\alpha \in A_\ell \cap \Lc(T)} (1+a_\alpha) \Lambda_{T_{\ell+1}}^2 \Vert u_\alpha - \Pc_{U_\alpha} u_\alpha \Vert^2 
   \\
   &+ \sum_{\alpha \in A_\ell \cap \Lc(T)} (1 + 2 a_\alpha \delta_\alpha^2 ) \Lambda_{T_{\ell+1}}^2 \Vert u - \Pc_{V_\alpha} u \Vert^2  + \sum_{\alpha\in T_\ell \setminus A_\ell} \Lambda_{T_{\ell+1}}^2 \Vert u - \Pc_{V_\alpha} u \Vert^2 \Big) ,
   \end{align*}
   which ends the proof.

\end{proof}

We now state the two main results about the proposed algorithm. 		
 \begin{theorem}\label{th:quasi-optim-ustar}
Assume that for all $\alpha \in A$, the subspace $U_\alpha$ is such that 
\begin{align}
\Vert u_\alpha - \Pc_{U_\alpha} u_\alpha \Vert^2 \le (1+ \tau^2) \min_{\rank_{\alpha}(v) \le r_\alpha} \Vert u_\alpha - v 
\Vert^2
\label{ass:quasi-opt-Ualpha}\end{align}
 holds with probability higher than $1-\eta$, for some $\tau \ge 1$. Then 
the approximation $u^\star \in \Tc_r^A(\Hc) \cap V$  is such that 
 \begin{align}
 \Vert u - u^\star \Vert^2 \le (1+\tau^2) C^2 \min_{v\in \Tc_r^A(\Hc)} \Vert u-v \Vert^2 +  
 \sum_{\alpha \in \Lc(T)} D_\alpha^2 \Vert u - \Pc_{V_\alpha} u \Vert^2
  \label{quasi-optim}\end{align}
holds with probability higher than $1-\#A \eta$, where $C$ is defined by 
   \begin{align}
 &  C^2 = (1+\delta(L-1)) \sum_{\ell=1}^{L} (1+\delta_{T_\ell}^2) \Lambda_{T_{\ell+1}}^2 \sum_{\alpha \in A_\ell}  (1+a_\alpha) \lambda_\alpha^2,   \label{constant_C}
 \end{align}
  with \begin{align}
  \lambda_\alpha =  \mathbf{1}_{\alpha \notin \Lc(A)}+ \mathbf{1}_{\alpha\in \Lc(A)} \Vert I_{V_\alpha} \Vert_{U^{min}_\alpha(u) \to \Hc_\alpha }
  \label{lambda_alpha}\end{align}
and $a_\alpha$ and $\delta_\alpha$ defined by \eqref{a_alpha} and \eqref{delta_alpha} respectively, 
and where $D_\alpha$  is defined by 
 \begin{align}
&   D_\alpha^2 = (1+\delta(L-1)) (1+\delta_{T_\ell}^2) \Lambda_{T_{\ell+1}}^2    (1+2a_\alpha \delta_\alpha^2) 
  \label{constant_Dalpha}
 \end{align}
  for $\alpha\in \Lc(T) \cap T_\ell$.
\end{theorem}
\begin{proof}
For $\alpha\in A$, let $\widehat U_\alpha$ be a subspace  such that
 $$
  \Vert u_\alpha - \Pc_{\widehat U_\alpha} u_\alpha \Vert =\min_{\rank_{\alpha}(v) \le r_\alpha} \Vert u_\alpha - v \Vert,
 $$
and let $ U_\alpha^\star \subset U^{min}_\alpha(u) $ be a subspace such that
$$\Vert u - \Pc_{ U_\alpha^\star} u \Vert = \min_{\rank_\alpha(v)\le r_\alpha } \Vert u - v \Vert \le \min_{v\in \Tc_r^A(\Hc)} \Vert u - v \Vert  .$$
For $\alpha \in \Lc(A)$, we have $u_\alpha = \Ic_{V_\alpha} u$.  We know that
$\rank_\alpha(\Ic_{V_\alpha} \Pc_{ U_\alpha^\star} u)\le r_\alpha$ from Proposition \ref{prop:rank_bound}. By the optimality of $\widehat U_\alpha$, we obtain
$$\Vert  u_\alpha- \Pc_{\widehat U_\alpha}  u_\alpha \Vert \le \Vert u_\alpha - \Ic_{V_\alpha}  \Pc_{ U_\alpha^\star} u \Vert \le \Vert I_{V_\alpha} \Vert_{U^{min}_\alpha(u)\to \Hc_\alpha } \Vert u -    \Pc_{ U_\alpha^\star} u \Vert   .$$ 
Now consider $\alpha \in A \setminus \Lc(A)$. 
We know that   $\rank_\alpha(\prod_{\beta \in S(\alpha)} \Ic_{U_\beta} \Pc_{ U_\alpha^\star}u) \le \rank_\alpha( \Pc_{ U_\alpha^\star}u) \le r_\alpha$ from Proposition \ref{prop:rank_bound}. By the optimality of $\widehat U_\alpha$, we obtain 
\begin{align*}
 \Vert  u_\alpha- \Pc_{\widehat U_\alpha}  u_\alpha \Vert  & \le \Vert u_\alpha - \prod_{\beta \in S(\alpha)} \Ic_{U_\beta} \Pc_{ U_\alpha^\star}u \Vert =  \Vert  \prod_{\beta \in S(\alpha)} \Ic_{U_\beta} (u - \Pc_{ U_\alpha^\star}u )\Vert \\
 &\le \Lambda_{S(\alpha)}\Vert u - \Pc_{ U_\alpha^\star}u \Vert.
 \end{align*}
Then, using Lemma 
\ref{lem:norm_diff}  and assumption \eqref{ass:quasi-opt-Ualpha}, we obtain 
 \begin{align*}
  \Vert u^{\ell+1} - u^\ell \Vert^2 \le&   (1+\delta_{T_\ell}^2) \Lambda_{T_{\ell+1}}^2 \Big(  \sum_{\alpha \in A_\ell }  (1+a_\alpha) \lambda_\alpha^2 (1+ \tau^2) \min_{\rank_\alpha(v)\le r_\alpha} \Vert u - v \Vert^2  \\
  &+   \sum_{\alpha \in T_\ell \cap \Lc(T)}   (1+ 2a_\alpha\delta_\alpha^2) \Vert u - \Pc_{V_\alpha} u \Vert^2 \Big).
   \end{align*}
 Then, using Lemma \ref{lem:bound_non_orthogonal}, we obtain \eqref{quasi-optim}.
\end{proof}

 \begin{remark}
 Assume $u\in V$ (no discretization). Then $\delta_\alpha = 0$ and $\Vert I_{V_\alpha} \Vert_{U^{min}_\alpha(u)\to \Hc_\alpha}=1$ for all $\alpha \in \Lc(T)$, $a_\alpha = 0$ and $ \lambda_\alpha = 1$ for all $\alpha \in T$, $\Lambda_{T_{\ell}} = \Lambda_{A_\ell}$ and $\delta_{T_\ell} = \delta_{A_\ell}$ for all $\ell$. 
  Also, the constant 
 $C$  defined by \eqref{constant_C}  is such that 
  \begin{align}
 &  C^2 = (1+\delta(L-1)) \sum_{\ell=1}^{L} (1+\delta_{A_\ell}^2) \Lambda_{A_{\ell+1}}^2 \#A_\ell .\label{constant_C_nodiscretization}
 \end{align}
 Moreover, if $U_\alpha = U^{min}_\alpha(u)$ for all $\alpha$, then $\Lambda_{T_\ell} = \Lambda_{A_\ell}=1$ and $\delta_{T_\ell}=\delta_{A_\ell}=0$ for all $\ell$, which implies 
  \begin{align}
 &  C^2 = \#A .  \label{constant_C_nodiscretization_recovery}   
 \end{align}

\end{remark}

 \begin{theorem}\label{th:interpolation_relative_precision}
Let $\epsilon,{\tilde \epsilon} \ge 0$. Assume that  for all $\alpha \in A$,  the subspace $U_\alpha$ is such that 
 \begin{align}
 \Vert u_\alpha - \Pc_{U_\alpha} u_\alpha \Vert \le \epsilon \Vert u_\alpha \Vert \label{ass:precision_Ualpha}
\end{align}
 holds with   probability higher than $1- \eta$, and further assume that the subspaces $V_\alpha$, $\alpha \in \Lc(T)$, are such that 
 \begin{align}
 \Vert u - \Pc_{V_\alpha} u  \Vert \le \tilde \epsilon \Vert u \Vert. \label{ass:precision_Valpha}
\end{align}
 Then 
 the approximation $u^\star$ is such that
 $$
 \Vert u - u^\star \Vert^2 \le (C^2 \epsilon^2 + D^2 \tilde \epsilon^2) \Vert u \Vert^2 
 $$
 holds with  probability higher than $1-\#A \eta$, where $C$ is defined by \eqref{constant_C} and where $D^2 = \sum_{\alpha\in \Lc(T)} D_\alpha^2$,  with $D_\alpha$ defined by \eqref{constant_Dalpha}, is such that 
 \begin{align}
D^2 = (1+\delta(L-1)) \sum_{\ell=1}^{L} (1+\delta_{T_\ell}^2) \Lambda_{T_{\ell+1}}^2  \sum_{\alpha \in T_\ell \cap \Lc(T)}   (1+2a_\alpha \delta_\alpha^2)   \label{constant_D}.
 \end{align}
 \end{theorem}
 \begin{proof}
 We first note that for $\alpha \in A \setminus \Lc(A),$ we have 
 $
 \Vert u_\alpha \Vert \le \Lambda_{S(\alpha)} \Vert u \Vert.
 $ Also, for $\alpha \in \Lc(T)$, we have  $
 \Vert u_\alpha \Vert \le \lambda_\alpha \Vert u \Vert,
 $ with $\lambda_\alpha$ defined in \eqref{lambda_alpha}.
 Using Lemma 
\ref{lem:norm_diff} and assumptions \eqref{ass:precision_Ualpha} and \eqref{ass:precision_Valpha}, we then obtain  
 \begin{align*}
  \Vert u^{\ell+1} - u^\ell \Vert^2 \le   (1+\delta_{T_\ell}^2)  \Lambda_{T_{\ell+1}}^2 \Big(  \sum_{\alpha \in A_\ell}  (1+a_\alpha) \lambda_\alpha^2 \epsilon^2 \Vert u \Vert^2 &
\\
+ \sum_{\alpha \in T_\ell \cap \Lc(T)}  (1+2a_\alpha \delta_\alpha^2) \tilde \epsilon^2 \Vert u \Vert^2 \Big).&
   \end{align*}
Finally, we obtain the desired result by using 
Lemma \ref{lem:bound_non_orthogonal}.
 \end{proof}

 \begin{example}
 For the Tucker format described in Example \ref{ex:tucker}, the constants $C$ and $D$ are given by
 $$
 C^2 = (1+\delta_{T_1}^2) \sum_{\alpha \in \Lc(T)} (1+\mathbf{1}_{\delta_\alpha\neq 0}) \Vert I_{V_\alpha} \Vert_{U^{min}_\alpha(u)\to \Hc_\alpha}^2,$$
 $$ 
 D^2 = (1+\delta_{T_1}^2)\sum_{\alpha \in \Lc(T)} (1+ 2\delta_\alpha^2),
 $$
with 
 \begin{align*}
 \delta_{T_1} & =  { \Vert \bigotimes_{\alpha\in \Lc(T)} I_{U_\alpha} -  \bigotimes_{\alpha\in \Lc(T)} P_{U_\alpha}  \Vert_{U^{min}_D(u)\to \Hc}} =   { \Vert \bigotimes_{\alpha\in \Lc(T)} I_{U_\alpha} u-  \bigotimes_{\alpha\in \Lc(T)} P_{U_\alpha} u \Vert}/{\Vert u \Vert} %\le \Vert  \bigotimes_{\alpha\in \Lc(T)} I_{U_\alpha} \Vert_{\bigotimes_{\alpha\in \Lc(T)} U^{min}_\alpha(u)} 
 . 
 \end{align*}
If $u\in V$, then
$$
C  =  (1+ \delta_{T_1}^2)^{1/2} \sqrt{d}.
$$

 \end{example}
 
  \begin{example}
 For the tensor train format described in Example \ref{ex:tt}, 
 the constant $C$ and $D$ are given by
   \begin{align*}
   C^2 = (1+\delta(d-2)) \Big(&\sum_{\ell=1}^{d-2} (1+\delta_{T_\ell}^2) \Lambda_{\{1,\hdots,d-\ell-1\}}^2 \Vert I_{V_{d-\ell}} \Vert^2_{U^{min}_{d-\ell}(u)\to \Hc_{d-\ell}}  \\
   &+ (1+\delta_{T_{d-1}}^2)  (1+\mathbf{1}_{\delta_1\neq 0})\Vert I_{V_1} \Vert^2_{U^{min}_1(u)\to \Hc_1} \Big)  ,
 \end{align*}
 \begin{align*}
   D^2 = (1+\delta(d-2)) \Big(& \sum_{\ell=1}^{d-2} (1+\delta_{T_\ell}^2) \Lambda_{\{1,\hdots,d-\ell-1\}}^2\Vert I_{V_{d-\ell}} \Vert_{U^{min}_{d-\ell}(u)\to \Hc_{d-\ell}}^2  \\
   &+ (1+\delta_{T_{d-1}}^2)(2+2\delta_1^2)  \Big),
 \end{align*}
% where $\Lambda_{\{1,\hdots,d-\ell-1\}}=\Vert I_{U_{\{1,\hdots,d-\ell-1\}}}\Vert_{U^{min}_{\{1,\hdots,d-\ell-1\}}(u)\to \Hc_{\{1,\hdots,d-\ell-1\}}},$
%   
with 
 $$
 \delta_{T_\ell} = \Vert I_{U_{\{1,\hdots,d-\ell\}}}\otimes I_{V_{\{d-\ell+1\}}} -P_{U_{\{1,\hdots,d-\ell\}}}\otimes P_{V_{\{d-\ell+1\}}} \Vert_{U^{min}_{\{1,\hdots,d-\ell+1\}}(u^{\ell+1}) \to \Hc_{\{1,\hdots,d-\ell+1\}}}.
 $$
If $u \in V$, then 
$$
C^2 = (1+\delta(d-2)) \left(\sum_{\ell=1}^{d-2} (1+\delta_{T_\ell}^2) \Lambda_{\{1,\hdots,d-\ell-1\}}^2    + (1+\delta_{T_{d-1}}^2)   \right).
$$
 \end{example}
 \begin{example}
 For the tensor train Tucker format described in Example \ref{ex:ttt}, 
 the constant $C$ and $D$ are given by
   \begin{align*}
   &C^2 = (1+\delta(d-2)) \times \\ &\Big(\sum_{\ell=1}^{d-2} (1+\delta_{T_\ell}^2) \Lambda_{\{1,\hdots,d-\ell-1\}}^2 
   \Lambda_{\{d-\ell\}}^2\big(1+(1+\mathbf{1}_{\delta_{d-\ell+1}\neq 0})\Vert I_{V_{\{d-\ell+1\}}} \Vert_{U^{min}_{\{d-\ell+1\}} \to \Hc_{\{d-\ell+1\}}}  \big) \\
   &+ (1+\delta_{T_{d-1}}^2)  \big((1+\mathbf{1}_{\delta_1\neq 0})\Vert I_{V_1} \Vert^2_{U^{min}_1(u)\to \Hc_1}+(1+\mathbf{1}_{\delta_2\neq 0})\Vert I_{V_2} \Vert^2_{U^{min}_2(u)\to \Hc_2}\big) \Big)  ,
 \end{align*}
 \begin{align*}
   D^2 = (1+\delta(d-2)) \Big(& \sum_{\ell=1}^{d-2} (1+\delta_{T_\ell}^2) \Lambda_{\{1,\hdots,d-\ell-1\}}^2 \Lambda_{\{d-\ell\}}^2 (2+2\delta_{d-\ell+1}^2)  \\
   &+ (1+\delta_{T_{d-1}}^2)(1+2\delta_1^2)(1+2\delta_2^2)  \Big).
 \end{align*}
 %where $\Lambda_{\{1,\hdots,d-\ell-1\}}=\Vert I_{U_{\{1,\hdots,d-\ell-1\}}}\Vert_{U^{min}_{\{1,\hdots,d-\ell-1\}}(u)\to \Hc_{\{1,\hdots,d-\ell-1\}}},$
 %  and 
% where $$
% \delta_{T_\ell} = \Vert I_{U_{\{1,\hdots,d-\ell\}}}\otimes I_{U_{\{d-\ell+1\}}} -P_{U_{\{1,\hdots,d-\ell\}}}\otimes P_{U_{\{d-\ell+1\}}} \Vert_{U^{min}_{\{1,\hdots,d-\ell+1\}}(u^{\ell+1}) \to \Hc_{\{1,\hdots,d-\ell+1\}}}.
 %$$
If $u \in V$, then 
$$
C^2 = (1+\delta(d-2)) \left(\sum_{\ell=1}^{d-2} 2(1+\delta_{T_\ell}^2) \Lambda_{\{1,\hdots,d-\ell-1\}}^2 \Lambda_{\{d-\ell\}}^2    + (1+\delta_{T_{d-1}}^2)   \right).
$$
 \end{example}

\subsection{Complexity}\label{sec:complexity}
Here we analyse the complexity of the algorithm in terms of the number of evaluations of the function.  
%Lemma \ref{lem:commute_Il} implies that $u^\ell = \Ic_{T_\ell} u^{\ell+1} = \Ic_{T_\ell} u$ for  $1\le \ell < L $, and in particular
%$$
%u^\star = \Ic_{T_1} u.
%$$
%Therefore, there is no need in practice to compute the iterates $u^\ell$, $2\le \ell \le  L$. 
Evaluations of the function $u$ are required (i) for the computation 
of the subspaces $\{U_\alpha\}_{\alpha\in A}$ through empirical principal component analysis of the ${V_\alpha}$-valued functions $u_\alpha(\cdot,X_{\alpha^c})$, with {$V_\alpha$ a given approximation space} if $\alpha \in \Lc(A)$ or ${V_\alpha} = \bigotimes_{\beta\in S(\alpha)} U_\beta$ if $\alpha\in A\setminus \Lc(A)$,   
 and (ii) for the computation  of the final interpolation ${\Ic_{V_D}} u$.
\medskip

We then obtain the following result about the number of evaluations of the function required by the algorithm

% (estimation of $\alpha$-principal components of $u_\alpha$). 

\begin{proposition}
The total number of evaluations of $u$ required by the algorithm for computing an approximation $u^\star $ in the tensor format $\Tc_{r}^A(V)$ 
is 
\begin{align*}
M(A,r,{m},n) = \sum_{\alpha \in \Lc(A)} m_\alpha n_\alpha + \sum_{\alpha \in A \setminus \Lc(A)} 
m_\alpha \prod_{\beta \in S(\alpha)\cap A} r_\beta \prod_{\beta\in S(\alpha)\setminus A} n_\beta &
\\
+ \prod_{\beta \in S(D)\cap A} r_\beta \prod_{\beta\in S(D)\setminus A} n_\beta.&
\end{align*}
where $n=(n_\alpha)_{\alpha \in \Lc(T)}$, with $n_\alpha=\dim(V_\alpha),$ and ${m}=(m_\alpha)_{\alpha \in A}$, with $m_\alpha$ the number of samples of the $Z_{\alpha}$-valued random variable $u_\alpha(\cdot,X_{\alpha^c})$ used for computing $U_\alpha$. 

\end{proposition}
\begin{proof}
For $\alpha \in A$, the function $u_\alpha $ is an interpolation of $u$ in $ Z_\alpha = V_\alpha$ if $\alpha \in \Lc(A)$, or  in $ Z_\alpha = \bigotimes_{\beta \in S(\alpha)} U_\beta = \left(\bigotimes_{\beta \in S(\alpha) \cap A} U_\beta\right) \otimes  \left(\bigotimes_{\beta \in S(\alpha) \setminus  A} V_\beta\right) $ if $\alpha \notin \Lc(A).$ Therefore, computing 
 $u_\alpha(\cdot,x_{\alpha^c}^k)$ for one realization $x^k_{\alpha^c}$ of $X_{\alpha^c}$ requires $\dim(V_\alpha) = n_\alpha$ evaluations of $u$ if $\alpha \in \Lc(A)$ or 
$\dim(\bigotimes_{\beta \in S(\alpha)} U_\beta) = \prod_{\beta \in S(\alpha)\cap A} r_\beta \prod_{\beta \in S(\alpha)\setminus A} n_\beta$ if $\alpha \notin \Lc(A)$. Finally, the computation of the interpolation $\Ic_{T_1} u = \Ic_{S(D)} u$ requires $\dim(\bigotimes_{\alpha \in S(D)} U_\alpha) = \prod_{\beta \in S(D)\cap A} r_\beta \prod_{\beta\in S(D)\setminus A} n_\beta$ evaluations of $u$.
\end{proof}
For computing a $r_\alpha$-dimensional subspace $U_\alpha$, the number of samples $m_\alpha$ of $u_\alpha(\cdot,X_{\alpha^c})$ has to be at least $r_\alpha$. 
\begin{corollary}\label{cor:eval_storage_fixedrank}
If the number of samples $m_\alpha =  r_\alpha$ for all $\alpha \in A$, then the number of  evaluations of the function required by the algorithm is 
$$
M(A,r,  r,n) =   \mathrm{storage}(\Tc_{r}^A(V)).
$$
\end{corollary}
The above result 
 states that for a prescribed rank $r = (r_\alpha)_{\alpha\in A}$, 
the algorithm is able to construct an approximation of $u$ using a number os samples equal to the storage complexity of the tensor format $\Tc_r^A(V)$. 
\\\par

When using the algorithm with a prescribed tolerance $\epsilon$, the rank $r_\alpha$  is not fixed a priori but defined as the minimal integer such that 
the condition \eqref{ass:precision_Ualpha} is satisfied. Since samples of $u_\alpha(\cdot,X_{\alpha^c})$ belongs to the subspace $U^{min}_\alpha(u_\alpha) \subset Z_\alpha$
with dimension $\rank_\alpha(u_\alpha) \le \dim(Z_\alpha)$, the selected rank $r_\alpha$ is at most 
$ \dim(Z_\alpha)$. Therefore, by taking $m_\alpha = \dim(Z_\alpha)$ for all $\alpha\in A$, if we assume that the set of $m_\alpha $ samples of $u(\cdot,X_{\alpha^c})$ contains $\rank_{\alpha}(u_\alpha)$ linearly independent functions in $Z_\alpha$, then 
the algorithm is able to produce an approximation with arbitrary small tolerance $\epsilon.$
\begin{corollary}\label{cor:eval_storage}
If the number of samples $m_\alpha = \dim(Z_\alpha)$ for all $\alpha \in A$, then 
\begin{align*}
M(A,r,{m},n) = \sum_{\alpha \in \Lc(A)} n_\alpha^2 + \sum_{\alpha \in A \setminus \Lc(A)} 
\prod_{\beta \in S(\alpha)\cap A} r_\beta^2 \prod_{\beta\in S(\alpha)\setminus A} n_\beta^2 &
\\+ \prod_{\beta \in S(D)\cap A} r_\beta \prod_{\beta\in S(D)\setminus A} n_\beta.&
\end{align*}
\end{corollary}
\begin{remark}
For numerical experiments, when working with prescribed tolerance, we will use 
$m_\alpha = \dim(Z_\alpha)$ for al $\alpha \in A$. 
\end{remark}

\section{Numerical examples}\label{sec:examples}

In all examples, we consider functions $u$ in the tensor space $L^2_\mu(\Xc)$, with $\Xc \subset \Rbb^d$, equipped with the natural norm $\Vert \cdot \Vert$ (see example \ref{ex:L2})\footnote{For the last example, $\Xc$ is a finite product set equipped with the uniform measure and $L^2_\mu(\Xc)$ then corresponds to the space of multidimensional arrays equipped with the canonical norm.}. 
For an approximation $u^\star$ provided by the algorithm, we estimate the relative error $\varepsilon(u^\star) = \Vert u - u^\star \Vert/\Vert u\Vert$ using Monte-Carlo integration. 
We denote by $M$ the total number of evaluations of the function $u$ required by the algorithm to provide an approximation $u^\star$, and by $S$ the storage complexity of the approximation $u^\star$. Since the algorithm uses random evaluations of the function $u$ (for the estimation of principal components), we run the algorithm several times and indicate confidence intervals of level $90\%$ for $\varepsilon(u^\star)$,  and also for $M$, $S$ and approximation ranks when these quantities are random.
\medskip

For the approximation with a prescribed $A$-rank, we use $m_\alpha = \gamma r_\alpha$ samples for the estimation of principal subspaces $U_\alpha$, $\alpha \in A$. If $\gamma=1$, then $M = S$ (see corollary \ref{cor:eval_storage_fixedrank}).  

For the approximation with a prescribed tolerance $\epsilon$, we use $m_\alpha = \dim(Z_\alpha) $ for all $\alpha \in A$ (see corollary \ref{cor:eval_storage} for the estimation of $M$). 
\medskip

In all examples except the last one, we use
polynomial approximation spaces $V_\nu = \Pbb_p(\Xc_\nu)$ over $\Xc_\nu\subset \Rbb$, $\nu \in D$, with the same polynomial degree $p$ in all dimensions. For each $\nu \in D$, we use an orthonormal polynomial basis of $V_\nu = \Pbb_p(\Xc_\nu)$ (Hermite polynomials for a Gaussian measure, Legendre polynomials for a uniform measure,...), and associated interpolation grids $\Gamma^\star_\nu$ selected in a set of 1000 random points (drawn from the measure $\mu_\nu$) by using the greedy algorithm described in Section \ref{sec:magicpoints}.

%Un In all examples except one, we consider  $\Xc  = [-1,1]^d $ equipped with the uniform measure $\mu$.
%, and consider approximations in the tensor train Tucker format  $\Tc_r^A(V)$ described in example \ref{ex:ttt}, or in the tensor train format $\Tc_r^A(V)$ described in example  \ref{ex:tt}, with a rank $r \le \rank_A(u)$.

%
%\subsection{Functions with finite ranks}
%
%Here, we consider functions $u$ with finite $\alpha$-ranks, i.e. $u$ belongs to the algebraic tensor space $\Hc_1\otimes_a \hdots \otimes_a \Hc_d$. 

\subsection{Henon-Heiles potential}

We consider $\Xc = \Rbb^d $ equipped with the standard Gaussian measure $\mu$ and the modified
Henon-Heiles potential \cite{kressner2014low}
$$
u(x_1,\hdots,x_d) = \frac{1}{2} \sum_{i=1}^d x_i^2 + \sigma_* \sum_{i=1}^{d-1} (x_ix_{i+1}^2-x_i^3) + \frac{\sigma_*^2}{16} \sum_{i=1}^{d-1} (x_i^2 + x_{i+1}^2)^2,
$$
with $\sigma_\star=0.2$. We consider approximation in the tensor train format  $\Tc_r^A(V)$ described in example \ref{ex:tt}. The function is such that $\rank_{\alpha}(u) = 3$ for all $\alpha\in A$.
We use a polynomial degree $p=4$, so that there is no discretization error, i.e. $u\in V$.

In Table \ref{tab:henonheiles-prescribed-rank}, 
we observe that the algorithm with a prescribed rank $r=(3,\hdots,3)$ is able to recover the function at very high precision with high probability with a number of samples equal to the storage complexity of the approximation (when $\gamma=1$), with no deterioration when the dimension $d$ increases from $5$ to $100$. The accuracy is slightly improved when   $\gamma=100$ but with a much higher number of evaluations of the function.

\begin{table}[h]
\caption{Henon-Heiles potential. Approximation with prescribed rank $r=(3,\hdots,3)$ and $\gamma =1$ and $\gamma=100$, for different values of $d$.}\label{tab:henonheiles-prescribed-rank}
\smallskip \centering$\begin{array}{|c|c|c|c|c|c|} \hline
  \multicolumn{6}{|c|}{\gamma=1}  \\ \hline
 d & 5 & 10 & 20 & 50 & 100  \\ \hline
\varepsilon(u^\star)\times 10^{14}  & [1.0;234.2] & [1.5;67.5] & [2.5;79.9] & [6.6;62.8] & [15.7;175.1]\\\hline
S=M  & 165  & 390  & 840  & 2190  & 4440 \\\hline\hline
  \multicolumn{6}{|c|}{\gamma=100}  \\ \hline
 d & 5 & 10 & 20 & 50 & 100  \\ \hline
\varepsilon(u^\star)\times 10^{14}   & [0.1;0.4] & [0.2;0.4] & [0.3;0.4] & [0.4;0.7] & [0.6;0.8]\\\hline
S  & 165  & 390  & 840  & 2190  & 4440 \\\hline
M  & 1515 & 3765 & 8265 & 21765 & 44265\\\hline
\end{array}
$
\end{table}
%
%In Table \ref{tab:henonheiles-prescribed-precision}, we verify that the algorithm with prescribed tolerance $\epsilon=10^{-14}$ provides an approximation with the desired precision for different values of $d$, with a number of samples between two or three times higher than the storage complexity. 
%
%\begin{table}[h]
%\caption{Henon-Heiles potential. Approximation with prescribed tolerance $\epsilon = 10^{-14}$ and $\gamma=1$ for different values of $d$.}\label{tab:henonheiles-prescribed-precision}
%\smallskip \centering$\begin{array}{|c|c|c|c|c|c|} \hline
%d & 5 & 10 & 20 & 50 & 100  \\ \hline
%\varepsilon(u^\star)\times 10^{15}   & [0.8;1.9] & [1.1;2.3] & [1.7;6.7] & [2.3;3.1] & [4.1;7.3]\\\hline
%S  & 165 & 390 & 840 & 2190 & 4440\\\hline
%M  & 715 & 1840 & 4090 & 10840 & 22090\\\hline
%\end{array}
%$
%\end{table}
%	
 
\subsection{Sine of a sum}

We consider $\Xc = [-1,1]^d $ equipped with the uniform measure
and the function
$$
u(x_1,\hdots,x_d) = \sin(x_1 + \hdots + x_d).
$$
We consider approximation in the tensor train Tucker format  $\Tc_r^A(V)$ described in example \ref{ex:ttt}. The function is such that $\rank_\alpha(u) = 2$ for all $\alpha \in A$. 
In Table \ref{tab:sinofasum-prescribed-rank}, 
we observe the behavior of the algorithm with a prescribed rank $r=(2,\hdots,2)$ for different polynomial degrees $p$ and different values of $d$. We observe a linear dependence of the complexity with respect to $d$.  

\begin{table}[h]
\caption{Sine of a sum. Approximation with prescribed rank $r=(2,\hdots,2)$ and $\gamma =1$. Relative error $\varepsilon(u^\star)$ and number of evaluations $M=S$ for different values of $d$ and $p$.}
\label{tab:sinofasum-prescribed-rank}
\smallskip \centering$\begin{array}{|c||c|c||c|c||c|c|} \hline
&   \multicolumn{2}{c||}{d=10}  & \multicolumn{2}{c||}{d=20} & \multicolumn{2}{c|}{d=50}   \\\hline
 & \varepsilon(u^\star) &  M & \varepsilon(u^\star) & M & \varepsilon(u^\star) & M \\\hline
p=3 & [3.2 ; 3.3] \times 10^{-1} & 148 & [ 5.2; 5.3] \times 10^{-1} & 308 & [8.8 ; 8.81] \times 10^{-1} & 788 \\\hline
p=5  & [1.29 ;1.31] \times 10^{-2}  & 188 & [ 2.3; 2.33] \times 10^{-2} & 388  &[5.2 ; 5.3] \times 10^{-2}  & 988 \\\hline
p=7  & [1.77 ; 1.81] \times 10^{-4} & 228 & [ 2.9 ; 3.0 ] \times 10^{-4} & 468  & [6.0 ; 6.1] \times 10^{-4}  & 1188 \\\hline
p=9  & [4.1 ; 4.2] \times 10^{-6}  & 268 & [ 6.4; 6.6] \times 10^{-6} & 548 & [1.27 ; 1.29] \times 10^{-5}  & 1388 \\\hline
p=11  & [2.17,2.2] \times 10^{-8} & 308 & [3.7 ;3.8 ] \times 10^{-8} & 628  & [8.2 ; 8.4] \times 10^{-8}  & 1588 \\\hline
p=13  & [7.6,7.7] \times 10^{-10} & 348 & [ 1.32; 1.24] \times 10^{-10} & 708  & [3.00 ; 3.04] \times 10^{-10}  & 1788 \\\hline
p=15  & [7.6,7.8] \times 10^{-12} & 388 & [1.0 ; 1.1] \times 10^{-12} & 788  & [1.7 ; 2.5] \times 10^{-12}  & 1988 \\\hline
p=17  & [4.1,13] \times 10^{-14} & 428 & [ 0.8 ; 4.9 ] \times 10^{-14} & 868  & [0.4 ; 6.7] \times 10^{-13}  & 2188 \\\hline
\end{array}
$
\end{table}

In Table \ref{tab:sinofasum-prescribed-precision}, we observe the behavior of the algorithm with prescribed tolerance $\epsilon = 10^{-12}$ and fixed polynomial degree $p=17$, for different values of $d$. For this value of $\epsilon$, the algorithm always provides an approximation with rank $(2,\hdots,2)$ with a fixed number of evaluations which is about ten times the storage complexity. 

\begin{table}[h]
\caption{Sine of a sum. Approximation with prescribed tolerance $\epsilon = 10^{-12}$, $p=17$ and $\gamma=1$ for different values of $d$.}
\label{tab:sinofasum-prescribed-precision}
\smallskip \centering$\begin{array}{|c|c|c||c|c|c||c|c|c|} \hline
   \multicolumn{3}{|c||}{d=10}  & \multicolumn{3}{c||}{d=20} & \multicolumn{3}{c|}{d=50}   \\\hline
  \varepsilon(u^\star) &  S & M & \varepsilon(u^\star) & S & M & \varepsilon(u^\star) & S & M \\\hline
  [3.7  ;  6.3]\times 10^{-13} & 428 & 3372 & [0.6   ; 1.3]\times 10^{-14} & 868 & 6772 &  [1.4;3.2]\times 10^{-14} & 2188 &16972 \\\hline
\end{array}
$
\end{table}

\subsection{Sum of bivariate functions}

We consider $\Xc = [-1,1]^d $ equipped with the uniform measure
and the function
\begin{align}
%u(x_1,\hdots,x_d) = \sum_{i=0}^{\lfloor (d-2)/k \rfloor} g(x_{1+ki} ,x_{2+ki}) \label{sumofbivariate}
%u(x_1,\hdots,x_d) = \sum_{(i,j) \in \Lambda}  g(x_{i} ,x_{j}) \label{sumofbivariate}
u(x_1,\hdots,x_d) = g(x_1,x_2) + g(x_3,x_4) + \hdots + g(x_{d-1},x_d) \label{sumofbivariate}
\end{align}
where $g$ is a bivariate function, and $d=10$.  We consider approximation in the tensor train Tucker format  $\Tc_r^A(V)$ described in example \ref{ex:ttt}. The function is such that $\rank_{\{\nu\}}(u) = \rank(g)+1$ for all $\nu \in D $, and 
$\rank_{\{1,\hdots,\nu\}}(u)=2$ if $\nu$ is even, or $\rank_{\{1,\hdots,\nu\}}(u)=\rank(g)+1$ if $\nu$ is odd. Here, we use the algorithm we a prescribed tolerance $\epsilon$.

We first consider the function  
$g(y,z) = \sum_{j=0}^3 y^j z^j$ whose rank is $4$ and 
we use polynomial spaces of degree $p=5$,  so that there is no discretization error. 
We observe in Table \ref{tab:manu-poly} the behavior of the algorithm for decreasing values of $\epsilon$. For $\epsilon= 10^{-4}$, the algorithm always provides the solution at almost machine precision, with an exact recovery of the rank of the function $u$. We observe that increasing $\gamma$ (i.e. the number of evaluations for the estimation of principal components)  allows us to obtain a more accurate approximation for a given prescribed tolerance but with  a significant increase in the number of evaluations. 

\begin{table}[h]\caption{Sum of bivariate functions \eqref{sumofbivariate} with $g(y,z) = \sum_{j=0}^3 y^j z^j$. Approximation with prescribed $\epsilon$, degree $p=5$, and different $\gamma$. Confidence intervals for relative error $\varepsilon(u^\star)$, storage complexity $S$ and number of evaluations $M$.}\label{tab:manu-poly}
\smallskip \centering 
$\begin{array}{|c|c|c|c|}
 \hline
& \multicolumn{3}{|c|}{\gamma=1}   \\\hline
\epsilon & \varepsilon(u^\star)   & M & S     \\\hline
%10^{-1} & 2.8\,10^{-1} &  521 &  192   & 2.0\,10^{-1} &  5484 & 212 &  1.9\,10^{-1} & 54804 & 212   \\\hline
%10^{-2} & 1.5\,10^{-1} & 1034&  373 & 1.1\,10^{-2} & 16412 & 500& 1.1\,10^{-2} & 164012 &  500 \\\hline
%10^{-3} & 2.6\,10^{-2} & 2088 &  560&  2.7\,10^{-15} & 20736 &  560& 9.5\,10^{-15}& 207216] &  560   \\\hline
%10^{-4} & 7.8\,10^{-15} & 2088 & 560 &  2.7\,10^{-15}& 20736 & 560& 1.1\,10^{-14} & 207216 &  560 \\\hline
10^{-1} & [1.4\,\times10^{-1};2.8\,\times10^{-1}] & [444,521] & [160, 192]    \\\hline
10^{-2} & [0.8\,\times 10^{-1};1.5\,\times 10^{-1}] & [918,1034] & [345, 373]  \\\hline
10^{-3} & [1.7\,10^{-15};2.6\, \times 10^{-2}] & [1916,2088] & [530, 560]   \\\hline
10^{-4} & [1.6\,\times 10^{-15} ;7.8\,\times 10^{-15} ]& 2088 & 560    \\\hline\hline
& \multicolumn{3}{|c|}{\gamma=10} \\\hline
\epsilon & \varepsilon(u^\star)   & M & S     \\\hline
 10^{-1} &[1.7\,10^{-1};2.0\,10^{-1}] & [5364,5484] & [202, 212]  \\\hline
10^{-2} & [0.9\,\times10^{-2} ;1.1\,\times10^{-2} ]& [16132,16412] & [486, 500] \\\hline
10^{-3}  &  [2.1\,\times 10^{-15} ;2.7\,\times 10^{-15} ]& 20736 & 560  \\\hline
10^{-4}  &  [1.7\,\times 10^{-15} ;2.7\,\times 10^{-15} ] & 20736 & 560  \\\hline
\end{array}
$
\end{table}

We now consider the function  
$g(y,z) = \exp^{-\frac{1}{8}(y-z)^2}$ 
with infinite rank. 
We observe in Tables \ref{tab:manu-exp} and \ref{tab:manu-exp-pepsilon} the behavior of the algorithm for decreasing values of $\epsilon$, and for a fixed polynomial degree $p=10$ in Table \ref{tab:manu-exp}, and an adaptive polynomial degree $p(\epsilon) = \log_{10}(\epsilon^{-1})$ in Table \ref{tab:manu-exp-pepsilon}. We observe that the relative error of the obtained approximation  is below the prescribed tolerance with high probability. Also, we clearly see the interest of adapting the discretization to the desired precision, which yields a lower complexity for small or moderate $\epsilon$. 

\begin{table}[h]
\caption{Sum of bivariate functions \eqref{sumofbivariate} with $g(y,z) = \exp^{-\frac{1}{8}(y-z)^2}$. Approximation with prescribed $\epsilon$, degree $p=10$, $\gamma=1$. Confidence intervals for relative error $\varepsilon(u^\star)$, storage complexity $S$ and number of evaluations $M$.}\label{tab:manu-exp}
\smallskip \centering $\begin{array}{|c|c|c|c|} \hline
\epsilon & \varepsilon(u^\star)   & M & S   \\
\hline
10^{-1} & [3.8\,10^{-2};5.3\,10^{-2}] & [1219,1222] & [119, 131]  \\\hline
10^{-2} & [1.8\,10^{-2};3.8\,10^{-2}] & [1282,1294] & [252, 256]  \\\hline
10^{-3} & [1.2\,10^{-4};2.0\,10^{-3}] & [1813,1876] & [507, 519]  \\\hline
10^{-4} & [1.2\,10^{-4};1.6\,10^{-4}] & [1876,1876] & [519, 519]  \\\hline
10^{-5} & [1.6\,10^{-5};6.9\,10^{-5}] & [3275,4063] & [821, 935]  \\\hline
10^{-6} & [1.8\,10^{-6};7.1\,10^{-6}] & [4135,4410] & [975, 995]  \\\hline
10^{-7} & [3.1\,10^{-8};2.5\,10^{-6}] & [4685,4960] & [1015, 1035]  \\\hline
10^{-8} & [2.7\,10^{-8};1.3\,10^{-7}] & [5048,6120] & [1056, 1164]  \\\hline
10^{-9} & [1.2\,10^{-8};4.8\,10^{-8}] & [9671,11595] & [1476, 1578]  \\\hline
10^{-10} & [1.9\,10^{-10};1.5\,10^{-8}] & [11647,13117] & [1603, 1659]   \\\hline
\end{array}
$
\end{table}
\begin{table}[h!]
\caption{Sum of bivariate functions \eqref{sumofbivariate} with $g(y,z) = \exp^{-\frac{1}{8}(y-z)^2}$. Approximation with prescribed $\epsilon$, degree $p(\epsilon)=\log_{10}(\epsilon^{-1})$, $\gamma=1$. Confidence intervals for relative error $\varepsilon(u^\star)$, storage complexity $S$ and number of evaluations $M$.}\label{tab:manu-exp-pepsilon}
\smallskip \centering $\begin{array}{|c|c|c|c|} \hline
\epsilon & \varepsilon(u^\star)   & M & S   \\
\hline
10^{-1} & [1.4\,10^{-1};3.3\,10^{-1}] & [52,70] & [32, 42]  \\\hline
10^{-2} & [2.9\,10^{-2};4.2\,10^{-2}] & [162,184] & [88, 100]  \\\hline
10^{-3} & [3.2\,10^{-3};1.1\,10^{-2}] & [598,778] & [258, 292]  \\\hline
10^{-4} & [1.7\,10^{-4};2.5\,10^{-4}] & [916,916] & [339, 339]  \\\hline
10^{-5} & [5.7\,10^{-5};1.5\,10^{-4}] & [2056,2759] & [562, 622]  \\\hline
10^{-6} & [1.1\,10^{-6};3.5\,10^{-5}] & [3190,3465] & [758, 778]  \\\hline
10^{-7} & [6.9\,10^{-8};2.1\,10^{-7}] & [4390,4390] & [885, 885]  \\\hline
10^{-8} & [3.2\,10^{-8};1.2\,10^{-7}] & [4560,5319] & [935, 998]  \\\hline
10^{-9} & [8.3\,10^{-9};4.1\,10^{-8}] & [9415,11385] & [1396, 1509]  \\\hline
10^{-10} & [1.6\,10^{-10};1.7\,10^{-8}] & [11647,12382] & [1603, 1631] \\\hline
\end{array}
$
\end{table}

\subsection{Borehole function}

We here consider the function 
\begin{align*}
f(Y_1,\hdots,Y_8)= \frac{2\pi Y_3(Y_4-Y_6)}{(Y_2-\log(Y_1)) (1+\frac{2 Y_7 Y_3}{(Y_2-\log(Y_1))Y_1^2  Y_8}+
    \frac{Y_3}{Y_5})}
\end{align*}
which models the water flow through a borehole as a function of $8$ independent random variables $Y_1 \sim \Nc(0.1,0.0161812) $, 
$Y_2 \sim \Nc(7.71, 1.0056)$, 
$Y_3 \sim \Uc(63070,115600)$,
$Y_4\sim \Uc(990,1110)$,
$Y_5\sim \Uc(63.1,116)$,
$Y_6\sim  \Uc(700,820)$,
$Y_7\sim \Uc(1120,1680)$,
$Y_8\sim \Uc(9855,12045)$.
We then consider the function 
$$
u(x_1,\hdots,x_d) = f(g_1(x_1),\hdots,g_8(x_8)),$$
where $g_\nu$ are functions such that $Y_\nu = g_\nu(X_\nu)$, with $X_\nu  \sim \Nc(0,1)$ for $\nu\in \{1,2\}$, and $X_\nu\sim \Uc(-1,1)$ for $\nu\in \{3,\hdots,8\}.$ Function $u$ is then defined on  $\Xc= \Rbb^2 \times [-1,1]^{6}.$ 
We use polynomial approximation spaces $V_\nu = \Pbb_p(\Xc_\nu)$, $\nu \in D$.  
%For each $\nu$, we use an orthonormal polynomial basis of $V_\nu$ (Hermite polynomials for $\nu\in\{1,2\}$, and Legendre polynomials for $\nu\in \{3,\hdots,8\}$), and associated interpolation grids $\Gamma^\star_\nu$ selected in a set of 1000 random points by the greedy algorithm described in Section \ref{sec:magicpoints}. 
 We consider approximation in the tensor train Tucker format  $\Tc_r^A(V)$ described in example \ref{ex:ttt}.

In Table \ref{tab:borehole-prescribed-rank-ttt}, we observe the behavior of the algorithm with prescribed ranks $(r,\hdots,r)$ and fixed degree $p=10$. We observe a very fast convergence of the approximation with the rank. Increasing $\gamma$ (i.e. the number of evaluations for the estimation of principal components) allows us to improve the accuracy for a given rank but it we look at the error as a function of the complexity $M$, $\gamma=1$ is much better than $\gamma=100$. 
\begin{table}[h]
\caption{Borehole function. Approximation in tensor train Tucker format with prescribed rank $(r,\hdots,r)$, fixed degree $p=10$. Relative error $\varepsilon(u^\star)$ and {storage complexity $S$} for different values of $r$ and $\gamma$.}\label{tab:borehole-prescribed-rank-ttt}
\smallskip \centering$\begin{array}{|c|c|c|c|} \hline
&  &  \gamma=1  & \gamma=100  \\\hline
r & {S}  &\varepsilon(u^\star) & \varepsilon(u^\star)   \\
\hline
1 & 88 & [2.4\,10^{-2};2.7\,10^{-2}] & [2.3\,10^{-2};2.4\,10^{-2}]  \\\hline
2 & 308 & [1.4\,10^{-3};1.4\,10^{-2}] & [4.1\,10^{-4};5.0\,10^{-4}]\\\hline  
3 & 660 & [1.8\,10^{-5};4.9\,10^{-5}] & [9.9\,10^{-6};2.3\,10^{-5}]  \\\hline
4 & 1144 & [2.9\,10^{-6};3.5\,10^{-6}] & [8.8\,10^{-7};1.9\,10^{-6}]  \\\hline
5 & 1760 & [5.2\,10^{-7};6.1\,10^{-7}] & [1.8\,10^{-7};7.4\,10^{-7}]  \\\hline
6 & 2508 & [9.0\,10^{-8};1.3\,10^{-7}] & [1.9\,10^{-8};5.2\,10^{-8}]  \\\hline
7 & 3388 & [5.7\,10^{-8};9.2\,10^{-8}] & [5.1\,10^{-9};1.1\,10^{-8}]  \\\hline
8 & 4400 & [1.6\,10^{-9};5.1\,10^{-9}] & [4.3\,10^{-10};2.0\,10^{-9}]  \\\hline
9 & 5544 & [1.5\,10^{-9};2.4\,10^{-9}] & [3.1\,10^{-10};8.6\,10^{-10}]  \\\hline
10 & 6820 & [5.5\,10^{-11};1.1\,10^{-10}] & [4.3\,10^{-11};7.6\,10^{-11}] \\\hline
\end{array}
$
\end{table}

%In Table \ref{}
%\begin{table}[h]\label{tab:borehole-prescribed-tolerance-ttt}
%\caption{Borehole function. Approximation in tensor train Tucker format with prescribed $\epsilon$, $p=10$, $\gamma=1$. Confidence intervals for relative error $\varepsilon(u^\star)$, storage complexity $S$ and number of evaluations $M$ for different $\epsilon$, and average ranks.}
%\smallskip \centering$\begin{array}{|c|c|c|c|c|} \hline
%\epsilon & \varepsilon(u^\star)  & M & S  &  [r_{\{1\}},\hdots,r_{\{d\}},r_{\{1,2\}},\hdots , r_{\{1,\hdots,d-1\}}]  \\\hline
%10^{-1} & [2.2;4.4] \times 10^{-2}& [975,975] & [95, 95]  &[1,1,1,1,1,1,1,1,1,1,1,1,1,1] \\\hline
%10^{-2} & [1.7;2.6] \times 10^{-3}& [996,996] & [126, 126]  & [1,1,1,2,1,1,1,1,1,2,1,2,1,1] \\\hline
%10^{-3}& [0.5;1.5] \times 10^{-3}& [1008,1058] & [150, 189]  & [2,2,1,2,2,1,1,2,1,2,2,2,1,1] \\\hline
%10^{-4} & [1.4;3.9] \times 10^{-5}& [1096,1096] & [231, 231]  &[2,2,2,3,3,2,2,2,1,2,2,2,2,2] \\\hline
%10^{-5} & [0.6;3.2] \times 10^{-5}& [1096,1152] & [231, 253]  &[2,2,2,4,4,2,2,2,1,2,2,2,2,2] \\\hline
%10^{-6} & [3.1;4.4]\times 10^{-6} & [1184,1240] & [283, 309]  & [3,2,3,4,4,2,2,2,2,2,2,2,2,2] \\\hline
%10^{-7} & [0.9;4.1]\times 10^{-7} & [1615,1723] & [452, 492]  & [3,3,4,6,5,3,3,3,2,2,3,2,3,2] \\\hline
%10^{-8} & [0.1;1.0]\times 10^{-7} & [1770,1887] & [530, 556]  & [4,3,4,7,6,3,3,3,2,2,3,2,3,3] \\\hline
%10^{-9} & [0.2;6.3]\times 10^{-8} & [1887,2050] & [556, 602]  & [4,4,5,8,6,3,3,3,2,2,3,2,3,3] \\\hline
%10^{-10} & [0.5;4.2] \times 10^{-9}& [2479,2825] & [759, 823] &[5,4,6,9,7,4,3,4,2,2,3,2,3,3] \\\hline
%\end{array}
%$
%\end{table}

In Table \ref{tab:borehole-prescribed-tolerance-ttt-adaptive-p}, 
we observe   the behavior of the algorithm for decreasing values of $\epsilon$, and for  an adaptive polynomial degree $p(\epsilon) = \log_{10}(\epsilon^{-1})$. We observe that for all $\epsilon,$ the relative error of the obtained approximation  is below $\epsilon$ with high probability. We note  that the required number of evaluations $M$ is about 2 to 4 times the storage complexity.

\begin{table}[h!]
\caption{Borehole function. Approximation in tensor train Tucker format with prescribed $\epsilon$, $p(\epsilon)=\log_{10}(\epsilon^{-1})$, $\gamma=1$. Confidence intervals for relative error $\varepsilon(u^\star)$, storage complexity $S$ and number of evaluations $M$ for different $\epsilon$, and average ranks.}\label{tab:borehole-prescribed-tolerance-ttt-adaptive-p}
\smallskip \centering$\begin{array}{|c|c|c|c|c|} \hline
\epsilon & \varepsilon(u^\star) & M & S   &  [r_{\{1\}},\hdots,r_{\{d\}},r_{\{1,2\}},\hdots , r_{\{1,\hdots,d-1\}}]  \\
\hline
10^{-1} & [1.8;2.7]\times 10^{-1} & [39,39] & [23, 23] &[1,1,1,1,1,1,1,1,1,1,1,1,1,1] \\\hline
10^{-2} & [0.3;4.0]\times 10^{-2} & [88,100] & [41, 46] & [1,1,1,1,1,1,1,1,1,2,1,2,1,1] \\\hline
10^{-3} & [0.8;1.9]\times 10^{-3} & [159,186] & [61, 78]  &[2,1,1,2,2,1,1,1,1,2,2,2,1,1]\\\hline
10^{-4} & [2.5;5.6]\times 10^{-5} & [328,328] & [141, 141] & [2,2,2,3,3,2,2,2,1,2,2,2,2,2] \\\hline
10^{-5} & [0.6;1.6] \times 10^{-5}& [444,472] & [166, 178] &[2,2,2,4,4,2,2,2,1,2,2,2,2,2] \\\hline
10^{-6} & [3.1;5.7] \times 10^{-6}& [596,664] & [204, 241] &[3,2,2,4,5,3,2,2,2,2,2,2,2,2] \\\hline
10^{-7} & [1.0;6.3] \times 10^{-7}& [1042,1267] & [374, 429] &[4,3,4,6,5,3,3,3,2,2,3,2,2,2] \\\hline
10^{-8} & [1.1;7.1] \times 10^{-8}& [1567,1567] & [512, 512] &[4,3,4,7,6,3,3,3,2,2,3,2,3,3] \\\hline
10^{-9} & [0.2;4.9] \times 10^{-8}& [1719,1854] & [534, 560]  &[4,4,4,8,6,3,3,3,2,2,3,2,3,3]\\\hline
10^{-10} & [0.3;1.9]\times 10^{-9} & [2482,2828] & [774, 838] &[5,4,6,10,7,4,3,3,2,2,3,2,3,3]\\\hline
\end{array}
$
\end{table}

 \subsection{{Tensorization of a univariate function}}
   
We consider the approximation of the univariate function $f : [0,1] \to \Rbb$ {using tensorization of functions} \cite{khoromskij2011dlog,oseledets2011algebraic}. We denote by $ f_N$ the piecewise constant approximation of $f$ on a uniform partition
$0 = t_0 \le t_1 \le \hdots \le t_{N} = 1$ with $N=2^d$ elements, such that $f_N(i h) = f(ih)$ for $0\le i \le N$ and $h=N^{-1}=2^{-d}.$ We denote by 
$v\in \mathbb{R}^N$ the vector with components $v(i) = f(ih)$, $0\le i \le N-1$. The vector $v \in \mathbb{R}^{2^d}$ can be identified with an order-$d$ tensor $u \in  \mathcal{H} = \mathbb{R}^2\otimes \hdots \otimes  \mathbb{R}^2$  such that 
$$
u(i_1,\hdots,i_d) = v(i), \quad i=\sum_{k=1}^{d} i_{k} 2^{d-k},
$$
where $(i_{1},\hdots,i_{d})\in \{0,1\}^d = \Xc$ is the binary representation of the integer $i\in \{0,\hdots,2^d-1\}$. The set $\Xc$ is equipped with the uniform measure $\mu$. Then we consider approximation of the tensor $u$ in tensor train format. The algorithm evaluates the tensor $u$ at some selected entries $(i_1,\hdots,i_d)$, which corresponds to evaluating the function $f$ at some particular points $t_i$. 

In this finite-dimensional setting, we consider $V = \mathcal{H}.$ 
%For each $\nu\in D$, we use an orthonormal polynomial basis with respect to the discrete uniform measure on $\{0,1\}$, composed by the two functions $ \varphi^\nu_1(i_\nu) = 1$ and 
%$\varphi^\nu_2(i_\nu) = 4(i_\nu-\frac{1}{2})$.   
In all examples, we consider $d=40,$ and $N=2^d \approx 10^{12}$. This corresponds to a storage complexity of one terabyte for the standard representation of $f_N$ as a vector $v$ of size $N$. 

We observe in Tables \ref{tab:square-prescribed-tolerance-tt} and \ref{tab:sqrt-prescribed-tolerance-tt} the behavior of the algorithm with prescribed tolerance $\epsilon$ applied to the functions $f(t)=t^2$ and $f(t)=t^{1/2}$ respectively. We indicate relative errors in $\ell^2$ and $\ell^\infty$ norms between the tensor $u$ and the approximation $u^\star$. Let us recall that for $f(t)=t^\alpha$, the approximation error $\Vert f - f_N \Vert_{L^\infty} = O(N^{-\beta}) = O(2^{-d\beta})$ with $\beta = \min\{1,\alpha\}$, which is an exponential convergence with respect to $d$.
For the function $f(t) = t^2$, we observe that the relative error in $\ell^2$ norm is below the prescribed tolerance with high probability. For the function $f(t) = t^{1/2}$, the probability of obtaining a relative error in $\ell^2$ norm below the prescribed tolerance decreases with $\epsilon$ but the ratio between the true relative error and the prescribed tolerance remains relatively small (below 100). 
We note that for $f(t)=t^2$, the approximation ranks are bounded by $3$, which is the effective rank of $f_N$. For $f(t)=t^{1/2}$, the approximation ranks slowly increase with $\epsilon^{-1}$.

In both cases, we observe a very good behavior of the algorithm, which requires a number of evaluations which scales as $\log(\epsilon^{-1})$.

\begin{table}[h!]
\caption{{Tensorization of} $f(t) = t^2$, $d=40$. Approximation  in  tensor train format with prescribed $\epsilon$, $\gamma=1$. Confidence intervals for relative $\ell^2$-error $\varepsilon(u^\star)$, relative $\ell^\infty$-error $\varepsilon_\infty(u^\star)$, number of evaluations $M$, storage complexity $S$ and maximal rank for different $\epsilon$.
}\label{tab:square-prescribed-tolerance-tt}
\smallskip \centering
$\begin{array}{|c|c|c|c|c|c|} \hline
\epsilon & \varepsilon(u^\star) & \varepsilon_\infty(u^\star) & M & S   &  \max_{\alpha} r_\alpha \\
\hline
10^{-1}  & [1.9\,10^{-2};1.2\,10^{-1}] & [2.2\,10^{-2};1.8\,10^{-1}] & [158,194] & [80, 96] & [1, 2]   \\\hline
10^{-2}  & [2.4\,10^{-3};7.7\,10^{-3}] & [3.1\,10^{-3};1.8\,10^{-2}] & [230,250] & [114, 122] & [2, 3]  \\\hline
10^{-3}  & [2.6\,10^{-4};3.1\,10^{-3}] & [3.1\,10^{-4};7.2\,10^{-3}] & [274,326] & [134, 160] & [3, 3]  \\\hline
10^{-4}  & [2.7\,10^{-5};1.2\,10^{-4}] & [4.2\,10^{-5};2.5\,10^{-4}] & [370,394] & [182, 194] & [3, 3]  \\\hline
10^{-5}  & [2.1\,10^{-6};8.9\,10^{-6}] & [2.9\,10^{-6};1.1\,10^{-5}] & [446,470] & [220, 232] & [3, 3]  \\\hline
10^{-6}  & [2.5\,10^{-7};7.8\,10^{-7}] & [3.1\,10^{-7};1.4\,10^{-6}] & [514,546] & [254, 270] & [3, 3]  \\\hline
10^{-7}  & [3.0\,10^{-8};2.4\,10^{-7}] & [4.0\,10^{-8};2.6\,10^{-7}] & [586,614] & [290, 304] & [3, 3]  \\\hline
10^{-8}  & [2.1\,10^{-9};4.8\,10^{-9}] & [3.4\,10^{-9};5.6\,10^{-9}] & [678,690] & [336, 342] & [3, 3]  \\\hline
10^{-9}  & [2.3\,10^{-10};4.8\,10^{-10}] & [2.8\,10^{-10};7.5\,10^{-10}] & [746,766] & [370, 380] & [3, 3]  \\\hline
10^{-10}  & [3.1\,10^{-11};7.5\,10^{-11}] & [3.9\,10^{-11};1.0\,10^{-10}] & [810,842] & [402, 418] & [3, 3]    \\\hline
\end{array}
$
\end{table}

\begin{table}[h!]
\caption{{Tensorization of} $f(t) = t^{1/2}$, $d=40$. Approximation in  tensor train format with prescribed $\epsilon$, $\gamma=1$. Confidence intervals for relative $\ell^2$-error $\varepsilon(u^\star)$, relative $\ell^\infty$-error $\varepsilon_\infty(u^\star)$, number of evaluations $M$, storage complexity $S$ and maximal rank for different $\epsilon$.
}\label{tab:sqrt-prescribed-tolerance-tt}
\smallskip \centering
$\begin{array}{|c|c|c|c|c|c|} \hline
\epsilon & \varepsilon(u^\star) & \varepsilon_\infty(u^\star) & M & S   &  \max_{\alpha} r_\alpha \\
\hline
10^{-1}  & [9.3\,10^{-3};5.5\,10^{-2}] & [4.1\,10^{-2};2.7\,10^{-1}] & [182,230] & [90, 114] & [2, 2]  \\\hline
10^{-2}  & [3.7\,10^{-3};8.6\,10^{-3}] & [2.6\,10^{-2};5.1\,10^{-2}] & [314,350] & [156, 172] & [2, 3]   \\\hline
10^{-3}  & [5.4\,10^{-4};9.2\,10^{-4}] & [3.0\,10^{-3};8.5\,10^{-3}] & [514,606] & [252, 300] & [3, 3]   \\\hline
10^{-4}  & [1.3\,10^{-4};3.3\,10^{-3}] & [7.9\,10^{-4};2.4\,10^{-2}] & [838,962] & [414, 474] & [4, 4]   \\\hline
10^{-5}  & [1.8\,10^{-5};8.2\,10^{-4}] & [1.6\,10^{-4};5.4\,10^{-3}] & [1270,1398] & [626, 692] & [4, 5]   \\\hline
10^{-6}  & [1.3\,10^{-6};6.3\,10^{-5}] & [1.2\,10^{-5};4.3\,10^{-4}] & [1900,2036] & [938, 1014] & [5, 5]   \\\hline
10^{-7}  & [4.9\,10^{-7};1.3\,10^{-6}] & [3.5\,10^{-6};1.5\,10^{-5}] & [2444,2718] & [1218, 1344] & [5, 6]   \\\hline
10^{-8}  & [1.0\,10^{-7};1.2\,10^{-6}] & [1.1\,10^{-6};1.5\,10^{-5}] & [3304,3468] & [1642, 1722] & [6, 6]   \\\hline
10^{-9}  & [2.2\,10^{-8};1.3\,10^{-7}] & [1.7\,10^{-7};1.2\,10^{-6}] & [4116,4328] & [2046, 2144] & [7, 7]   \\\hline
10^{-10}  & [8.6\,10^{-10};6.7\,10^{-8}] & [8.8\,10^{-9};4.0\,10^{-7}] & [5024,5136] & [2490, 2552] & [7, 7]   \\\hline

\end{array}
$
\end{table}

\bibliographystyle{plain}

\end{document}